\documentclass[11pt]{article}

\usepackage{geometry}
\usepackage{latexsym}
\usepackage{mathrsfs}
\usepackage{pifont}
\usepackage{makecell}
\usepackage{amsfonts}
\usepackage{amssymb}
\usepackage{microtype}
\usepackage{subfigure}    
\usepackage{graphicx}
\usepackage{epstopdf}
\usepackage{float}
\usepackage{multirow}
\usepackage{bm}
\usepackage{amsmath}
\usepackage{amsthm}
\usepackage{color}
\usepackage[subnum]{cases}
\usepackage{rotating}
\usepackage{enumitem}
\usepackage{accents}
\usepackage{algorithm}
\usepackage{algorithmic}
\usepackage[title]{appendix}
\usepackage{mathtools}
\usepackage{caption}
\usepackage{url}
\usepackage{booktabs}
\usepackage{tcolorbox}

\makeatletter
\@addtoreset{equation}{section}
\makeatother

\makeatletter
\def \wideubar{\underaccent{{\cc@style\underline{\mskip15mu}}}}
\def \widebar{\accentset{{\cc@style\underline{\mskip10mu}}}}
\makeatother

\newcommand{\email}[1]{\protect\href{mailto:#1}{#1}}

\definecolor{blue}{rgb}{0,0,0.9}
\definecolor{red}{rgb}{0.9,0,0}
\definecolor{green}{rgb}{0,0.9,0}
\definecolor{brown}{rgb}{0.6,0.1,0.1}
\definecolor{lightgreen}{rgb}{0.1,0.5,0.1}

\usepackage[colorlinks=true,
breaklinks=true,
bookmarks=true,
urlcolor=blue,
citecolor=lightgreen,
linkcolor=lightgreen,
bookmarksopen=false,
draft=false]{hyperref}

\graphicspath{{images/}}

\begin{document}
\newtheorem{property}{Property}[section]
\newtheorem{proposition}{Proposition}[section]
\newtheorem{append}{Appendix}[section]
\newtheorem{definition}{Definition}[section]
\newtheorem{lemma}{Lemma}[section]
\newtheorem{corollary}{Corollary}[section]
\newtheorem{theorem}{Theorem}[section]
\newtheorem{remark}{Remark}[section]
\newtheorem{problem}{Problem}[section]
\newtheorem{example}{Example}[section]
\newtheorem{assumption}{Assumption}
\renewcommand*{\theassumption}{\Alph{assumption}}

\title{A Nonmonotone Extrapolated Proximal Gradient-subgradient Algorithm beyond Global Lipschitz Gradient Continuity}

%\author{Anonymous Author(s)}

\author{
Lei Yang\thanks{School of Computer Science and Engineering, and Guangdong Province Key Laboratory of Computational Science, Sun Yat-sen University, Guangzhou, China (\email{yanglei39@mail.sysu.edu.cn}). 
}
\and
Jingjing Hu\thanks{School of Computer Science and Engineering, Sun Yat-sen University, Guangzhou, China (\email{hujj53@mail2.sysu.edu.cn}).
}
\and
Tianxiang Liu\thanks{(Corresponding author) Institute of Systems and Information Engineering, University of Tsukuba, Tsukuba, Japan (\email{liutx@sk.tsukuba.ac.jp}). 
}
}

\maketitle

\begin{abstract}
With the advancement of modern applications, an increasing number of composite optimization problems arise whose smooth component does \emph{not} possess a globally Lipschitz continuous gradient. This setting prevents the direct use of the proximal gradient (PG) method and its variants, and has motivated a growing body of research on new PG-type methods and their convergence theory, in particular, global convergence analysis \textit{without} imposing any explicit or implicit boundedness assumptions on the iterates. Until recently, the first complete analysis of this kind has been established for the PG method and its specific nonmonotone variants, which has since stimulated further exploration along this research direction. In this paper, we consider a general composite optimization model beyond the global Lipschitz gradient continuity setting. We propose a novel problem-parameter-free algorithm that incorporates a carefully designed nonmonotone line search to handle the non-global Lipschitz gradient continuity, together with an extrapolation step to achieve potential acceleration. Despite the added technical challenges introduced by combining extrapolation with nonmonotone line search, we establish a refined convergence analysis for the proposed algorithm under the Kurdyka-{\L}ojasiewicz property, \textit{without} requiring any boundedness assumptions on the iterates. This work thus further advances the theoretical understanding of PG-type methods in the non-global Lipschitz gradient continuity setting. Finally, we conduct numerical experiments to illustrate the effectiveness of our algorithm and highlight the advantages of integrating extrapolation with a nonmonotone line search.

\vspace{2mm}
\end{abstract}

% REQUIRED
%\begin{MSCcodes}
%68Q25, 68R10, 68U05
%\end{MSCcodes}

%%%%%%%%%%%%%%%%%%%%%%%%%%%%%%%%%%%%%%%%%%%
\section{Introduction}\label{sec-intro}

In this paper, we consider the following composite optimization problem:
\begin{equation}\label{ge-DC-problem}
\min\limits_{\bm{x}\in \mathbb{R}^n}\quad F(\bm{x}):=f(\bm{x})+P_1(\bm{x})-P_2(\bm{x}),
\end{equation}
where $f: \mathbb{R}^n \rightarrow \mathbb{R}$ is a continuously differentiable function, $P_1:\mathbb{R}^n\to(-\infty, \infty]$ is a proper closed function, and $P_2:\mathbb{R}^n\to\mathbb{R}$ is a convex function. More specific assumptions on model~\eqref{ge-DC-problem} are given in Assumption~\ref{assum-funs1}. 
This generalization of the assumptions emerges from modern application-driven problems, for example, see \cite{bbt2017descent,bstv2017first,lpt2017successive,wls2023linear}. In particular, we only require $\nabla f$ to have local Lipschitz continuity, which motivates the development of new algorithms suitable for this setting.

Most of the existing literature on first-order methods for solving problems with form \eqref{ge-DC-problem} relies on the assumption that $\nabla f$ is globally Lipschitz continuous  (also known as the $L$-smoothness of $f$). These methods are typically developed based on a majorization–minimization framework, whose fundamental iterative step is given by
\begin{equation}\label{PG-subpro}
\bm{x}^{k+1} \in \arg\min\limits_{\bm{x}\in{\mathbb{R}^n}}
\left\{\langle\nabla f(\bm{x}^k) - \bm{\xi}^k,\,\bm{x} - \bm{x}^k\rangle
+ \frac{\gamma}{2}\|\bm{x}-\bm{x}^k\|^2 + P_1(\bm{x})\right\},
\end{equation}
where $\bm{\xi}^k\in\partial P_2(\bm{x}^k)$ and $\gamma>0$ is a proximal parameter depending on the global Lipschitz constant of $\nabla f$. A well-known example is the classical proximal gradient (PG) method \cite{bt2009fast, cp2011proximal, lm1979splitting}. Accelerated variants of the PG method have also been extensively studied, including those incorporating extrapolation techniques such as the well-known fast iterative shrinkage-thresholding algorithm (FISTA) \cite{bt2009fast} and its nonconvex extensions \cite{ll2015accelerated,wcp2017linear,wcp2018proximal}, as well as those using nonmonotone line search strategies \cite{clp2016penalty,wnf2009sparse,yang2024proximal}.

The study of problems beyond global Lipschitz gradient continuity has started to attract increasing attention only recently. Related works in this direction mainly assume either the \textit{relative} $L$-smoothness of $f$ (e.g., \cite{bbt2017descent,bstv2017first,tft2022new,yht2025inexact,yt2026inexact}) or the \textit{local} Lipschitz continuity of $\nabla f$ (e.g., \cite{d2023proximal,jw2024advances,kl2025convergence,km2022convergence}). Although relative-$L$-smoothness-based approaches admit relatively mature theoretical guarantees, they typically require modifying the iterative step \eqref{PG-subpro} by replacing the simple quadratic proximal term with an appropriate Bregman proximal term, and also require prior knowledge of a global relative Lipschitz constant of $\nabla f$. In contrast, local-Lipschitz-gradient-based approaches aim to retain the standard iterative step \eqref{PG-subpro} and use a suitable line search mechanism to adapt to the local geometry of $\nabla f$.

However, the absence of global Lipschitz gradient continuity poses substantial challenges for the convergence analysis of the resulting algorithms, particularly when one aims to establish global sequential convergence and convergence rates.
To address this difficulty, most existing works further impose either a boundedness assumption on the iterates or a level-boundedness assumption on the objective function (which, in turn, ensures the boundedness of iterates). Under such boundedness-type assumptions, the local Lipschitz gradient continuity indeed amounts to a form of global Lipschitz gradient continuity, thereby simplifying the subsequent theoretical analysis. However, these boundedness-type assumptions can be restrictive and are often violated in practice, for example, in statistical regression problems with certain DC regularizers (e.g., SCAD \cite{fl2001variable}, MCP \cite{zhang2010nearly}) or in subproblems arising from the augmented Lagrangian method \cite{birgin2014practical}.

More recently, important progress has been made for the PG method and its nonmonotone variants under the local Lipschitz continuity of $\nabla f$, \textit{without} assuming the boundedness of iterates or the level-boundedness of the objective function; see, e.g., \cite{jkm2023convergence,jw2024advances,kl2025convergence}. These works have substantially deepened our understanding of PG-type methods beyond the standard $L$-smoothness setting. Nevertheless, they all focus on methods \textit{without} extrapolation. Since extrapolation is widely used in practice and often leads to superior numerical performance, it is therefore natural and important to ask whether one can incorporate an extrapolation step into such a framework while still establishing rigorous convergence guarantees under similarly weak assumptions. This question is also noted in the conclusions of \cite{jkm2023convergence}.

In this paper, we attempt to address the above question by proposing a \underline{n}onmonotone \underline{ex}trapolated \underline{p}roximal \underline{g}radient-subgradient \underline{a}lgorithm (nexPGA) for solving problem \eqref{ge-DC-problem}. Specifically, building on the basic iterative step \eqref{PG-subpro}, nexPGA employs a carefully designed Zhang--Hager (ZH)-type nonmonotone line search to accommodate both the extrapolation step and the non-$L$-smoothness setting. In contrast to existing ZH-type nonmonotone PG methods developed in \cite{d2023proximal,jw2024advances,kl2025convergence} beyond the standard $L$-smoothness setting but \textit{without} extrapolation, our line search is built on the following potential function:
\begin{equation}\label{definition-H-delta}
H_\delta(\bm{u}, \bm{v}, \gamma) := F(\bm{u}) + \frac{\delta \gamma}{8}\|\bm{u}-\bm{v}\|^2, \quad \forall\,\bm{u}, \,\bm{v} \in \mathbb{R}^n, ~\gamma>0,
\end{equation}
rather than on the objective function $F$ itself. This technique is essential for handling extrapolation, but it also introduces new analytical difficulties, since the convergence arguments developed for non-extrapolated methods can no longer be applied directly. These difficulties are further compounded by our goal of establishing the desired convergence properties \textit{without} assuming the global Lipschitz continuity of $\nabla f$ and \textit{without} imposing any explicit or implicit boundedness assumptions on the iterates. Consequently, a considerably more delicate and refined analysis is required. Another distinction from \cite{d2023proximal,jw2024advances,kl2025convergence} is that we additionally allow the presence of the convex term $P_2$ so that the proposed algorithmic framework can handle a broader class of nonsmooth structures. From an algorithmic perspective, this extension is meaningful because it allows more complicated nonsmooth terms to be handled through tractable subproblems involving only the proximal mapping of $P_1$; see, e.g., the subproblem arising in SDCAM \cite{lpt2017successive}.

The key contributions and findings of this paper are summarized as follows:
\begin{itemize}[leftmargin=0.7cm]
\item We develop nexPGA, a novel extrapolated algorithm with a ZH-type line search, for solving the general composite model \eqref{ge-DC-problem} beyond the standard $L$-smoothness setting; see Algorithm~\ref{algo-nexPGA}. Our nexPGA provides a unified and comprehensive \textit{problem-parameter-free} algorithmic framework that not only encompasses and complements numerous existing PG-type methods, but also introduces their potentially accelerated variants. 

\vspace{0.5mm}
\item We establish a complete and delicate global convergence analysis for nexPGA \textit{without} imposing any boundedness assumptions on the generated sequence. Specifically, we establish the global sequential convergence and local convergence rates of both the generated sequence and objective function values, see Theorems~\ref{nexPGA-theorem-wholesequence}, \ref{theorem-fun-rate} and \ref{theorem-seq-rate}. Our analysis and findings provide new insights into handling the extrapolation step under relaxed assumptions and a nonmonotone line search, and would contribute to the growing body of research on the ZH-type nonmonotone algorithms.

\vspace{0.5mm}
\item We conduct numerical experiments to evaluate the performance of nexPGA. Comparisons with several existing algorithms highlight the advantages of incorporating extrapolation and employing a ZH-type nonmonotone line search.
\end{itemize}

\vspace{1mm}
The rest of this paper is organized as follows. In Section \ref{sec-not-pre}, we present the notation and preliminaries used in this paper. We then describe nexPGA and establish the global subsequential convergence in Section \ref{sec-algo}, followed by a comprehensive study on global sequential convergence and convergence rates in Section \ref{sec-kl-analysis}. Some numerical results are presented in Section \ref{sec-num-exp}, with some concluding remarks given in Section \ref{sec-conclusion}.

%%%%%%%%%%%%%%%%%%%%%%%%%%%%%%%%%%%%%%%%%%%%%%%%%%%%%%%
\section{Notation and preliminaries}\label{sec-not-pre}

In this paper, we present scalars, vectors, and matrices in lowercase letters, bold lowercase letters, and uppercase letters, respectively. We use $\mathbb{N}$, $\mathbb{R}$, $\mathbb{R}^n$ ($\mathbb{R}^n_+$), and $\mathbb{R}^{m\times n}$ ($\mathbb{R}^{m\times n}_+$) to denote the sets of natural numbers, real numbers, $n$-dimensional real (nonnegative) vectors, and $m\times n$ real (nonnegative) matrices, respectively. For a vector $\bm{x}\in\mathbb{R}^n$, $x_i$ denotes its $i$-th entry, $\|\bm{x}\|$ denotes its Euclidean norm, and $\|\bm{x}\|_1:=\sum^n_{i=1}|x_i|$ denotes its $\ell_1$ norm. For a closed set $\mathcal{X}\subseteq\mathbb{R}^{n}$, we use $\mathrm{dist}(\bm{x}, \mathcal{X})$ to denote the distance from $\bm{x}$ to $\mathcal{X}$, i.e., $\mathrm{dist}(\bm{x}, \mathcal{X}) := \inf_{\bm{y}\in\mathcal{X}}\|\bm{x}-\bm{y}\|$.

For an extended-real-valued function $h: \mathbb{R}^n \rightarrow [-\infty,\infty]$, we say that it is \textit{proper} if $h(\bm{x}) > -\infty$ for all $\bm{x} \in \mathbb{R}^n$ and its domain ${\rm dom}\,h:=\{\bm{x} \in \mathbb{R}^n : h(\bm{x}) < \infty\}$ is nonempty. A proper function $h$ is said to be closed if it is lower semicontinuous. We use the notation $\bm{y} \xrightarrow{h} \bm{x}$ to denote $\bm{y} \rightarrow \bm{x}$ and $h(\bm{y}) \rightarrow h(\bm{x})$. The (limiting) subdifferential \cite[Definition~8.3]{rw1998variational} of $h$ at $\bm{x}\in \mathrm{dom}h$ is defined as
\begin{equation*}
\partial h(\bm{x}):=\left\{ \bm{d} \in \mathbb{R}^{n}: \exists\,\bm{x}^k \xrightarrow{h} \bm{x}~\mathrm{and}~\bm{d}^k \rightarrow \bm{d} ~\mathrm{with}~\bm{d}^k \in \widehat{\partial} h(\bm{x}^k) ~\mathrm{for~each}~k\right\},
\end{equation*}
where $\widehat{\partial} h(\widetilde{\bm{y}})$ denotes the Fr\'{e}chet subdifferential of $h$ at $\widetilde{\bm{y}}\in \mathrm{dom}h$, which is the set of all $\bm{d} \in \mathbb{R}^{n}$ satisfying $\liminf\limits_{\bm{y} \neq \widetilde{\bm{y}}, \,\bm{y} \rightarrow \widetilde{\bm{y}}} \frac{h(\bm{y})-h(\widetilde{\bm{y}})-\langle \bm{d},\,\bm{y}-\widetilde{\bm{y}}\rangle}{\|\bm{y}-\widetilde{\bm{y}}\|} \geq 0$. It can be observed from the above definition that
\begin{equation}\label{robust}
\left\{ \bm{d}\in\mathbb{R}^n: \exists \,\bm{x}^k \xrightarrow{h} \bm{x}, ~\bm{d}^k \rightarrow \bm{d} ~\mathrm{with}~\bm{d}^k \in \partial h(\bm{x}^k)~\text{for each } k \right\} \subseteq \partial h(\bm{x}).  
\end{equation}
When $h$ is continuously differentiable or convex, the above subdifferential coincides with the classical concept of gradient or convex subdifferential of $h$, respectively; see, e.g., \cite[Exercise~8.8]{rw1998variational} and \cite[Proposition~8.12]{rw1998variational}.

We next recall the Kurdyka-{\L}ojasiewicz (KL) property (see \cite{abs2013convergence,bdl2007the,bst2014proximal,lp2017calculus} for more details), which is now a standard technical condition for establishing the convergence of the whole sequence in the nonconvex setting. For simplicity, let $\Phi_{\nu}$ ($\nu>0$) denote a class of concave functions $\varphi:[0,\nu) \rightarrow \mathbb{R}_{+}$ satisfying: (i) $\varphi(0)=0$; (ii) $\varphi$ is continuously differentiable on $(0,\nu)$ and continuous at $0$; (iii) $\varphi'(t)>0$ for all $t\in(0,\nu)$. The KL property is described as follows.

\begin{definition}[\textbf{KL property and exponent}]\label{property-KL}
Let $h: \mathbb{R}^n \rightarrow \mathbb{R} \cup \{+\infty\}$ be a proper closed function. It is said to satisfy the \textbf{Kurdyka-{\L}ojasiewicz (KL)} property at $\tilde{\bm{x}}\in{\rm dom}\,\partial h$, if there exist a $\nu\in(0, +\infty]$, a neighborhood $\mathcal{V}$ of $\tilde{\bm{x}}$ and a function $\varphi \in \Phi_{\nu}$ such that for all $\bm{x} \in \mathcal{V} \cap \{\bm{x}\in \mathbb{R}^{n} : h(\tilde{\bm{x}})<h(\bm{x})<h(\tilde{\bm{x}})+\nu\}$, it holds that 
\begin{equation*}
\varphi'(h(\bm{x})-h(\tilde{\bm{x}}))\,\mathrm{dist}(\bm{0}, \,\partial h(\bm{x})) \geq 1.
\end{equation*}
The function $h$ is called a KL function, if it satisfies the KL property at each point of ${\rm dom}\,\partial h$. Furthermore, it is said to be a KL function with an exponent $\theta$ if $\varphi$ can be chosen as $\varphi(t)=\tilde{a}t^{1-\theta}$ for some $\tilde{a} > 0$ and $\theta\in[0, 1)$.
\end{definition}

We recall the uniformized KL property, which was established in \cite[Lemma 6]{bst2014proximal}.

\begin{proposition}[\textbf{Uniformized KL property}]\label{uniKL}
Suppose that $h: \mathbb{R}^n \rightarrow \mathbb{R} \cup \{+\infty\}$ is a proper closed function and $\Gamma$ is a compact set. If $h \equiv \zeta$ on $\Gamma$ for some constant $\zeta$ and satisfies the KL property at each point of $\Gamma$, then there exist $\varepsilon>0$, $\nu>0$ and $\varphi \in \Phi_{\nu}$ such that
\begin{equation*}
\varphi'(h(\bm{x}) - \zeta)\,\mathrm{dist}(\bm{0}, \,\partial h(\bm{x})) \geq 1
\end{equation*}
for all $\bm{x} \in \{\bm{x}\in\mathbb{R}^{n}: \mathrm{dist}(\bm{x},\,\Gamma)<\varepsilon\} \cap \{\bm{x}\in \mathbb{R}^{n} : \zeta < h(\bm{x}) < \zeta + \nu\}$.
\end{proposition}

%%%%%%%%%%%%%%%%%%%%%%%%%%%%%%%%%%%%%%%%%%%%%%%%%%%%%%%%%%%%%%%%%
\section{A nonmonotone extrapolated proximal gradient-subgradient algorithm}\label{sec-algo}

In this section, we develop a \underline{n}onmonotone \underline{ex}trapolated \underline{p}roximal \underline{g}radient-subgradient \underline{a}lgorithm (nexPGA) for solving problem \eqref{ge-DC-problem}, and study its preliminary convergence properties. The complete framework is outlined in Algorithm \ref{algo-nexPGA}, where the input parameters are chosen according to Assumption \ref{assum-para}.

\begin{assumption}\label{assum-para}
The input parameters satisfy $0<\gamma_{\min}\leq\gamma_{\max}<\infty$, $\beta_{\max}\geq0$, $0<p_{\min}<1$, $0\leq\delta<1$, $\tau>1$, and $0<\eta<\frac{1}{\sqrt{\tau}}$.
\end{assumption}

\begin{algorithm}[ht]
\caption{A nonmonotone extrapolated proximal gradient-subgradient algorithm (nexPGA) for solving problem \eqref{ge-DC-problem}}\label{algo-nexPGA}
\textbf{Input:} Follow Assumption \ref{assum-para} to choose $\gamma_{\min}$, $\gamma_{\max}$, $\beta_{\max}$, $p_{\min}$, $\delta$, $\tau$, $\eta$. Set $\bm{x}^{-1}=\bm{x}^0$, $\overline{\gamma}_{-1}=\gamma_{\min}$, $\mathcal{R}_0=F(\bm{x}^0)$, and $k=0$. \\[3pt]
\textbf{while} a termination criterion is not met, \textbf{do} \vspace{-1mm}
\begin{itemize}[leftmargin=1.8cm]
\item[\textbf{Step 1}.] Take any $\bm{\xi}^k\in \partial P_2(\bm{x}^k)$, and arbitrarily choose $\beta_{k,0}\in[0,\, \delta \beta_{\max}]$ and $\gamma_{k,0}\in[\gamma_{\min},\,\gamma_{\max}]$. Set $i=0$. 
    \begin{itemize}
    \item[\textbf{(1a)}] Compute
    \begin{equation}\label{extrapolation-yki}
    \bm{y}^{k,i}=\bm{x}^k+\beta_{k,i}(\bm{x}^k-\bm{x}^{k-1}).
    \end{equation}

    \item[\textbf{(1b)}] Solve the subproblem
    \begin{equation}\label{nexPGA-subpro}
    \hspace{-5mm}
    \bm{x}^{k,i} \in \arg\min\limits_{\bm{x}\in\mathbb{R}^n}\left\{\langle\nabla f(\bm{y}^{k,i})-\bm{\xi}^k, \,\bm{x}-\bm{y}^{k,i}\rangle + \frac{\gamma_{k,i}}{2}\|\bm{x}-\bm{y}^{k,i}\|^2+P_1(\bm{x})\right\}.
    \end{equation}

    \item[\textbf{(1c)}] If 
    \begin{equation}\label{nexPGA-lscond}
    H_\delta(\bm{x}^{k,i}, \bm{x}^k, \gamma_{k,i})-\mathcal{R}_k \leq-\frac{(1-\delta)\gamma_{k,i}}{8}\|\bm{x}^{k,i}-\bm{x}^k\|^2
    \end{equation}
    is satisfied, then go to \textbf{Step 2}.

    \vspace{1mm}
    \item[\textbf{(1d)}] Set $i\leftarrow i+1$, $\beta_{k,i}\leftarrow\eta\beta_{k,i-1}$, $\gamma_{k,i}\leftarrow\tau \gamma_{k,i-1}$, and go to \textbf{Step (1a)}.
    \end{itemize}

\item[\textbf{Step 2}.] Set $i_k\leftarrow i$, $\overline{\beta}_k\leftarrow\beta_{k,i}$, $\overline{\gamma}_k \leftarrow \gamma_{k,i}$, $\overline{\bm{y}}^k \leftarrow\bm{x}^k+\overline{\beta}_k(\bm{x}^k-\bm{x}^{k-1})$, and $\bm{x}^{k+1}\leftarrow\bm{x}^{k,i}$. Choose $p_{k+1}\in\left[p_{\min }, 1\right]$ to update  
    \begin{equation*}
    \mathcal{R}_{k+1}\leftarrow(1-p_{k+1}) \mathcal{R}_k + p_{k+1}H_\delta(\bm{x}^{k+1},\bm{x}^k,\overline{\gamma}_k).  
    \end{equation*}
    Then, set $k \leftarrow k+1$ and go to \textbf{Step 1}.
\end{itemize}
\textbf{end while}  \\
\textbf{Output}: $\bm{x}^k$ \vspace{0.5mm}
\end{algorithm}

The iterative framework of nexPGA is motivated by the proximal gradient method with extrapolation and line search (PGels) proposed by Yang \cite{yang2024proximal}, in which encouraging acceleration in practice was achieved. Nevertheless, the proposed nexPGA differs significantly from PGels. Specifically, nexPGA is \textit{problem-parameter-free} and does not require prior knowledge of a global Lipschitz constant of $\nabla f$, while PGels relies on this constant to guarantee the well-definedness of its nonmonotone line search. More importantly, nexPGA adopts an entirely different strategy for defining the reference value $\mathcal{R}_k$ used in the line search criterion. Indeed, PGels follows the spirit of the nonmonotone line search proposed by Grippo, Lampariello and Lucidi \cite{gll1986nonmonotone} by setting
\begin{equation*}
\mathcal{R}_k:=\max\left\{H_{\delta}(\bm{x}^{t}, \bm{x}^{t-1}, \overline{\gamma}_{t-1})\,:\, t=k,\,k-1,\cdots,[k-N]_{+}\right\}
\end{equation*}
for some fixed $N\in\mathbb{N}$. We refer to this as the GLL-type strategy, which often enables the use of larger $\beta_{k,i}$ and smaller $\gamma_{k,i}$, resulting in better practical performance.

In contrast, our nexPGA adopts a nonmonotone line search strategy inspired by Zhang and Hager \cite{zh2004nonmonotone}, where the reference value $\mathcal{R}_k$ is set as a convex combination of the previous reference value $\mathcal{R}_{k-1}$ and the latest potential function value $H_\delta(\bm{x}^{k}, \bm{x}^{k-1}, \overline{\gamma}_{k-1})$. We refer to this as the ZH-type strategy. Both GLL-type and ZH-type nonmonotone strategies have been widely adopted in the literature to enhance the numerical performance of proximal-gradient-type methods; see, e.g., \cite{clp2016penalty,gzlhy2013general,ll2015accelerated,tsp2018forward,wnf2009sparse, yang2024proximal}. Since the GLL-type line search typically requires stronger conditions to guarantee the desired convergence properties (see, e.g., \cite{qp2023convergence,qtpq2025gll}), we instead employ the ZH-type line search in nexPGA. As will be shown later, this choice allows us to establish strong convergence guarantees under weaker assumptions while maintaining comparable, or even superior, numerical performance as shown in Section \ref{sec-num-exp}.

The nexPGA in Algorithm \ref{algo-nexPGA} also provides a flexible algorithmic framework that not only encompasses and complements many existing methods but also facilitates the development of new potentially accelerated variants that incorporate both the ZH-type nonmonotone line search strategy and the extrapolation step. Below, we highlight several representative examples. When $P_2\equiv0$, nexPGA recovers the nonmonotone proximal gradient method studied in \cite{kl2025convergence} by setting $\delta=0$, and reduces to the monotone proximal gradient method studied in \cite{km2022convergence} by additionally setting $p_k\equiv1$. Beyond these recoveries, nexPGA naturally introduces new variants of these methods by integrating an extrapolation step, which potentially yield better practical performance as shown in Section \ref{sec-num-exp}. For example, in the case where $P_1$ is convex and $f$ is convex with a global Lipschitz continuous gradient, nexPGA, by incorporating the ZH-type line search strategy, gives an enhanced variant of the proximal difference-of-convex algorithm with extrapolation (pDCAe) proposed in \cite{wcp2018proximal}.

Finally, we would like to emphasize that nexPGA is developed \textit{without} assuming the global Lipschitz continuity of $\nabla f$. This line of research has recently attracted increasing attention in the study of (non)monotone PG methods \textit{without} extrapolation; see \cite{d2023proximal,jkm2023convergence,jw2024advances,kl2025convergence,km2022convergence}. Our nexPGA complements these approaches by providing a unified and problem-parameter-free algorithmic framework that incorporates an extrapolation step. Moreover, similar to the aforementioned works, the convergence analysis of nexPGA will be established \textit{without} requiring any explicit or implicit boundedness assumptions on the generated sequence $\{\bm{x}^k\}$. Specifically, under milder assumptions together with the KL property and exponent, we establish the global convergence of the entire sequence as well as convergence rates for both the objective values and the iterates. These results extend and strengthen the relevant theoretical results established in recent works \cite{d2023proximal,jkm2023convergence,jw2024advances,kl2025convergence,km2022convergence}. 

To clearly highlight the distinctions between nexPGA and existing nonmonotone PG methods developed under similar settings, we provide a detailed comparison in Table \ref{Table-Comp}. To the best of our knowledge, the proposed nexPGA is the first \textit{extrapolated} PG-type algorithm developed under these weaker assumptions.

\begin{table}[ht]
\caption{A comparison of our work with recent studies on the nonmonotone proximal gradient method under the local Lipschitz continuity assumption on the gradient, without requiring the boundedness of iterates. In the table, ``nls-type'' denotes the type of the nonmonotone line search used, ``assum on $P_2$'' denotes the assumption imposed on $P_2$, ``extra'' denotes whether the extrapolation step is allowed. Moreover, ``full seq conv'', ``seq rate'', and ``obj rate'' denote whether the convergence of the whole sequence, the convergence rate of the sequence, and the convergence rate of the objective function value sequence are established, respectively. In particular, the symbol ``\ding{51}$^P$'' indicates that the corresponding property is only partially established, either without considering the KL exponent $\theta\in(\frac{1}{2},1)$ or without providing proofs.}\label{Table-Comp}
\centering \tabcolsep 3.5pt
\scalebox{0.96}{
\begin{tabular}{ccccccc}
\hline
Reference & nls-type & \makecell[c]{assum \\ on $P_2$} & \makecell[c]{extra}
& \makecell[c]{full \\ seq conv} & \makecell[c]{seq rate} & \makecell[c]{obj rate}  \\
\hline
(Kanzow\,\&\,Mehlitz, 2022) \cite{km2022convergence} & GLL & 0 & \ding{55} & \ding{55} & \ding{55} & \ding{55} \\
(De Marchi, 2023) \cite{d2023proximal} & ZH & 0 & \ding{55} & \ding{55} & \ding{55} & \ding{55} \\
(Kanzow\,\&\,Lehmann, 2025) \cite{kl2025convergence} & ZH & 0 & \ding{55} & \ding{51} & \ding{51}$^P$ & \ding{51}$^P$ \\
(Jia\,\&\,Wang, 2024) \cite{jw2024advances} & GLL/ZH & 0 & \ding{55} & \ding{51} &\ding{51}$^P$ & \ding{51}$^P$ \\
This work & ZH & cvx & \ding{51} & \ding{51} & \ding{51} & \ding{51} \\
\hline
\end{tabular}}
\end{table}

We start the convergence analysis by making some blanket technical assumptions.

\begin{assumption}\label{assum-funs1}
Problem \eqref{ge-DC-problem} satisfies the following assumptions.
\begin{itemize}
\item[\bf{ B1.}] $f: \mathbb{R}^n \rightarrow \mathbb{R}$ is continuously differentiable; $\nabla f$ is locally Lipschitz continuous.
\item[\bf{ B2.}] $P_1: \mathbb{R}^n \rightarrow (-\infty,\infty]$ is proper, closed and prox-bounded, i.e., there exists some $\gamma > 0$ such that $\Delta\coloneqq\inf\left\{P_1 + \frac{\gamma}{2}\|\cdot\|^2\right\} > -\infty$; its proximal mapping is easy to compute.
\item[\bf{ B3.}] $P_2:\mathbb{R}^n\to\mathbb{R}$ is a convex function.
\item[\bf{ B4.}] $F$ is bounded from below on $\mathrm{dom}\,P_1$.
\end{itemize}
\end{assumption}

Assumption \ref{assum-funs1}2, which is also adopted in \cite{d2023proximal}, ensures that the subproblem \eqref{nexPGA-subpro} is well defined for all sufficiently large $\gamma_{k,i}$ and admits an easily computable minimizer $\bm{x}^{k,i}$, although such a minimizer need not be unique. Notably, Assumption \ref{assum-funs1}2 is weaker than those imposed in several existing studies on proximal-gradient-type algorithms, where $P_1$ is assumed to be either convex (see, e.g., \cite{gtt2018DC,tft2022new,wcp2018proximal}) or bounded below by an affine function (see, e.g., \cite{jkm2023convergence,kl2025convergence,km2022convergence}). Later, we will show that, under Assumptions \ref{assum-funs1}1 and \ref{assum-funs1}2, the line search criterion \eqref{nexPGA-lscond}, and hence the proposed algorithm, is also well defined. In addition, under Assumption \ref{assum-funs1}, one can show that any local minimizer $\widehat{\bm{x}}$ of problem \eqref{ge-DC-problem} satisfies
\begin{equation*}
0 \in \partial F(\widehat{\bm{x}})
= \nabla f(\widehat{\bm{x}}) + \partial\left(P_1-P_2\right)(\widehat{\bm{x}})
\subseteq \nabla f(\widehat{\bm{x}}) + \partial P_1(\widehat{\bm{x}})-\partial P_2(\widehat{\bm{x}}),
\end{equation*}
where the first inclusion follows from the generalized Fermat's rule \cite[Theorem 10.1]{rw1998variational}, the equality follows from \cite[Exercise 8.8(c)]{rw1998variational}, and the last inclusion follows from \cite[Corollary~3.4]{mny2006frechet}. We then define a stationary point of problem \eqref{ge-DC-problem} as follows.
\begin{definition}[\textbf{Stationary point}]\label{def-stationary}
We say that $\bm{x}^*$ is a stationary point of problem \eqref{ge-DC-problem} if $\bm{x}^*\in\mathrm{dom}\,F$ and it satisfies
\begin{equation*}
0 \in \nabla f\left(\bm{x}^*\right)
+ \partial P_1\left(\bm{x}^*\right)
- \partial P_2\left(\bm{x}^*\right).
\end{equation*}
In this paper, we denote the set of all stationary points of problem \eqref{ge-DC-problem} by $\mathcal{S}$.
\end{definition}

We next establish an asymptotic sufficient descent property, which plays a key role in ensuring the well-definedness of the ZH-type line search criterion \eqref{nexPGA-lscond}.

\begin{lemma}[\textbf{Asymptotic sufficient descent property}]\label{lem-descent-H-o}
Suppose that Assumptions \ref{assum-para} and \ref{assum-funs1} hold. For each $k \geq 0$, if $\bm{x}^{k, i}-\bm{y}^{k,i}\to0$ and $\bm{x}^k-\bm{y}^{k,i}\to0$ as $i\to\infty$, and $\beta_{k,i} \leq \sqrt{\frac{\delta\overline{\gamma}_{k-1}}{8\gamma_{k,i}}}$ holds for all sufficiently large $i$, then
\begin{equation*}
H_\delta(\bm{x}^{k,i}, \bm{x}^k, \gamma_{k,i})-H_\delta(\bm{x}^k, \bm{x}^{k-1}, \overline{\gamma}_{k-1})
\leq -\textstyle{\frac{(1-\delta)\gamma_{k,i}}{8}}\|\bm{x}^{k,i}-\bm{x}^k\|^2
\end{equation*}
holds for all sufficiently large $i$.
\end{lemma}
\begin{proof}
First, since $\bm{x}^{k,i}$ is a solution of the subproblem \eqref{nexPGA-subpro} with $\gamma_{k,i}$ and $\bm{y}^{k,i}$, we have that 
\begin{equation*}
\begin{aligned}
&\quad\langle\nabla f(\bm{y}^{k,i})-\bm{\xi}^k, \,\bm{x}^{k, i}-\bm{y}^{k,i}\rangle
+ \textstyle{\frac{\gamma_{k,i}}{2}}\|\bm{x}^{k, i}-\bm{y}^{k,i}\|^2+P_1(\bm{x}^{k, i})\\
&\leq \langle\nabla f(\bm{y}^{k,i})-\bm{\xi}^k, \,\bm{x}^k-\bm{y}^{k,i}\rangle
+ \textstyle{\frac{\gamma_{k,i}}{2}}\|\bm{x}^k-\bm{y}^{k,i}\|^2+P_1(\bm{x}^k),
\end{aligned}
\end{equation*}
which implies that
\begin{equation*}
P_1(\bm{x}^{k,i}) - P_1(\bm{x}^k)
\leq - \langle\nabla f(\bm{y}^{k,i})-\bm{\xi}^k, \,\bm{x}^{k,i}-\bm{x}^k\rangle
- \textstyle{\frac{\gamma_{k,i}}{2}}\|\bm{x}^{k,i}-\bm{y}^{k,i}\|^2
+ \frac{\gamma_{k,i}}{2}\|\bm{x}^k-\bm{y}^{k,i}\|^2.
\end{equation*}
Then, from the local Lipschitz continuity of $\nabla f$ and $-\nabla f$ and the given assumptions that $\bm{x}^{k, i}-\bm{y}^{k,i}\to0$ and $\bm{x}^k-\bm{y}^{k,i}\to0$ as $i\to\infty$, one can verify that
\begin{equation*}
f(\bm{x}^{k,i})-f(\bm{y}^{k,i}) \leq \langle\nabla f(\bm{y}^{k,i}), \,\bm{x}^{k, i}-\bm{y}^{k,i}\rangle
+ \mathcal{O}(\|\bm{x}^{k, i}-\bm{y}^{k,i}\|^2)
\end{equation*}
and
\begin{equation*}
(-f(\bm{x}^k))-(-f(\bm{y}^{k,i})) \leq \langle-\nabla f(\bm{y}^{k,i}), \,\bm{x}^k-\bm{y}^{k,i}\rangle
+ \mathcal{O}(\|\bm{x}^k-\bm{y}^{k,i}\|^2)
\end{equation*}
for all sufficiently large $i$. Moreover, by the convexity of $P_2$ and $\bm{\xi}^k\in\partial P_2(\bm{x}^k)$, we get
\begin{equation*}
P_2(\bm{x}^{k,i}) \geq P_2(\bm{x}^k) + \langle\bm{\xi}^k, \,\bm{x}^{k,i}-\bm{x}^{k}\rangle
~~ \Longleftrightarrow ~~
-P_2(\bm{x}^{k,i}) - (-P_2(\bm{x}^k)) \leq - \langle\bm{\xi}^k, \,\bm{x}^{k,i}-\bm{x}^{k}\rangle.
\end{equation*}
Summing the above four relations, we see that  
\begin{equation*}
\begin{aligned}
&\quad F(\bm{x}^{k, i})-F(\bm{x}^k)\\
&= f(\bm{x}^{k,i}) + P_1(\bm{x}^{k,i}) - P_2(\bm{x}^{k,i})
- \big[f(\bm{x}^k) + P_1(\bm{x}^k) - P_2(\bm{x}^k)\big]   \\
&\leq -{\textstyle\frac{\gamma_{k,i}}{2}}\|\bm{x}^{k,i}-\bm{y}^{k,i}\|^2
+ {\textstyle\frac{\gamma_{k,i}}{2}}\|\bm{x}^k-\bm{y}^{k,i}\|^2
+ \mathcal{O}(\|\bm{x}^{k,i}-\bm{y}^{k,i}\|^2) + \mathcal{O}(\|\bm{x}^k-\bm{y}^{k,i}\|^2)  \\
&= -{\textstyle\frac{\gamma_{k,i}}{4}}\|\bm{x}^{k,i}-\bm{x}^k
+\bm{x}^k-\bm{y}^{k,i}\|^2
+ {\textstyle\frac{\gamma_{k,i}}{4}}\|\bm{x}^k-\bm{y}^{k,i}\|^2
+ \mathcal{O}(\|\bm{x}^{k,i}-\bm{y}^{k,i}\|^2)    \\
&\qquad
+ \mathcal{O}(\|\bm{x}^k-\bm{y}^{k,i}\|^2)
- {\textstyle\frac{\gamma_{k,i}}{4}}\|\bm{x}^{k,i}-\bm{y}^{k,i}\|^2
+ {\textstyle\frac{\gamma_{k,i}}{4}}\|\bm{x}^k-\bm{y}^{k,i}\|^2   \\
&= -{\textstyle\frac{\gamma_{k,i}}{4}}\|\bm{x}^{k,i}-\bm{x}^k\|^2
- {\textstyle\frac{\gamma_{k,i}}{2}}\langle\bm{x}^{k,i}-\bm{x}^k,
\,\bm{x}^k-\bm{y}^{k,i}\rangle
+ \mathcal{O}(\|\bm{x}^{k,i}-\bm{y}^{k,i}\|^2)   \\
&\qquad
+ \mathcal{O}(\|\bm{x}^k-\bm{y}^{k,i}\|^2)
- {\textstyle\frac{\gamma_{k,i}}{4}}\|\bm{x}^{k,i}-\bm{y}^{k,i}\|^2
+ {\textstyle\frac{\gamma_{k,i}}{4}}\|\bm{x}^k-\bm{y}^{k,i}\|^2   \\
&\leq -{\textstyle\frac{\gamma_{k,i}}{4}}\|\bm{x}^{k,i}-\bm{x}^k\|^2
+ {\textstyle\frac{\gamma_{k,i}}{2}}
\|\bm{x}^{k,i}-\bm{x}^k\|\|\bm{x}^k-\bm{y}^{k,i}\|
+ \mathcal{O}(\|\bm{x}^{k,i}-\bm{y}^{k,i}\|^2)   \\
&\qquad
+ \mathcal{O}(\|\bm{x}^k-\bm{y}^{k,i}\|^2)
- {\textstyle\frac{\gamma_{k,i}}{4}}\|\bm{x}^{k,i}-\bm{y}^{k,i}\|^2
+ {\textstyle\frac{\gamma_{k,i}}{4}}\|\bm{x}^k-\bm{y}^{k,i}\|^2  \\
&\leq -{\textstyle\frac{\gamma_{k,i}}{8}}\|\bm{x}^{k,i}-\bm{x}^k\|^2
+ {\textstyle\frac{3\gamma_{k,i}}{4}}\|\bm{x}^k-\bm{y}^{k,i}\|^2  \\
&\qquad
+ \mathcal{O}(\|\bm{x}^{k,i}-\bm{y}^{k,i}\|^2)
+ \mathcal{O}(\|\bm{x}^k-\bm{y}^{k,i}\|^2)
- {\textstyle\frac{\gamma_{k,i}}{4}}\|\bm{x}^{k,i}-\bm{y}^{k,i}\|^2   \\
&= -{\textstyle\frac{\gamma_{k,i}}{8}}\|\bm{x}^{k,i}-\bm{x}^k\|^2
+ \gamma_{k,i}\|\bm{x}^k-\bm{y}^{k,i}\|^2   \\
&\qquad
+ \mathcal{O}(\|\bm{x}^{k,i}-\bm{y}^{k,i}\|^2)
+ \mathcal{O}(\|\bm{x}^k-\bm{y}^{k,i}\|^2)
-{\textstyle\frac{\gamma_{k,i}}{4}}\|\bm{x}^{k,i}-\bm{y}^{k,i}\|^2
- {\textstyle\frac{\gamma_{k,i}}{4}}\|\bm{x}^k-\bm{y}^{k,i}\|^2   \\
&\leq -{\textstyle\frac{\gamma_{k,i}}{8}}\|\bm{x}^{k,i}-\bm{x}^k\|^2
+ \gamma_{k,i}\beta_{k,i}^2\|\bm{x}^k-\bm{x}^{k-1}\|^2 \\
&\leq -{\textstyle\frac{(1-\delta)\gamma_{k,i}}{8}}\|\bm{x}^{k,i}-\bm{x}^k\|^2
- {\textstyle\frac{\delta\gamma_{k,i}}{8}}\|\bm{x}^{k,i}-\bm{x}^k\|^2
+ {\textstyle\frac{\delta\overline{\gamma}_{k-1}}{8}}\|\bm{x}^k-\bm{x}^{k-1}\|^2
\end{aligned}
\end{equation*}
holds for all sufficiently large $i$, where the third inequality follows from the relation $ab\leq\frac{a^2+b^2}{2}$ with $a:=\frac{\sqrt{{\gamma}_{k,i}}}{2}\|\bm{x}^{k,i}-\bm{x}^k\|$ and $b:=\sqrt{{\gamma}_{k,i}}\|\bm{x}^k-\bm{y}^{k,i}\|$, the second last inequality follows from \eqref{extrapolation-yki} and the fact that $\mathcal{O}(\|\bm{x}^{k,i}-\bm{y}^{k,i}\|^2)
+ \mathcal{O}(\|\bm{x}^k-\bm{y}^{k,i}\|^2)-{\textstyle\frac{\gamma_{k,i}}{4}}\|\bm{x}^{k,i}-\bm{y}^{k,i}\|^2- {\textstyle\frac{\gamma_{k,i}}{4}}\|\bm{x}^k-\bm{y}^{k,i}\|^2\leq0$ holds for all sufficiently large $i$ since $\bm{x}^{k, i}-\bm{y}^{k,i}\to0$, $\bm{x}^k-\bm{y}^{k,i}\to0$, $\gamma_{k,i}\to\infty$ as $i\to\infty$, the last inequality follows from the condition that $\beta_{k,i} \leq \sqrt{\frac{\delta\overline{\gamma}_{k-1}}{8\gamma_{k,i}}}$ holds for all sufficiently large $i$. Rearranging terms in the above relation and using the definition of 
$H_\delta$ in \eqref{definition-H-delta} yield the desired result.
\end{proof}

We highlight that the key difference between the asymptotic sufficient descent property established in Lemma \ref{lem-descent-H-o} and the common sufficient descent property (explicitly or implicitly used in, e.g., \cite[Eq.~(4.3)]{wcp2018proximal} and \cite[Lemma 3.1]{yang2024proximal}) lies in the fact that the former holds only for all sufficiently large $i$, when $\bm{x}^{k,i} - \bm{y}^{k,i} \to 0$ and $\bm{x}^k - \bm{y}^{k,i} \to 0$ as $i \to \infty$. This asymptotic nature arises because Lemma \ref{lem-descent-H-o} is established under the weaker assumption that $\nabla f$ is merely locally Lipschitz continuous, whereas the descent results in \cite{wcp2018proximal,yang2024proximal} require the global Lipschitz continuity of $\nabla f$. This relaxation makes Lemma \ref{lem-descent-H-o} more broadly applicable and suitable for a wider class of practical problems. Building upon this lemma, we proceed to establish the well-definedness of the ZH-type line search criterion \eqref{nexPGA-lscond}.

\begin{lemma}[\textbf{Well-definedness of the line search criterion \eqref{nexPGA-lscond}}]\label{lem-well-definedness}
Suppose that Assumptions \ref{assum-para} and \ref{assum-funs1} hold.
Then, for each $k\geq0$, the line search criterion \eqref{nexPGA-lscond} is satisfied after finitely many inner iterations.
\end{lemma}
\begin{proof}
We prove this lemma by contradiction and will divide the proof into four steps. Assume that there exists some $k\geq0$ such that the line search criterion \eqref{nexPGA-lscond} cannot be satisfied after finitely many inner iterations.

\textit{Step 1}. For this $k$, we first claim that
\begin{equation}\label{Rk-Hdelta}
\mathcal{R}_{k} \geq H_\delta(\bm{x}^k, \bm{x}^{k-1}, \overline{\gamma}_{k-1}).
\end{equation}
Indeed, when $k=0$, it follows from the initial settings of $\bm{x}^{-1}=\bm{x}^0$ and $\overline{\gamma}_{-1}=\gamma_{\min}$ that $\mathcal{R}_0=F(\bm{x}^0)=H_\delta(\bm{x}^0,\bm{x}^{0},\gamma_{\min})$ and hence \eqref{Rk-Hdelta} holds. When $k\geq1$, we see that 
\begin{equation*}
\hspace{-2mm}
\begin{aligned}
\mathcal{R}_{k}
&=(1\!-\!p_{k})\mathcal{R}_{k-1}+p_kH_\delta(\bm{x}^k, \bm{x}^{k-1}, \overline{\gamma}_{k-1})  \\
&\geq (1\!-\!p_{k})\!\left(H_\delta(\bm{x}^k, \bm{x}^{k-1}, \overline{\gamma}_{k-1})
+ {\textstyle\frac{(1-\delta)\overline{\gamma}_{k-1}}{8}}\|\bm{x}^k-\bm{x}^{k-1}\|^2\right)
+ p_kH_\delta(\bm{x}^k, \bm{x}^{k-1}, \overline{\gamma}_{k-1})  \\
&\geq H_\delta(\bm{x}^k, \bm{x}^{k-1}, \overline{\gamma}_{k-1}),
\end{aligned}
\end{equation*}
where the first inequality follows from the fact that the line search criterion \eqref{nexPGA-lscond} is satisfied at the $(k\!-\!1)$-th iteration.

\textit{Step 2}. We next claim that
\begin{equation}\label{xki-yki-xk-limit}
\lim\limits_{i\rightarrow\infty}\|\bm{x}^{k,i}-\bm{y}^{k,i}\|=0
\quad\mathrm{and}\quad
\lim\limits_{i\rightarrow\infty}\|\bm{x}^k-\bm{y}^{k,i}\|=0.
\end{equation}
It follows from \eqref{extrapolation-yki} and $\beta_{k,i}\downarrow0$ as $i\rightarrow\infty$ that $\lim\limits_{i\rightarrow\infty}\|\bm{y}^{k,i}-\bm{x}^k\| = 0$. Moreover, since $\bm{x}^{k,i}$ is an optimal solution of the subproblem \eqref{nexPGA-subpro} with $\gamma_{k,i}$ and $\bm{y}^{k,i}$, we see that
\begin{equation*}
\begin{aligned}
&\quad \langle\nabla f(\bm{y}^{k,i})-\bm{\xi}^k, \,\bm{x}^{k, i}-\bm{y}^{k,i}\rangle
+ \textstyle{\frac{\gamma_{k,i}}{2}\|\bm{x}^{k,i}-\bm{y}^{k,i}\|^2+P_1(\bm{x}^{k, i})}\\
&\leq \langle\nabla f(\bm{y}^{k,i})-\bm{\xi}^k, \,\bm{x}^k-\bm{y}^{k,i}\rangle
+ \textstyle{\frac{\gamma_{k,i}}{2}\|\bm{x}^k-\bm{y}^{k,i}\|^2+P_1(\bm{x}^k)} \\
&=-\beta_{k,i}\langle\nabla f(\bm{y}^{k,i})-\bm{\xi}^k, \,\bm{x}^k-\bm{x}^{k-1}\rangle
+ \textstyle{\frac{\gamma_{k,0}\beta_{k,0}^2[\tau\eta^2]^i}{2}} \|\bm{x}^k-\bm{x}^{k-1}\|^2+P_1(\bm{x}^k),
\end{aligned}
\end{equation*}
where the equality follows from the updating rules of $\bm{y}^{k,i}$, $\gamma_{k,i}$ and $\beta_{k,i}$. This, together with Assumption \ref{assum-funs1}2 ($P_1$ is prox-bounded with $\Delta\coloneqq\inf\left\{P_1 + \frac{\gamma}{2}\|\cdot\|^2\right\} > -\infty$ for some $\gamma>0$), implies that 
\begin{equation} \label{opt-con}
\begin{aligned}
&\quad \langle\nabla f(\bm{y}^{k,i})-\bm{\xi}^k - \gamma\bm{y}^{k,i}, \,\bm{x}^{k, i}-\bm{y}^{k,i}\rangle
+ \textstyle{\frac{\gamma_{k,i} - \gamma}{2}}\|\bm{x}^{k,i}-\bm{y}^{k,i}\|^2
- \frac{\gamma}{2}\|\bm{y}^{k,i}\|^2 + \Delta \\
&=\langle\nabla f(\bm{y}^{k,i})-\bm{\xi}^k , \,\bm{x}^{k, i}-\bm{y}^{k,i}\rangle
+ \textstyle{\frac{\gamma_{k,i}}{2}}\|\bm{x}^{k,i}-\bm{y}^{k,i}\|^2
- \frac{\gamma}{2}\|\bm{x}^{k,i}\|^2 + \Delta \\
&\leq \langle\nabla f(\bm{y}^{k,i})-\bm{\xi}^k, \,\bm{x}^{k, i}-\bm{y}^{k,i}\rangle
+ \textstyle{\frac{\gamma_{k,i}}{2}\|\bm{x}^{k,i}-\bm{y}^{k,i}\|^2+P_1(\bm{x}^{k, i})}\\
&\leq -\beta_{k,i}\langle\nabla f(\bm{y}^{k,i})-\bm{\xi}^k, \,\bm{x}^k-\bm{x}^{k-1}\rangle
+ \textstyle{\frac{\gamma_{k,0}\beta_{k,0}^2[\tau\eta^2]^i}{2}} \|\bm{x}^k-\bm{x}^{k-1}\|^2+P_1(\bm{x}^k).
\end{aligned}
\end{equation}
Using \eqref{opt-con}, we will prove $\lim\limits_{i\rightarrow\infty}\bm{x}^{k,i}-\bm{y}^{k,i}=\bm{0}$ by contradiction. Assume that it does not hold. Then, there must exist a subsequence $\{(\bm{x}^{k,i_j},\bm{y}^{k,i_j})\}$ such that for some $\varepsilon > 0$, we have $\|\bm{x}^{k,i_j}-\bm{y}^{k,i_j}\| \ge \varepsilon$ for all $j$. Since $\bm{y}^{k,i}\to\bm{x}^k$, $\nabla f$ is continuous (by Assumption \ref{assum-funs1}1), $\tau\eta^2<1$ (by Assumption \ref{assum-para}), and $\gamma_{k,i}\uparrow\infty$, $\beta_{k,i}\downarrow0$ as $i\to\infty$, we see that, along the subsequence $\{(\bm{x}^{k,i_j},\bm{y}^{k,i_j})\}$, the left-hand side of \eqref{opt-con} would go to infinity, while the right-hand side of \eqref{opt-con} converges to $P_1(\bm{x}^k)$. This leads to a contradiction. Thus, we have $\lim\limits_{i\rightarrow\infty}\|\bm{x}^{k,i}-\bm{y}^{k,i}\|=0$ and prove \eqref{xki-yki-xk-limit}.

\textit{Step 3}. Now, we show that 
\begin{equation}\label{beta_ki}
\beta_{k,i} \leq \textstyle{\sqrt{\frac{\delta\overline{\gamma}_{k-1}}{8\gamma_{k,i}}}}
\end{equation}
holds for all sufficiently large $i$. Indeed, it follows from $\tau\eta^2<1$ (by Assumption \ref{assum-para}) and the updating rules of $\beta_{k,i}$ and $\gamma_{k,i}$ that 
\begin{equation*}
\gamma_{k,i}\beta_{k,i}^2
= \gamma_{k,0}\beta_{k,0}^2(\tau\eta^2)^i
\leq \gamma_{\max}\beta_{\max}^2\delta^2(\tau\eta^2)^i
\leq \textstyle{\frac{\delta\gamma_{\min}}{8}}
\leq \textstyle{\frac{\delta\overline{\gamma}_{k-1}}{8}}
\end{equation*}
holds for all sufficiently large $i$. This then yields the desired result.

\textit{Step 4}. Finally, using \eqref{Rk-Hdelta}, \eqref{xki-yki-xk-limit} and \eqref{beta_ki},  we have from Lemma \ref{lem-descent-H-o} that 
\begin{equation*}
\begin{aligned}
H_\delta(\bm{x}^{k,i},\bm{x}^k,\gamma_{k,i}) - \mathcal{R}_k
\leq H_\delta(\bm{x}^{k,i},\bm{x}^k,\gamma_{k,i})
- H_\delta(\bm{x}^k,\bm{x}^{k-1},\overline{\gamma}_{k-1})
\leq -\textstyle{\frac{(1-\delta)\gamma_{k,i}}{8}}\|\bm{x}^{k, i}-\bm{x}^k\|^2
\end{aligned}
\end{equation*}
holds for all sufficiently large  $i$. This implies that, for this $k$, criterion \eqref{nexPGA-lscond} must hold for all sufficiently large $i$, leading to a contradiction. Thus, we complete the proof.
\end{proof}

Based on Lemma \ref{lem-well-definedness}, we further establish the following
properties.

\begin{proposition}\label{pro-basic-property}
Suppose that Assumptions \ref{assum-para} and \ref{assum-funs1} hold. Let $\{\bm{x}^k\}_{k=-1}^{\infty}$ and $\{\overline{\gamma}_k\}_{k=-1}^{\infty}$ be the sequences generated by the nexPGA in Algorithm \ref{algo-nexPGA}. Then, the following statements hold. \vspace{1mm}
\begin{itemize}
\item[{\rm(i)}] $\mathcal{R}_{k} \geq H_{\delta}(\bm{x}^k, \bm{x}^{k-1}, \overline{\gamma}_{k-1})$ for all $k\geq0$; \vspace{0.5mm}
\item[{\rm(ii)}] The sequence $\{\mathcal{R}_k\}$ is non-increasing and $\zeta:=\lim\limits_{k\to\infty}\mathcal{R}_k$ exists;  
\item[{\rm(iii)}] The sequence $\left\{H_{\delta}(\bm{x}^k, \bm{x}^{k-1}, \overline{\gamma}_{k-1})\right\}$ converges to the same $\zeta$;  \vspace{1mm}
\item[{\rm(iv)}] $\lim_{k\rightarrow\infty}\|\bm{x}^{k+1}-\bm{x}^{k}\|=0$;  \vspace{1mm}
\item[{\rm(v)}] $\{\bm{x}^k\}\subseteq\left\{\bm{x}\in\mathbb{R}^n \mid F(\bm{x}) \leq F(\bm{x}^0)\right\}$.
\end{itemize}
\end{proposition}
\begin{proof}
\textit{Statement (i)}. It follows directly from \eqref{Rk-Hdelta} and the subsequent discussion.

\textit{Statement (ii)}. By Lemma \ref{lem-well-definedness}, the sequence $\{\bm{x}^k\}$ is well-defined. From the line search criterion \eqref{nexPGA-lscond}, we have
\begin{equation*}
H_{\delta}(\bm{x}^{k+1}, \bm{x}^k, \overline{\gamma}_{k}) - \mathcal{R}_k
\leq -\textstyle{\frac{(1-\delta)\overline{\gamma}_{k}}{8}}
\|\bm{x}^{k+1}-\bm{x}^k\|^2, \quad \forall\,k\geq0.
\end{equation*}
Using the update rule of $\mathcal{R}_k$, $\overline{\gamma}_{k}\geq\gamma_{k,0}\geq\gamma_{\min}$ and $p_{k+1}\in[p_{\min},1]$, we obtain 
\begin{equation}\label{Rk-descent}
\begin{aligned}
\mathcal{R}_{k+1}
&= (1-p_{k+1})\mathcal{R}_k
+ p_{k+1}H_\delta(\bm{x}^{k+1},\bm{x}^k,\overline{\gamma}_k)  \\
&\leq (1-p_{k+1})\mathcal{R}_k
+ p_{k+1}\left(\mathcal{R}_k-\textstyle{\frac{(1-\delta)\overline{\gamma}_{k}}{8}}
\|\bm{x}^{k+1}-\bm{x}^k\|^2\right)   \\
&= \mathcal{R}_k - \textstyle{\frac{(1-\delta)p_{k+1}\overline{\gamma}_{k}}{8}}
\|\bm{x}^{k+1}-\bm{x}^k\|^2 \leq \mathcal{R}_k
- \frac{(1-\delta)p_{\min}\gamma_{\min}}{8}\|\bm{x}^{k+1}-\bm{x}^k\|^2
\end{aligned}
\end{equation}
for all $k\geq0$. Hence, $\{\mathcal{R}_k\}$ is non-increasing. By Assumption \ref{assum-funs1}4, $\{\bm{x}^k\}\subset \operatorname{dom} P_1$, and the definition of $H_\delta$ in \eqref{definition-H-delta}, the sequence $\{H_{\delta}(\bm{x}^k, \bm{x}^{k-1}, \overline{\gamma}_{k-1})\}$ is bounded below. Thus, $\{\mathcal{R}_k\}$ is bounded below by statement (i), and hence convergent. In particular, $\zeta:=\lim_{k\to\infty}\mathcal{R}_k$ exists.

\textit{Statement (iii)}. By the update rule of $\mathcal{R}_k$ and $p_{k+1}\geq p_{\min}>0$, we have
\begin{equation*}
H_{\delta}(\bm{x}^{k+1}, \bm{x}^k, \overline{\gamma}_k)
= \mathcal{R}_k + \textstyle{\frac{1}{p_{k+1}}}(\mathcal{R}_{k+1}-\mathcal{R}_k),
\quad \forall\,k\geq0.
\end{equation*}
Taking $k\to\infty$, and using statement (ii) and $p_{k+1}\in[p_{\min},1]$ yields the desired result.

\textit{Statement (iv)}. From \eqref{Rk-descent}, it follows that
\begin{equation*}
\|\bm{x}^{k+1}-\bm{x}^k\|^2
\leq \textstyle{\frac{8}{(1-\delta)p_{\min}\gamma_{\min}}}
\left(\mathcal{R}_k-\mathcal{R}_{k+1}\right).
\end{equation*}
Together with statement (ii), this implies that $\|\bm{x}^{k+1}-\bm{x}^{k}\|\to0$.

\textit{Statement (v)}. By the definition of $H_{\delta}$ in \eqref{definition-H-delta}, statements (i)--(ii), and $\mathcal{R}_0=F(\bm{x}^0)$, we obtain 
\begin{equation*}
F(\bm{x}^k)\leq H_\delta(\bm{x}^k, \bm{x}^{k-1}, \overline{\gamma}_{k-1})
\leq \mathcal{R}_k\leq \mathcal{R}_0 = F(\bm{x}^0),
\quad \forall\,k\geq0.
\end{equation*}
Hence, $\{\bm{x}^k\}\subseteq\left\{\bm{x} \in \mathbb{R}^n \mid F(\bm{x}) \leq F(\bm{x}^0)\right\}$.
\end{proof}

We next show that any proximal parameter $\overline{\gamma}_k$ associated with a point $\bm{x}^{k+1}$ that lies locally around an accumulation point (if it exists) remains bounded. This result serves as the extrapolated counterpart of \cite[Corollary 3.1]{km2022convergence} and \cite[Lemma 4.1]{jkm2023convergence}.

\begin{proposition}\label{pro-gammak-bounded}
Suppose that Assumptions \ref{assum-para} and \ref{assum-funs1} hold. Let $\{\bm{x}^k\}$ be the sequence generated by the nexPGA in Algorithm \ref{algo-nexPGA}, which is well-defined due to Lemma \ref{lem-well-definedness}. Moreover, if the sequence $\{\bm{x}^k\}$ has an accumulation point $\bm{x}^*$. Then, for any $\rho>0$, there exists a positive constant $\overline{\gamma}_{\rho,\bm{x}^*}>0$ (depending on $\rho$ and $\bm{x}^*$) such that $\overline{\gamma}_k\leq\overline{\gamma}_{\rho,\bm{x}^*}$ for all $k\in\mathcal{J}_{\rho}(\bm{x}^*)
:=\{j\in\mathbb{N}:\bm{x}^{j+1}\in\mathcal{B}_{\rho}(\bm{x}^*)\}$ with $\mathcal{B}_{\rho}(\bm{x}^*):=\left\{\bm{x}\in\mathbb{R}^n:\|\bm{x}-\bm{x}^*\|\leq\rho\right\}$.
\end{proposition}
\begin{proof}
First, it follows from $\tau\eta^2<1$ (by Assumption \ref{assum-para}) and the updating rules of $\beta_{k,i}$ and $\gamma_{k,i}$ that there exists a positive integer $N$ such that for each $k\in\mathbb{N}$, the following inequalities hold:
\begin{equation}\label{betageqN}
\gamma_{k,i}\beta_{k,i}^2
= \gamma_{k,0}\beta_{k,0}^2(\tau\eta^2)^i
\leq \gamma_{\max}\beta_{\max}^2\delta^2(\tau\eta^2)^i
\leq \textstyle{\frac{\delta\gamma_{\min}}{64}}
\leq \textstyle{\frac{\delta\overline{\gamma}_{k-1}}{64}},
\quad \forall\,i \geq N.
\end{equation}
Moreover, for any $\rho>0$, since $\bm{x}^*$ is an accumulation point of $\{\bm{x}^k\}$, there exist infinitely many iterates in $\mathcal{B}_{\rho}(\bm{x}^*)$. With these preparations, we will prove the desired result by contradiction and divide the proof into four steps.

\vspace{1mm}
\textit{Step 1}. Assume that for some $\rho > 0$, there exists a subsequence $\{k_j\}_{j\in\mathbb{N}}\subseteq\mathcal{J}_{\rho}(\bm{x}^*)$ such that $\{\overline{\gamma}_{k_j}\}_{j\in\mathbb{N}}$ is unbounded, i.e., $\overline{\gamma}_{k_j}\to\infty$ as $j\to\infty$. By passing to a further subsequence if necessary, we may also assume that the subsequence $\{\bm{x}^{k_j+1}\}\subseteq\mathcal{B}_{\rho}(\bm{x}^*)$ converges to some point $\bar{\bm{x}}$ (not necessarily equal to $\bm{x}^*$), and that for each $j \in \mathbb{N}$, the line search criterion \eqref{nexPGA-lscond} fails to be satisfied within the first $N$ inner iterations, where $N$ is the index ensuring \eqref{betageqN}. Let $i_{k_j}$ denote the number of inner iterations at the $k_j$-th outer iteration, and let $\widehat{\gamma}_{k_j}:=\gamma_{k_j,{i_{k_j}-1}}\,(=\overline{\gamma}_{k_j}/\tau)$, $\widehat{\beta}_{k_j}:=\beta_{k_j,i_{k_j}-1}\,(=\overline{\beta}_{k_j}/\eta)$, $\widehat{\bm{x}}^{k_j}:=\bm{x}^{k_j,i_{k_j}-1}$, and $\widehat{\bm{y}}^{k_j}:=\bm{y}^{k_j,i_{k_j}-1}(=\bm{x}^{k_j}+\widehat{\beta}_{k_j}(\bm{x}^{k_j}-\bm{x}^{k_j-1}))$. We see that $i_{k_j}\geq N+1$ (by the hypothesis) and $(\widehat{\bm{x}}^{k_j},\widehat{\gamma}_{k_j})$ fails to satisfy \eqref{nexPGA-lscond}, that is, we have
\begin{equation*}
H_{\delta}(\widehat{\bm{x}}^{k_j},\bm{x}^{k_j},\widehat{\gamma}_{k_j})
- \mathcal{R}_{k_j}
> - \textstyle{\frac{(1-\delta)\widehat{\gamma}_{k_j}}{8}}
\|\widehat{\bm{x}}^{k_j}-\bm{x}^{k_j}\|^2,
\end{equation*}
which, together with Proposition \ref{pro-basic-property}(i), yields
\begin{equation*}
H_{\delta}(\widehat{\bm{x}}^{k_j},\bm{x}^{k_j},\widehat{\gamma}_{k_j})
- H_{\delta}({\bm{x}}^{k_j},\bm{x}^{k_j-1},\overline{\gamma}_{k_j-1})
> - \textstyle{\frac{(1-\delta)\widehat{\gamma}_{k_j}}{8}}
\|\widehat{\bm{x}}^{k_j}-\bm{x}^{k_j}\|^2.
\end{equation*}
Recall from the definition of $H_{\delta}$ in \eqref{definition-H-delta}, we further have that
\begin{equation}\label{psi-inequality}
F(\widehat{\bm{x}}^{k_j}) - F(\bm{x}^{k_j})
> \textstyle{\frac{\delta\overline{\gamma}_{k_j-1}}{8}}
\|\bm{x}^{k_j}-\bm{x}^{k_j-1}\|^2
- \textstyle{\frac{\widehat{\gamma}_{k_j}}{8}}\|\widehat{\bm{x}}^{k_j}-\bm{x}^{k_j}\|^2.
\end{equation}
Moreover, it follows from \eqref{extrapolation-yki} and Proposition \ref{pro-basic-property}(iv) that
\begin{equation}\label{ykj-xkj-limit}
\|\widehat{\bm{y}}^{k_j}-\bm{x}^{k_j}\|
=\widehat{\beta}_{k_j}\|\bm{x}^{k_j}-\bm{x}^{k_j-1}\| \to 0,
\quad \text{as} ~~ j\to\infty,
\end{equation}
which, together with the boundedness of $\{\bm{x}^{k_j}\}$ (since $\{\bm{x}^{k_j+1}\}$ lies in the neighborhood $\mathcal{B}_{\rho}(\bm{x}^*)$ and $\|\bm{x}^{k_j+1}-\bm{x}^{k_j}\| \to 0$ by Proposition \ref{pro-basic-property}(iv)), implies that the sequence $\{\widehat{\bm{y}}^{k_j}\}$ is bounded. Additionally, since $P_2$ is a continuous convex function (by Assumption \ref{assum-funs1}3) and $\{\bm{x}^{k_j}\}$ is bounded, we know from \cite[Theorem 23.4]{r1970convex} that $\{\bm{\xi}^{k_j}\}$ ($\bm{\xi}^{k_j}\in\partial P_2(\bm{x}^{k_j})$) is also bounded. Combining these facts with the continuity of $\nabla f$ (by Assumption \ref{assum-funs1}1), 
we conclude that $\{\|\nabla f(\widehat{\bm{y}}^{k_j})-\bm{\xi}^{k_j}\|\}$ is bounded.

\vspace{1mm}
\textit{Step 2}. We next claim that $\widehat{\bm{x}}^{k_j}, \,\widehat{\bm{y}}^{k_j}\in \mathcal{B}_{2\rho}(\bm{x}^*)$ for all sufficiently large $j \in \mathbb{N}$. To this end, we first show that
\begin{equation}\label{diffhatxy}
\widehat{\bm{x}}^{k_j}-\widehat{\bm{y}}^{k_j}\to0
\quad \text{as} ~~ j\to\infty.
\end{equation}
For each $j\in\mathbb{N}$, since $\widehat{\bm{x}}^{k_j}$ is an optimal solution of the subproblem \eqref{nexPGA-subpro} with parameters $\widehat{\gamma}_{k_j}$, $\widehat{\beta}_{k_j}$ and $\widehat{\bm{y}}^{k_j}$, we have that
\begin{equation}\label{P1-ineq}
\begin{aligned}
&\quad \langle\nabla f(\widehat{\bm{y}}^{k_j})-\bm{\xi}^{k_j}, \widehat{\bm{x}}^{k_j}-\widehat{\bm{y}}^{k_j}\rangle
+ \textstyle{\frac{\widehat{\gamma}_{k_j}}{2}}\|\widehat{\bm{x}}^{k_j}-\widehat{\bm{y}}^{k_j}\|^2
+ P_1(\widehat{\bm{x}}^{k_j})   \\
&\leq \langle\nabla f(\widehat{\bm{y}}^{k_j})-\bm{\xi}^{k_j}, \bm{x}^{k_j}-\widehat{\bm{y}}^{k_j}\rangle
+ \textstyle{\frac{\widehat{\gamma}_{k_j}}{2}}\|\bm{x}^{k_j}-\widehat{\bm{y}}^{k_j}\|^2
+ P_1(\bm{x}^{k_j}).
\end{aligned}
\end{equation}
This further implies that 
\begin{equation}\label{CS-fur-property}
\begin{aligned}
\textstyle
&\quad \textstyle{\frac{\widehat{\gamma}_{k_j}}{2}}\|\widehat{\bm{x}}^{k_j}
-\widehat{\bm{y}}^{k_j}\|^2 \\
&\leq \|\nabla f(\widehat{\bm{y}}^{k_j})-\bm{\xi}^{k_j}\|
\|\widehat{\bm{x}}^{k_j}-\bm{x}^{k_j}\|
+ \textstyle{\frac{\widehat{\gamma}_{k_j}}{2}}\|\bm{x}^{k_j}
- \widehat{\bm{y}}^{k_j}\|^2
+ P_1(\bm{x}^{k_j}) - P_1(\widehat{\bm{x}}^{k_j})  \\
&= \|\nabla f(\widehat{\bm{y}}^{k_j})-\bm{\xi}^{k_j}\| \|\widehat{\bm{x}}^{k_j}-\bm{x}^{k_j}\|
+ \textstyle{\frac{\widehat{\gamma}_{k_j}\widehat{\beta}_{k_j}^2}{2}}\|\bm{x}^{k_j}-\bm{x}^{k_j-1}\|^2
+ F(\bm{x}^{k_j})- f(\bm{x}^{k_j}) + P_2(\bm{x}^{k_j}) - P_1(\widehat{\bm{x}}^{k_j})  \\
&\leq \|\nabla f(\widehat{\bm{y}}^{k_j})-\bm{\xi}^{k_j}\| \|\widehat{\bm{x}}^{k_j}-\bm{x}^{k_j}\| + H_{\delta}(\bm{x}^{k_j}, \bm{x}^{k_j-1}, \overline{\gamma}_{k_j-1}) - f(\bm{x}^{k_j}) + P_2(\bm{x}^{k_j}) - P_1(\widehat{\bm{x}}^{k_j})\\
&\leq \|\nabla f(\widehat{\bm{y}}^{k_j})-\bm{\xi}^{k_j}\| \|\widehat{\bm{x}}^{k_j}-\bm{x}^{k_j}\| + H_{\delta}(\bm{x}^{k_j}, \bm{x}^{k_j-1}, \overline{\gamma}_{k_j-1}) - f(\bm{x}^{k_j}) + P_2(\bm{x}^{k_j}) + \frac{\gamma}{2}\|\widehat{\bm{x}}^{k_j}\|^2 - \Delta \\
& =\|\nabla f(\widehat{\bm{y}}^{k_j})-\bm{\xi}^{k_j}\| \|\widehat{\bm{x}}^{k_j}-\bm{x}^{k_j}\| + H_{\delta}(\bm{x}^{k_j}, \bm{x}^{k_j-1}, \overline{\gamma}_{k_j-1}) \\
&\textstyle\qquad - f(\bm{x}^{k_j}) + P_2(\bm{x}^{k_j}) + \frac{\gamma}{2}\|\widehat{\bm{x}}^{k_j}-\widehat{\bm{y}}^{k_j}\|^2 + \frac{\gamma}{2}\|\widehat{\bm{y}}^{k_j}\|^2 + \langle\widehat{\bm{x}}^{k_j}-\widehat{\bm{y}}^{k_j},\, \gamma\widehat{\bm{y}}^{k_j}\rangle - \Delta\\
&\leq \big(\|\nabla f(\widehat{\bm{y}}^{k_j})-\bm{\xi}^{k_j}\| + \gamma\|\widehat{\bm{y}}^{k_j}\|\big) \|\widehat{\bm{x}}^{k_j}-\widehat{\bm{y}}^{k_j}\|
+ \|\nabla f(\widehat{\bm{y}}^{k_j})-\bm{\xi}^{k_j}\| \|\widehat{\bm{y}}^{k_j}-\bm{x}^{k_j}\| + F(\bm{x}^0) \\
&\textstyle\qquad - f(\bm{x}^{k_j}) + P_2(\bm{x}^{k_j}) + \frac{\gamma}{2}\|\widehat{\bm{x}}^{k_j}-\widehat{\bm{y}}^{k_j}\|^2 + \frac{\gamma}{2}\|\widehat{\bm{y}}^{k_j}\|^2 - \Delta,
\end{aligned}
\end{equation}
where the equality follows from $\widehat{\bm{y}}^{k_j}=\bm{x}^{k_j}+\widehat{\beta}_{k_j}(\bm{x}^{k_j}-\bm{x}^{k_j-1})$ and $P_1=F-f+P_2$, the second inequality follows from the definition of $H_{\delta}$ in \eqref{definition-H-delta}, along with the fact that \eqref{betageqN} holds for $\widehat{\gamma}_{k_j}=\gamma_{k_j,{i_{k_j}-1}}$ and $\widehat{\beta}_{k_j}=\beta_{k_j,i_{k_j}-1}$ with $i_{k_j}\geq N+1$, the third inequality follows from Assumption \ref{assum-funs1}2 ($P_1$ is prox-bounded with $\Delta\coloneqq\inf\left\{P_1 + \frac{\gamma}{2}\|\cdot\|^2\right\} > -\infty$ for some $\gamma>0$), and the last inequality follows from the triangle inequality and $H_{\delta}(\bm{x}^{k_j},\bm{x}^{k_j-1},\overline{\gamma}_{k_j-1})\leq\mathcal{R}_{k_j}\leq\mathcal{R}_0=F(\bm{x}^0)$ by Proposition \ref{pro-basic-property}(i)\&(ii).

Using \eqref{CS-fur-property}, we can prove \eqref{diffhatxy} by contradiction as follows. Assume that \eqref{diffhatxy} does not hold. Then, there must exist a subsequence $\{(\widehat{\bm{x}}^{k_{j_t}},\widehat{\bm{y}}^{k_{j_t}})\}$ such that for some $\varepsilon > 0$ we have $\|\widehat{\bm{x}}^{k_{j_t}}-\widehat{\bm{y}}^{k_{j_t}}\|\ge \varepsilon$ for all $t$. Now, from the boundedness of $\{\bm{x}^{k_j}\}$ together with the continuity of $f$ (by Assumption \ref{assum-funs1}1) and $P_2$ (by Assumption \ref{assum-funs1}3), the boundedness of $\{\widehat{\bm{y}}^{k_j}\}$ and $\left\{\|\nabla f(\widehat{\bm{y}}^{k_j})-\bm{\xi}^{k_j}\|\right\}$ (as shown in \textit{Step 1}), \eqref{ykj-xkj-limit} and $\widehat{\gamma}_{k_j}=\overline{\gamma}_{k_j}/\tau\to\infty$, one can see that the left-hand side of \eqref{CS-fur-property} would grow faster than the right-hand side of \eqref{CS-fur-property} along the subsequence $\{(\widehat{\bm{x}}^{k_{j_t}},\widehat{\bm{y}}^{k_{j_t}})\}$, thereby leading to a contradiction as $t\to\infty$. Thus, using \eqref{ykj-xkj-limit}, \eqref{diffhatxy} and $\bm{x}^{k_j}\to\bar{\bm{x}}$ (from $\|\bm{x}^{k_j+1}-\bm{x}^{k_j}\| \to 0$ and the assumption that $\{\bm{x}^{k_j+1}\}\subseteq\mathcal{B}_{\rho}(\bm{x}^*)$ converges to $\bar{\bm{x}}$), we see that $\widehat{\bm{x}}^{k_j}\to\bar{\bm{x}}$ and $\widehat{\bm{y}}^{k_j} \to \bar{\bm{x}}$ as $j\to\infty$. These, together with $\bar{\bm{x}}\in\mathcal{B}_{\rho}(\bm{x}^*)$, imply that $\bm{x}^{k_j},\,\widehat{\bm{x}}^{k_j},\,\widehat{\bm{y}}^{k_j}\in \mathcal{B}_{2\rho}(\bm{x}^*)$ for all sufficiently large $j\in\mathbb{N}$.

\textit{Step 3}. In the following, we shall derive an upper bound for $\frac{\widehat{\gamma}_{k_j}}{2}\|\widehat{\bm{x}}^{k_j}-\bm{x}^{k_j}\|^2$. First, we see from Assumption \ref{assum-funs1}1 that $\nabla f$ is (globally) Lipschitz continuous on the compact set $\mathcal{B}_{2\rho}(\bm{x}^*)$ and we denote the corresponding Lipschitz constant by $L_{2\rho,\bm{x}^*}$. Next, by the mean-value theorem, the convexity of $P_2$, and the fact that $\bm{\xi}^{k_j}\in\partial P_2(\bm{x}^{k_j})$, there exists a point $\bm{s}^{k_j}$ on the line segment between $\bm{x}^{k_j}$ and $\widehat{\bm{x}}^{k_j}$ such that
\begin{equation*}
\begin{aligned}
F(\widehat{\bm{x}}^{k_j})-F(\bm{x}^{k_j})
&= f(\widehat{\bm{x}}^{k_j}) + P_1(\widehat{\bm{x}}^{k_j}) - P_2(\widehat{\bm{x}}^{k_j})
- f(\bm{x}^{k_j}) - P_1(\bm{x}^{k_j}) + P_2(\bm{x}^{k_j})   \\
&= \langle\nabla f(\bm{s}^{k_j}),\,\widehat{\bm{x}}^{k_j}-\bm{x}^{k_j}\rangle
+ P_1(\widehat{\bm{x}}^{k_j}) - P_1(\bm{x}^{k_j})
- P_2(\widehat{\bm{x}}^{k_j}) + P_2(\bm{x}^{k_j})  \\
&\leq \langle\nabla f(\bm{s}^{k_j})-\bm{\xi}^{k_j},\, \widehat{\bm{x}}^{k_j}-\bm{x}^{k_j}\rangle
+ P_1(\widehat{\bm{x}}^{k_j}) - P_1(\bm{x}^{k_j}).
\end{aligned}
\end{equation*}
Combining this with \eqref{P1-ineq} further yields that
\begin{equation*}
\hspace{-4mm}
\begin{aligned}
&\quad F(\widehat{\bm{x}}^{k_j})-F(\bm{x}^{k_j})  \\
&\leq\langle\nabla f(\bm{s}^{k_j})-\nabla f(\widehat{\bm{y}}^{k_j}), \,\widehat{\bm{x}}^{k_j}-\bm{x}^{k_j}\rangle
+ {\textstyle\frac{\widehat{\gamma}_{k_j}}{2}}\|\bm{x}^{k_j}-\widehat{\bm{y}}^{k_j}\|^2
- {\textstyle\frac{\widehat{\gamma}_{k_j}}{2}}\|\widehat{\bm{x}}^{k_j}-\widehat{\bm{y}}^{k_j}\|^2 \\
&= \langle\nabla f(\bm{s}^{k_j}) - \nabla f(\widehat{\bm{y}}^{k_j}), \,\widehat{\bm{x}}^{k_j}-\bm{x}^{k_j}\rangle
- {\textstyle\frac{\widehat{\gamma}_{k_j}}{2}}\|\bm{x}^{k_j}-\widehat{\bm{x}}^{k_j}\|^2
- \widehat{\gamma}_{k_j}\langle\widehat{\bm{x}}^{k_j}-\bm{x}^{k_j},\,\bm{x}^{k_j}-\widehat{\bm{y}}^{k_j}\rangle \\
&\leq\|\nabla f(\bm{s}^{k_j})-\nabla f(\widehat{\bm{y}}^{k_j})\|
\|\widehat{\bm{x}}^{k_j}-\bm{x}^{k_j}\|
- {\textstyle\frac{\widehat{\gamma}_{k_j}}{2}}\|\bm{x}^{k_j}-\widehat{\bm{x}}^{k_j}\|^2
+ \widehat{\gamma}_{k_j}\|\widehat{\bm{x}}^{k_j}-\bm{x}^{k_j}\|
\|\bm{x}^{k_j}-\widehat{\bm{y}}^{k_j}\|.
\end{aligned}
\end{equation*}
This implies that, for all sufficiently large $j\in\mathbb{N}$,
\begin{equation*}
\hspace{-4mm}
\begin{aligned}
&\quad {\textstyle\frac{\widehat{\gamma}_{k_j}}{2}}\|\widehat{\bm{x}}^{k_j}-\bm{x}^{k_j}\|^2 \\
&\leq \|\nabla f(\bm{s}^{k_j})-\nabla f(\widehat{\bm{y}}^{k_j})\|
\|\widehat{\bm{x}}^{k_j}-\bm{x}^{k_j}\|
+ \widehat{\gamma}_{k_j}\|\widehat{\bm{x}}^{k_j}-\bm{x}^{k_j}\|
\|\bm{x}^{k_j}-\widehat{\bm{y}}^{k_j}\| 
+ F(\bm{x}^{k_j}) - F(\widehat{\bm{x}}^{k_j})  \\
&\leq\|\nabla f(\bm{s}^{k_j})-\nabla f(\bm{x}^{k_j})\|
\|\widehat{\bm{x}}^{k_j}-\bm{x}^{k_j}\|
+ \|\nabla f(\bm{x}^{k_j})-\nabla f(\widehat{\bm{y}}^{k_j})\|
\|\widehat{\bm{x}}^{k_j}-\bm{x}^{k_j}\|  \\
&\qquad
+\widehat{\gamma}_{k_j}\|\widehat{\bm{x}}^{k_j}-\bm{x}^{k_j}\|\|\bm{x}^{k_j}-\widehat{\bm{y}}^{k_j}\|
-{\textstyle\frac{\delta\overline{\gamma}_{k_j-1}}{8}}\|\bm{x}^{k_j}-\bm{x}^{k_j-1}\|^2
+{\textstyle\frac{\widehat{\gamma}_{k_j}}{8}}\|\widehat{\bm{x}}^{k_j}-\bm{x}^{k_j}\|^2 \\
&\leq L_{2\rho,\bm{x}^*}\|\bm{s}^{k_j}-\bm{x}^{k_j}\|\|\widehat{\bm{x}}^{k_j}-\bm{x}^{k_j}\|
+ (L_{2\rho,\bm{x}^*}+\widehat{\gamma}_{k_j})\|\bm{x}^{k_j}-\widehat{\bm{y}}^{k_j}\|
\|\widehat{\bm{x}}^{k_j}-\bm{x}^{k_j}\|  \\
&\qquad
-{\textstyle\frac{\delta\overline{\gamma}_{k_j-1}}{8}}\|\bm{x}^{k_j}-\bm{x}^{k_j-1}\|^2
+{\textstyle\frac{\widehat{\gamma}_{k_j}}{8}}\|\widehat{\bm{x}}^{k_j}-\bm{x}^{k_j}\|^2 \\
&\leq L_{2\rho,\bm{x}^*}\|\bm{s}^{k_j}-\bm{x}^{k_j}\|\|\widehat{\bm{x}}^{k_j}-\bm{x}^{k_j}\|
+ 2\widehat{\gamma}_{k_j}\|\bm{x}^{k_j}-\widehat{\bm{y}}^{k_j}\|\|\widehat{\bm{x}}^{k_j}-\bm{x}^{k_j}\| \\
&\qquad
-{\textstyle\frac{\delta\overline{\gamma}_{k_j-1}}{8}}\|\bm{x}^{k_j}-\bm{x}^{k_j-1}\|^2
+{\textstyle\frac{\widehat{\gamma}_{k_j}}{8}}\|\widehat{\bm{x}}^{k_j}-\bm{x}^{k_j}\|^2 \\
&\leq L_{2\rho,\bm{x}^*}\|\bm{s}^{k_j}-\bm{x}^{k_j}\|\|\widehat{\bm{x}}^{k_j}-\bm{x}^{k_j}\|
+{\textstyle\frac{3\widehat{\gamma}_{k_j}}{8}}\|\widehat{\bm{x}}^{k_j}-\bm{x}^{k_j}\|^2
+4\widehat{\gamma}_{k_j}\|\bm{x}^{k_j}-\widehat{\bm{y}}^{k_j}\|^2 -{\textstyle\frac{\delta\overline{\gamma}_{k_j-1}}{8}}\|\bm{x}^{k_j}\!-\!\bm{x}^{k_j-1}\|^2 \\
&= L_{2\rho,\bm{x}^*}\|\bm{s}^{k_j}-\bm{x}^{k_j}\|
\|\widehat{\bm{x}}^{k_j}-\bm{x}^{k_j}\|
+{\textstyle\frac{3\widehat{\gamma}_{k_j}}{8}}\|\widehat{\bm{x}}^{k_j}-\bm{x}^{k_j}\|^2
+\left(\textstyle{4\widehat{\gamma}_{k_j}\widehat{\beta}_{k_j}^2}-{\textstyle\frac{\delta\overline{\gamma}_{k_j-1}}{8}}\right)\|\bm{x}^{k_j}-\bm{x}^{k_j-1}\|^2 \\
&\leq L_{2\rho,\bm{x}^*}\|\bm{s}^{k_j}-\bm{x}^{k_j}\|\|\widehat{\bm{x}}^{k_j}-\bm{x}^{k_j}\|
+ {\textstyle\frac{3\widehat{\gamma}_{k_j}}{8}}\|\widehat{\bm{x}}^{k_j}-\bm{x}^{k_j}\|^2,
\end{aligned}
\end{equation*}
where the second inequality follows from the triangle inequality and \eqref{psi-inequality}, the third inequality follows from the Lipschitz continuity of $\nabla f$ on the compact set $\mathcal{B}_{2\rho}(\bm{x}^*)$, the fourth inequality follows from $\widehat{\gamma}_{k_j}\to\infty$ as $j\to\infty$, the second last inequality follows from the relation $2ab \leq a^2+b^2$ with $a=2\sqrt{\widehat{\gamma}_{k_j}}\|\bm{x}^{k_j}-\widehat{\bm{y}}^{k_j}\|$ and $b=\frac{1}{2}\sqrt{\widehat{\gamma}_{k_j}}\|\widehat{\bm{x}}^{k_j}-\bm{x}^{k_j}\|$, the equality follows from $\widehat{\bm{y}}^{k_j}=\bm{x}^{k_j}+\widehat{\beta}_{k_j}(\bm{x}^{k_j}-\bm{x}^{k_j-1})$, the last inequality follows from the fact that \eqref{betageqN} holds for $\widehat{\gamma}_{k_j}=\gamma_{k_j,{i_{k_j}-1}}$ and $\widehat{\beta}_{k_j}=\beta_{k_j,i_{k_j}-1}$ with $i_{k_j}\geq N+1$.

\vspace{1mm}
\textit{Step 4}. Finally, using the above relation and the fact that $\bm{s}^{k_j}$ is on the line segment between ${\bm{x}}^{k_j}$ and $\widehat{\bm{x}}^{k_j}$,
we obtain that
\begin{equation*}
\textstyle{\frac{\widehat{\gamma}_{k_j}}{8}}\|\widehat{\bm{x}}^{k_j}-\bm{x}^{k_j}\|^2
\leq L_{2\rho,\bm{x}^*}\|\bm{s}^{k_j}-\bm{x}^{k_j}\|\|\widehat{\bm{x}}^{k_j}-\bm{x}^{k_j}\|
\leq L_{2\rho,\bm{x}^*}\|\widehat{\bm{x}}^{k_j}-\bm{x}^{k_j}\|^2 
\end{equation*}
holds for all sufficiently large $j \in \mathbb{N}$. Since $\widehat{\bm{x}}^{k_j}-\bm{x}^{k_j}\neq0$, which is guaranteed by \eqref{psi-inequality}, the above relation then leads to a contradiction if $\widehat{\gamma}_{k_j}\to\infty$ as $j\to\infty$. This implies that $\{\overline{\gamma}_k\}_{k\in\mathcal{J}_{\rho}(\bm{x}^*)}$ is bounded and completes the proof.
\end{proof}

We are now ready to establish the subsequential convergence for nexPGA.

\begin{theorem}\label{nexPGA-theorem-subsequence}
Suppose that Assumptions \ref{assum-para} and \ref{assum-funs1} hold. Let $\{\bm{x}^k\}$ be the sequence generated by the nexPGA in Algorithm \ref{algo-nexPGA}, and let $\zeta$ be given in Proposition \ref{pro-basic-property}(ii). Moreover, if the sequence $\{\bm{x}^k\}$ has an accumulation point $\bm{x}^*$. Then, $\bm{x}^*$ is a stationary point of problem \eqref{ge-DC-problem}, and $F(\bm{x}^*)={\zeta}$.
\end{theorem}
\begin{proof}
Let $\{\bm{x}^{k_j}\}_{j\in\mathbb{N}}$ be a subsequence converging to $\bm{x}^*$. It follows from Proposition \ref{pro-basic-property}(iv) and the updating rules in \textbf{Step 2} that $\{\bm{x}^{k_j+1}\}_{j\in\mathbb{N}}$ and $\{\overline{\bm{y}}^{k_j}\}_{j\in\mathbb{N}}$ converge to $\bm{x}^*$. Thus, for any $\rho>0$, we have $\bm{x}^{k_j+1}\in\mathcal{B}_{\rho}(\bm{x}^*)(:=\left\{\bm{x}\in\mathbb{R}^n:\|\bm{x}-\bm{x}^*\|\leq\rho\right\})$ for all sufficiently large $j\in\mathbb{N}$. Then, by Proposition \ref{pro-gammak-bounded}, there exists a positive constant $\overline{\gamma}_{\rho,\bm{x}^*}>0$ such that $\overline{\gamma}_{k_j}\leq\overline{\gamma}_{\rho,\bm{x}^*}$ for all sufficiently large $j\in\mathbb{N}$. Moreover, since $P_2$ is a continuous convex function (by Assumption \ref{assum-funs1}3) and $\{\bm{x}^{k_j}\}$ is bounded, we know from \cite[Theorem 23.4]{r1970convex} that $\{\bm{\xi}^{k_j}\}$ ($\bm{\xi}^{k_j}\in\partial P_2(\bm{x}^{k_j})$) is also bounded. Thus, by passing to a further subsequence if necessary, we may assume without loss of generality that $\bm{\xi}^*:=\lim_{j\to\infty}\bm{\xi}^{k_j}$ exists and $\bm{\xi}^*\in\partial P_2(\bm{x}^*)$ due to the closedness of $\partial P_2$.

We next prove that $P_1(\bm{x}^{k_j+1}) \to P_1(\bm{x}^*)$. Since $\bm{x}^{k_j+1}$ is an optimal solution of the subproblem \eqref{nexPGA-subpro} with parameters $\overline{\gamma}_{k_j}$ and $\overline{\bm{y}}^{k_j}=\bm{x}^{k_j}+\overline{\beta}_{k_j}(\bm{x}^{k_j}-\bm{x}^{k_j-1})$, we have
\begin{equation*}
\begin{aligned}
&\quad \langle\nabla f(\overline{\bm{y}}^{k_j})-\bm{\xi}^{k_j}, \,\bm{x}^{k_j+1}-\overline{\bm{y}}^{k_j}\rangle
+ \textstyle{\frac{\overline{\gamma}_{k_j}}{2}}\|\bm{x}^{k_j+1}-\overline{\bm{y}}^{k_j}\|^2
+ P_1(\bm{x}^{k_j+1})  \\
&\leq \langle\nabla f(\overline{\bm{y}}^{k_j})-\bm{\xi}^{k_j}, \,\bm{x}^*-\overline{\bm{y}}^{k_j}\rangle
+ \textstyle{\frac{\overline{\gamma}_{k_j}}{2}}\|\bm{x}^*-\overline{\bm{y}}^{k_j}\|^2
+ P_1(\bm{x}^*),
\quad \forall\,j\in\mathbb{N},
\end{aligned}
\end{equation*}
which implies that
\begin{equation*}
\begin{aligned}
 P_1(\bm{x}^{k_j+1})
&\leq P_1(\bm{x}^*) + \langle\nabla f(\overline{\bm{y}}^{k_j})-\bm{\xi}^{k_j}, \,\bm{x}^*-\bm{x}^{k_j+1}\rangle + \textstyle\frac{\overline{\gamma}_{k_j}}{2}\|\bm{x}^*-\overline{\bm{y}}^{k_j}\|^2
- \frac{\overline{\gamma}_{k_j}}{2}\|\bm{x}^{k_j+1}-\overline{\bm{y}}^{k_j}\|^2.
\end{aligned}
\end{equation*}
Passing to the limit in the above relation, and invoking $\bm{x}^{k_j+1}-\bm{x}^{k_j}\to0$ (by Proposition \ref{pro-basic-property}(iv)), $\bm{x}^{k_j}\to\bm{x}^*$, $\overline{\bm{y}}^{k_j}\to\bm{x}^*$, $\bm{\xi}^{k_j}\to \bm{\xi}^*$, the continuity of $\nabla f$ (by Assumption \ref{assum-funs1}1) and the fact that $\overline{\gamma}_{k_j}\leq\overline{\gamma}_{\rho,\bm{x}^*}$ for all sufficiently large $j\in\mathbb{N}$, we obtain that $\limsup\limits_{j\to\infty}\,P_1(\bm{x}^{k_j+1})\leq P_1(\bm{x}^*)$. On the other hand, since $\bm{x}^{k_j+1} \rightarrow \bm{x}^*$ and $P_1$ is lower semicontinuous (by Assumption \ref{assum-funs1}2), we have that $P_1(\bm{x}^*) \leq \liminf\limits_{j\to\infty}\,P_1(\bm{x}^{k_j+1})$. Therefore, we can conclude that $P_1(\bm{x}^{k_j+1}) \to P_1(\bm{x}^*)$.

Finally, it follows from the first-order optimality condition for \eqref{nexPGA-subpro} that
\begin{equation*}
\textstyle 0 \in \nabla f(\overline{\bm{y}}^{k_j}) - \bm{\xi}^{k_j}
+ \overline{\gamma}_{k_j}(\bm{x}^{k_j+1}-\overline{\bm{y}}^{k_j})
+ \partial P_1(\bm{x}^{k_j+1}),
\quad \forall\,j\in\mathbb{N}.
\end{equation*}
Passing to the limit in the above relation, and invoking \eqref{robust}, $P_1(\bm{x}^{k_j+1}) \to P_1(\bm{x}^*)$, $\bm{x}^{k_j+1}\to\bm{x}^*$, $\overline{\bm{y}}^{k_j}\to\bm{x}^*$, $\bm{\xi}^{k_j}\to \bm{\xi}^*$, the continuity of $\nabla f$ (by Assumption \ref{assum-funs1}1) and the fact that $\overline{\gamma}_{k_j}\leq\overline{\gamma}_{\rho,\bm{x}^*}$ for all sufficiently large $j\in\mathbb{N}$, we obtain that
\begin{equation*}
0 \in \nabla f(\bm{x}^*)+\partial P_1(\bm{x}^*)-\partial P_2(\bm{x}^*),
\end{equation*}
which implies that $\bm{x}^*$ is a stationary point of problem \eqref{ge-DC-problem}.

In addition, it follows from Proposition \ref{pro-basic-property}(iii)\&(iv), the definition of $H_{\delta}$ in \eqref{definition-H-delta}, and the bound $\overline{\gamma}_{k_j}\leq\overline{\gamma}_{\rho,\bm{x}^*}$ for all sufficiently large $j\in\mathbb{N}$ that
\begin{equation*}
\lim\limits_{j\rightarrow \infty} F(\bm{x}^{k_j+1})
= \lim_{j\to\infty} H_\delta(\bm{x}^{k_j+1}, \bm{x}^{k_j}, \overline{\gamma}_{k_j})={\zeta}. 
\end{equation*}

Using this relation and the continuity of $f$ and $P_2$, together with $\bm{x}^{k_j+1} \rightarrow \bm{x}^*$ and $P_1(\bm{x}^{k_j+1}) \rightarrow P_1(\bm{x}^*)$ as shown above, we can conclude that
\begin{equation*}
F(\bm{x}^*)
=f(\bm{x}^*)+P_1(\bm{x}^*)-P_2(\bm{x}^*)
=\lim\limits_{j\rightarrow\infty}\left\{f(\bm{x}^{k_j+1})+P_1(\bm{x}^{k_j+1})-P_2(\bm{x}^{k_j+1})\right\}
=\zeta. 
\end{equation*}
This completes the proof.
\end{proof}

Before closing this section, we would like to emphasize that Theorem \ref{nexPGA-theorem-subsequence} is established \textit{without} requiring either the global Lipschitz continuity of $\nabla f$, a boundedness assumption on the generated sequence $\{\bm{x}^k\}$, or a level-boundedness assumption on the objective (or potential) function (which, in turn, ensures bounded iterates). In contrast, one or more of these assumptions are typically imposed in prior works when analyzing the convergence of PG-type algorithms; see, e.g., \cite{gzlhy2013general,ll2015accelerated,wcp2018proximal,wnf2009sparse,yang2024proximal}. In the next section, we go further by establishing global sequential convergence, along with convergence rates for both the objective values and the iterates, all under the same relaxed assumptions. These results enhance the flexibility of nexPGA and broaden its applicability to a wider range of practical problems.

%%%%%%%%%%%%%%%%%%%%%%%%%%%%%%%%%%%%%%%%%%%%%%%%%%%%%%%%%%%%%%%%%
\section{The KL-based convergence analysis}\label{sec-kl-analysis}

In this section, we further investigate the convergence properties of nexPGA under the Kurdyka-{\L}ojasiewicz (KL) property and its associated exponent. The KL property has become a fundamental tool for analyzing the global sequential convergence and convergence rates of nonconvex first-order methods. However, most existing convergence analyses rely on a sufficiently \textit{monotone} decrease in the sequence of (potential) objective function values, which is a crucial condition for effectively applying the KL inequality; see, e.g., \cite{bbt2017descent,jkm2023convergence,tft2022new}. This reliance poses a significant challenge when analyzing \textit{nonmonotone} methods, where such monotonicity does not hold.

Only recently the KL-based analysis has been successfully extended to the nonmonotone setting. Specifically, Yang \cite{yang2024proximal} established convergence rate results for the objective values in a GLL-type nonmonotone proximal gradient (NPG) method, but the sequential convergence of the iterates was not addressed. Subsequently, Qian and Pan \cite{qp2023convergence} were the first to prove both global sequential convergence and corresponding rate results for a class of GLL-type methods. These results were further improved by Qian et al. \cite{qtpq2025gll}, who derived similar guarantees under weaker and more verifiable conditions. More recently, Qian et al. \cite{qtpq2024convergence} developed a novel KL-based analysis framework for ZH-type nonmonotone descent methods, which notably avoids the restrictive gap condition imposed in their earlier GLL-type analyses \cite{qp2023convergence,qtpq2025gll}. Their findings indicate that the ZH-type line search strategy would be theoretically more favorable. Motivated by this development, Kanzow and Lehmann \cite{kl2025convergence} further established the global sequential convergence and the linear convergence rate for a ZH-type NPG method, \textit{without} requiring the global Lipschitz continuity of $\nabla f$ or any boundedness assumption on the generated sequence.

These advancements motivate us to establish strong convergence properties for the proposed nexPGA (a ZH-type nonmonotone algorithm with extrapolation), thereby broadening its theoretical foundation. Here, we would also like to point out that, while our analysis is inspired by \cite{kl2025convergence,qtpq2024convergence}, the presence of extrapolation indeed introduces new technical challenges. In particular, the reference value $\mathcal{R}_k$ used in our line search criterion \eqref{nexPGA-lscond} is constructed based on the potential function $H_{\delta}$ in \eqref{definition-H-delta}, rather than on the original objective function $F$ as considered in \cite{kl2025convergence,qtpq2024convergence}. Consequently, the convergence analyses developed in those works cannot be directly applied to nexPGA. More importantly, our goal is to establish these convergence properties \textit{without} assuming the global Lipschitz continuity of $\nabla f$ and \textit{without} imposing any explicit or implicit boundedness assumptions on the generated sequence $\{\bm{x}^k\}$. Therefore, a more delicate and refined analysis is required. Our analysis and results would contribute to the growing body of research on the ZH-type nonmonotone algorithm.

We begin our analysis by introducing the following additional assumption, which is a standard technical condition commonly used in the global convergence analysis of existing proximal DC-type algorithms; see, e.g., \cite{tft2022new,wcp2018proximal}.

\begin{assumption}\label{assumC}
Assume that $P_2$ is continuously differentiable on an open set $\mathcal{N}_{\mathcal{S}}$ that contains $\mathcal{S}$, where $\mathcal{S}$ is the set of all stationary points of problem \eqref{ge-DC-problem}. Moreover, $\nabla P_2$ is locally Lipschitz continuous on $\mathcal{N}_{\mathcal{S}}$. 
\end{assumption}

\begin{lemma}\label{nexPGA-lem-dist-H-delta}
Suppose that Assumptions \ref{assum-para}, \ref{assum-funs1} and \ref{assumC} hold. Let $\{\bm{x}^k\}$ be the sequence generated by the nexPGA in Algorithm \ref{algo-nexPGA}. Moreover, if the sequence $\{\bm{x}^k\}$ has an accumulation point $\bm{x}^*$, then there exist $\tilde{c}>0$, $J\in\mathbb N$, and $\alpha>0$ such that 
\begin{equation}\label{dist-H-delta-k}
\operatorname{dist}\left(\bm{0}, \,\partial H_{\delta}(\bm{x}^{k}, \bm{x}^{k-1}, \overline{\gamma}_{k-1})\right)
\leq \tilde{c}\left(\|\bm{x}^{k}-\bm{x}^{k-1}\|+\|\bm{x}^{k-1}-\bm{x}^{k-2}\|\right)
\end{equation}
for all $k\in\{j\in\mathbb{N}:\bm{x}^{j}\in\mathcal{B}_{\alpha}(\bm{x}^*),\,j\geq J\}$ with $\mathcal{B}_{\alpha}(\bm{x}^*):=\left\{\bm{x}\in\mathbb{R}^n:\|\bm{x}-\bm{x}^*\|\leq\alpha\right\}$.
\end{lemma}
\begin{proof}
First, since $\bm{x}^*$ is an accumulation point of $\{\bm{x}^k\}$, it follows from Theorem \ref{nexPGA-theorem-subsequence} that $\bm{x}^*$ is a stationary point of problem \eqref{ge-DC-problem} and thus $\bm{x}^*\in\mathcal{S}\subset\mathcal{N}_{\mathcal{S}}$, where $\mathcal{S}$ and $\mathcal{N}_{\mathcal{S}}$ are defined in Assumption \ref{assumC}. This ensures the existence of a positive constant $\vartheta>0$ such that $\mathcal{B}_{\vartheta}(\bm{x}^*):=\left\{\bm{x}\in\mathbb{R}^n:\|\bm{x}-\bm{x}^*\|<\vartheta\right\}\subset\mathcal{N}_{\mathcal{S}}$. Moreover, since $\lim\limits_{k\rightarrow\infty}\|\bm{x}^{k+1}-\bm{x}^{k}\|=0$ by Proposition \ref{pro-basic-property}(iv), there exist an
integer $K_1$ and a positive constant $\alpha>0$ such that the following inequality holds:
\begin{equation}\label{sup-alpha-vartheta}
(1+\delta\beta_{\max})\,{\textstyle\sup_{t \geq K_1}}\left\{\|\bm{x}^{t+1}-\bm{x}^t\|\right\}+\alpha
< \min\{\vartheta,1\}. 
\end{equation}
In the following, consider an arbitrary $k\in\{j\in\mathbb{N}:\bm{x}^{j}\in\mathcal{B}_{\alpha}(\bm{x}^*),\,j\geq K_1+2\}$ with $\mathcal{B}_{\alpha}(\bm{x}^*)=\left\{\bm{x}\in\mathbb{R}^n:\|\bm{x}-\bm{x}^*\|\leq\alpha\right\}$. Clearly, $\bm{x}^{k}\in\mathcal{B}_{\alpha}(\bm{x}^*)\subset\mathcal{B}_{\vartheta}(\bm{x}^*)\subset\mathcal{N}_{\mathcal{S}}$. Thus, for such $\vartheta$ and $k$, Proposition \ref{pro-gammak-bounded} guarantees the existence of a positive constant $\overline{\gamma}_{\vartheta,\bm{x}^*}>0$ (depending on $\vartheta$ and $\bm{x}^*$) such that $\overline{\gamma}_{k-1}\leq\overline{\gamma}_{\vartheta,\bm{x}^*}$. Moreover, using \eqref{sup-alpha-vartheta} and the updating rule for $\overline{\bm{y}}^{k}$, we have 
\begin{equation*}
\begin{aligned}
\|\bm{x}^{k-1}-\bm{x}^*\|
&\leq
\|\bm{x}^{k-1}-\bm{x}^{k}\|
+ \|\bm{x}^{k}-\bm{x}^*\| \leq \,{\textstyle\sup_{t \geq K_1}} \left\{\|\bm{x}^{t+1}-\bm{x}^t\|\right\} + \alpha
< \vartheta, \\[3pt]
\|\overline{\bm{y}}^{k-1}-\bm{x}^*\|
&\leq\|\overline{\bm{y}}^{k-1}-\bm{x}^{k-1}\|+\|\bm{x}^{k-1}-\bm{x}^*\| \leq \delta\beta_{\max}\|\bm{x}^{k-1}-\bm{x}^{k-2}\|+\|\bm{x}^{k-1}-\bm{x}^{*}\| \\
&\leq (1+\delta\beta_{\max})\,{\textstyle\sup_{t \geq K_1}}\left\{\|\bm{x}^{t+1}-\bm{x}^t\|\right\} + \alpha
<\vartheta.
\end{aligned}  
\end{equation*}
These inequalities imply that $\bm{x}^{k-1},\,\overline{\bm{y}}^{k-1} \in \mathcal{B}_{\vartheta}(\bm{x}^*)\subset\mathcal{N}_{\mathcal{S}}$. On the other hand, recall from Assumption \ref{assumC} that $P_2$ is continuously differentiable on $\mathcal{N}_{\mathcal{S}}$, which, as shown above, contains both $\bm{x}^{k}$ and $\bm{x}^{k-1}$. Then, by the updating rule for $\overline{\bm{y}}^{k}$ and the first-order optimality condition of the subproblem \eqref{nexPGA-subpro}, we have that 
\begin{numcases}{}
\textstyle\overline{\bm{y}}^{k-1}=\bm{x}^{k-1}+\overline{\beta}_{k-1}(\bm{x}^{k-1}-\bm{x}^{k-2}),  \label{yk-and-optimal-1}  \\[3pt]
0 \in \nabla f(\overline{\bm{y}}^{k-1})+\overline{\gamma}_{k-1}(\bm{x}^{k}-\overline{\bm{y}}^{k-1})+\partial P_1(\bm{x}^{k})-\nabla P_2(\bm{x}^{k-1}). \label{yk-and-optimal-2}
\end{numcases}

Now, we consider the subdifferential of $H_{\delta}(\bm{u}, \bm{v}, \gamma)$ at the point $(\bm{x}^{k}, \bm{x}^{k-1}, \overline{\gamma}_{k-1})$. From the definition of $H_{\delta}$ in \eqref{definition-H-delta}, we have that
\begin{equation*}
\begin{aligned}
&\quad\textstyle\partial_{\bm{u}} H_{\delta}(\bm{x}^{k}, \bm{x}^{k-1}, \overline{\gamma}_{k-1})\\ 
&= \nabla f(\bm{x}^{k})+\partial P_1(\bm{x}^{k})-\nabla P_2(\bm{x}^{k})+\textstyle{\frac{\delta \overline{\gamma}_{k-1}}{4}}(\bm{x}^{k}-\bm{x}^{k-1}) \\
&\textstyle= \nabla f(\overline{\bm{y}}^{k-1})+\partial P_1(\bm{x}^{k})-\nabla P_2(\bm{x}^{k-1})+\nabla f(\bm{x}^{k})-\nabla f(\overline{\bm{y}}^{k-1})\\
&\qquad+\nabla P_2(\bm{x}^{k-1})-\nabla P_2(\bm{x}^{k})+\frac{\delta \overline{\gamma}_{k-1}}{4}(\bm{x}^{k}-\bm{x}^{k-1}) \\
&\textstyle\ni \nabla f(\bm{x}^{k})-\nabla f(\overline{\bm{y}}^{k-1})+ \nabla P_2(\bm{x}^{k-1})-\nabla P_2(\bm{x}^{k})+\frac{\delta \overline{\gamma}_{k-1}}{4}(\bm{x}^{k}-\bm{x}^{k-1})-\overline{\gamma}_{k-1}(\bm{x}^{k}-\overline{\bm{y}}^{k-1}),
\end{aligned}
\end{equation*}
where the inclusion follows from \eqref{yk-and-optimal-2}. Similarly,  we can obtain that
\begin{equation*}
    \textstyle\partial_{\bm{v}} H_{\delta}(\bm{x}^{k}, \bm{x}^{k-1}, \overline{\gamma}_{k-1})
= -\frac{\delta \overline{\gamma}_{k-1}}{4}(\bm{x}^{k}-\bm{x}^{k-1}),\,
\partial_\gamma H_{\delta}(\bm{x}^{k}, \bm{x}^{k-1}, \overline{\gamma}_{k-1})
= \frac{\delta}{8}\|\bm{x}^{k}-\bm{x}^{k-1}\|^2.
\end{equation*}
Thus, by combining the above relations, \eqref{sup-alpha-vartheta} and \eqref{yk-and-optimal-1}, together with the bounds $\overline{\beta}_{k-1}\leq\delta\beta_{\max}$ and $\overline{\gamma}_{k-1}\leq\overline{\gamma}_{\vartheta,\bm{x}^*}$, the local Lipschitz continuity of $\nabla f$ and $\nabla P_2$ on $\mathcal{N}_{\mathcal{S}}$, and the inclusion $\bm{x}^k,\bm{x}^{k-1},\overline{\bm{y}}^{k-1}\in\mathcal{B}_{\vartheta}(\bm{x}^*)\subset\mathcal{N}_{\mathcal{S}}$, we obtain the desired result.
\end{proof}

We are now ready to establish the sequential convergence for nexPGA.

\begin{theorem}\label{nexPGA-theorem-wholesequence}
Suppose that Assumptions \ref{assum-para}, \ref{assum-funs1} and \ref{assumC} hold. Let $\{\bm{x}^k\}$ be the sequence generated by the nexPGA in Algorithm \ref{algo-nexPGA}. Moreover, suppose that the sequence $\{\bm{x}^k\}$ has an accumulation point $\bm{x}^*$ and that the potential function $H_{\delta}$ in \eqref{definition-H-delta} is a KL function. Then, $\{\bm{x}^k\}$ converges to $\bm{x}^*$, which is a stationary point of problem \eqref{ge-DC-problem}.
\end{theorem}
\begin{proof}
In view of Theorem \ref{nexPGA-theorem-subsequence}, we only need to show that $\{\bm{x}^k\}$ is convergent.

We first see from Lemma \ref{nexPGA-lem-dist-H-delta} that there exist $\tilde{c}>0$, $K_1>0$ and $\alpha>0$ such that 
\begin{equation}\label{dist-H-delta}
\operatorname{dist}\left(\bm{0}, \partial H_{\delta}(\bm{x}^{k}, \bm{x}^{k-1}, \overline{\gamma}_{k-1})\right)
\leq \tilde{c}\left(\|\bm{x}^{k}-\bm{x}^{k-1}\|+\|\bm{x}^{k-1}-\bm{x}^{k-2}\|\right)
\end{equation}
for all $k\in\{j\in\mathbb{N}:\bm{x}^{j+1}\in\mathcal{B}_{\alpha}(\bm{x}^*),\,j\geq K_1\}$ with $\mathcal{B}_{\alpha}(\bm{x}^*)=\left\{\bm{x}\in\mathbb{R}^n:\|\bm{x}-\bm{x}^*\|\leq\alpha\right\}$. For such $\alpha$ and $k$, Proposition \ref{pro-gammak-bounded} guarantees the existence of a positive constant $\overline{\gamma}_{\alpha,\bm{x}^*}>0$ (depending on $\alpha$ and $\bm{x}^*$) such that $\overline{\gamma}_k\leq\overline{\gamma}_{\alpha,\bm{x}^*}$. Next, let $\Upsilon:=\{(\bm{x}^*, \bm{x}^*, \overline{\gamma}): {\gamma}_{\min } \leq \overline{\gamma} \leq \overline{\gamma}_{\alpha,\bm{x}^*}\}$, which is a compact subset of $\operatorname{dom} \partial H_\delta$. Then, for any $\overline{\gamma}$ satisfying ${\gamma}_{\min } \leq \overline{\gamma} \leq \overline{\gamma}_{\alpha,\bm{x}^*}$, it follows from Theorem \ref{nexPGA-theorem-subsequence} that $H_\delta(\bm{x}^*,\bm{x}^*,\overline{\gamma}) = F(\bm{x}^*)={\zeta}$. Thus, we can conclude that $H_{\delta}\equiv{\zeta}$ on $\Upsilon$. This fact, together with the assumption that $H_{\delta}$ is a KL function, and the uniformized KL property (Proposition \ref{uniKL}), implies that there exist $\varepsilon>0$, $\nu>0$, and $\varphi\in\Phi_{\nu}$ such that
\begin{equation}\label{phi-dist}
\varphi'(H_{\delta}(\bm{u},\bm{v},\gamma)-{\zeta})\cdot\mathrm{dist}(\bm{0},\,\partial H_{\delta}(\bm{u},\bm{v},\gamma))\geq1,
\end{equation}
for all $(\bm{u},\bm{v},\gamma)$ satisfying $\mathrm{dist}((\bm{u},\bm{v},\gamma),\,\Upsilon)<\varepsilon$ and ${\zeta}<H_{\delta}(\bm{u},\bm{v},\gamma)<{\zeta}+\nu$. Moreover, since $\{\mathcal{R}_k\}$ is non-increasing (by Proposition \ref{pro-basic-property} (ii)), there exists an integer $K_2$ such that $\zeta<\mathcal{R}_k<\zeta+\nu$ holds for all $k\geq K_2$\footnote{In case that there exists some $\bar{k}$ such that $\mathcal{R}_{\bar{k}} = \zeta$, we see from the monotonicity of $\{\mathcal{R}_k\}$ that $\mathcal{R}_{k} = \zeta$ holds for all $k\ge\bar{k}$. Then one can see from \eqref{Rk-descent} that $\bm{x}^k=\bm{x}^{\bar{k}}$ holds for all $k\ge\bar{k}$. That said, the convergence of $\{\bm{x}^k\}$ is shown in this trivial case.}.

In the following, for notational simplicity, we define 
\begin{equation}\label{defNotion}
\begin{aligned}
M&:=\textstyle\left\lceil\frac{2\big(1+\sqrt{1-p_{\min}}\big)}{1-\sqrt{1-p_{\min}}}\right\rceil^2, \quad \ell(k):=k+M-1, \quad \Xi_k:=\sqrt{\mathcal{R}_k-\mathcal{R}_{k+1}}, \vspace{-1mm} \\
\Delta_{i,j}^{\varphi}&:=\varphi(\mathcal{R}_i-\zeta)-\varphi(\mathcal{R}_j-\zeta),
\quad \pi:=\sqrt{\textstyle\frac{(1-\delta)p_{\min}\gamma_{\min}}{8}},
\end{aligned}
\end{equation}
where $\lceil a \rceil$ is the smallest integer greater than or equal to $a$. It follows from \eqref{Rk-descent} that 
\begin{equation}\label{Xik-xkdiff}
\textstyle\|\bm{x}^{k+1}-\bm{x}^k\|\leq\frac{\Xi_k}{\pi}.
\end{equation}
In addition, let $\{\bm{x}^{k_j}\}_{j\in\mathbb{N}}$ be a subsequence converging to $\bm{x}^*$. With these preparations, we proceed to prove the convergence of $\{\bm{x}^k\}$ and divide the proof into four steps.

\vspace{1mm}
\textit{Step 1.} We claim that, for the above positive constants $\alpha,\varepsilon>0$, there exists an index $J\in\mathbb N$ such that the following inequality holds:
\begin{equation}\label{nexPGA-sum-rj}
\textstyle\widehat{Q}
:=\sup_{j\geq J}\left\{\mathcal{Q}_j:=\|\bm{x}^{k_j}-\bm{x}^*\|
+ \frac{4}{\pi}\sum_{i=k_j-2}^{\ell(k_j)}\Xi_i
+ \frac{\tilde{c}}{2\pi^2}\sum_{i=k_j}^{\ell(k_j)} \varphi\big(\mathcal{R}_i-\zeta\big)\right\}
\leq \min\left\{\alpha,\,\frac{\varepsilon}{2}\right\}.
\end{equation}
Recall from \eqref{defNotion} that $\ell(k)-k=M-1$, which is a fixed constant. Thus, the number of terms in each summation within $\mathcal{Q}_j$ is fixed and independent of $k$. Moreover, since $\{\mathcal{R}_k\}$ converges monotonically to $\zeta$ (by Proposition \ref{pro-basic-property}(ii)) and $\varphi$ is continuous on $[0,\nu)$ with $\varphi(0)=0$ (by the properties required on the desingularization function $\varphi$ in the KL property), it follows that $\sum_{i=k_j-2}^{\ell(k_j)}\Xi_i\to0$ and $\sum_{i=k_j}^{\ell(k_j)} \varphi\big(\mathcal{R}_i-\zeta\big)\to0$ as $j\to\infty$. These, together with $\bm{x}^{k_j} \to \bm{x}^*$ as $j\to\infty$ imply $\mathcal{Q}_j\to0$, and therefore there exists an index $J$ such that \eqref{nexPGA-sum-rj} holds.

\vspace{1mm}
\textit{Step 2.} We show that
\begin{equation}\label{ineq-sum-Ei}
\textstyle\frac{1-\sqrt{1-p_{\min }}}{\sqrt{M}}\sum_{i=k}^{\ell(k)}\Xi_i
\leq \left(\frac{1}{2}+\sqrt{1-p_{\min}}\right)
\big(\Xi_{k-2}+\Xi_{k-1}\big)
+ \frac{\tilde{c}}{2\pi}\Delta^{\varphi}_{k,k+M}
\end{equation}
holds for all $k\in\big\{j\in\mathbb{N}\,:\,\bm{x}^{j-1},\,\bm{x}^j\in\mathcal{B}_{\widehat{Q}}(\bm{x}^*),\,j\geq\max\{K_1,K_2\}\big\}$ with $\mathcal{B}_{\widehat{Q}}(\bm{x}^*):=\big\{\bm{x}\in\mathbb{R}^n:\|\bm{x}-\bm{x}^*\|\leq \widehat{Q}\big\}$.

To prove this, consider an arbitrary index $k$ from the above index set. For such $k\geq\max\{K_1,K_2\}$, we have that $\|\bm{x}^{k-1}-\bm{x}^*\|\leq\min\{\alpha,\frac{\varepsilon}{2}\}$, $\|\bm{x}^k-\bm{x}^*\|\leq\min\{\alpha,\frac{\varepsilon}{2}\}$, $\zeta<\mathcal{R}_k<\zeta+\nu$, $\overline{\gamma}_{k-1}\leq\overline{\gamma}_{\alpha,\bm{x}^*}$, and the inequality \eqref{dist-H-delta} holds. Moreover, by Jensen's inequality, we have that 
\begin{equation*}
\begin{aligned}
{\textstyle\frac{1}{\sqrt{M}}\sqrt{\mathcal{R}_k-\mathcal{R}_{k+M}}}
&= {\textstyle\sqrt{\frac{1}{M}\left(\mathcal{R}_k-\mathcal{R}_{k+1}+\cdots
+\mathcal{R}_{k+M-1}-\mathcal{R}_{k+M}\right)}}  \\
&\geq {\textstyle\frac{1}{M}\left(\sqrt{\mathcal{R}_k-\mathcal{R}_{k+1}}+\cdots
+\sqrt{\mathcal{R}_{k+M-1}-\mathcal{R}_{k+M}}\right)}.
\end{aligned}
\end{equation*}
This, together with $\ell(k)=k+M-1$, $\Xi_k = \sqrt{\mathcal{R}_k - \mathcal{R}_{k+1}}$ and $1-\sqrt{1-p_{\min }}>0$ (due to $p_{\min}\in(0,1)$), yields
\begin{equation}\label{jensen-ineq}
\textstyle\frac{1-\sqrt{1-p_{\min }}}{\sqrt{M}}\sum_{i=k}^{\ell(k)} \Xi_i
\leq \left(1-\sqrt{1-p_{\min}}\right) \sqrt{\mathcal{R}_k-\mathcal{R}_{k+M}}.
\end{equation}
Next, we prove \eqref{ineq-sum-Ei} by estimating the right-hand side of \eqref{jensen-ineq} in two cases.

\vspace{1mm}
\textbf{Case 1:} $H_{\delta}(\bm{x}^{k}, \bm{x}^{k-1}, \overline{\gamma}_{k-1}) \leq \mathcal{R}_{k+M}$. In this case, by the updating rule of $\mathcal{R}_k$,
\begin{equation*}
\begin{aligned}
& \mathcal{R}_k-\mathcal{R}_{k+M}  =(1-p_k) \mathcal{R}_{k-1}+p_k H_{\delta}(\bm{x}^{k}, \bm{x}^{k-1}, \overline{\gamma}_{k-1})-\mathcal{R}_{k+M} \\
\leq &  (1-p_k) \mathcal{R}_{k-1}+p_k \mathcal{R}_{k+M}-\mathcal{R}_{k+M}=(1-p_k)(\mathcal{R}_{k-1}-\mathcal{R}_{k+M}) \\
\leq &  (1-p_{\min})(\mathcal{R}_{k-1}-\mathcal{R}_{k+M}) \qquad(\text{by} ~p_k\in[p_{\min},1]~\text{and}~\mathcal{R}_{k-1}\geq\mathcal{R}_{k+M})\\
= &  (1-p_{\min})(\mathcal{R}_{k-1}-\mathcal{R}_k+\mathcal{R}_k-\mathcal{R}_{k+M}).
\end{aligned}
\end{equation*}
Taking square roots on both sides of this inequality and applying the inequality $\sqrt{a+b} \leq \sqrt{a}+\sqrt{b}$ for all $a,b\geq 0$, we obtain by rearranging the resulting terms 
\begin{equation*}
\big(1-\sqrt{1-p_{\min}}\big)\sqrt{\mathcal{R}_k-\mathcal{R}_{k+M}}
\leq \sqrt{1-p_{\min}}\,\Xi_{k-1}.
\end{equation*}
Substituting this into \eqref{jensen-ineq} yields the desired inequality \eqref{ineq-sum-Ei} for this case.

\vspace{1mm}
\textbf{Case 2:} $H_{\delta}(\bm{x}^{k}, \bm{x}^{k-1}, \overline{\gamma}_{k-1}) >\mathcal{R}_{k+M}$. In this case, it follows from the choice of $k$ and Proposition \ref{pro-basic-property}(i) that $\zeta<\mathcal{R}_{k+M}<H_{\delta}(\bm{x}^{k}, \bm{x}^{k-1}, \overline{\gamma}_{k-1})\leq\mathcal{R}_k<\zeta+\nu$ and $\mathrm{dist}((\bm{x}^k,\bm{x}^{k-1},\overline{\gamma}_{k-1}),\,\Upsilon)<\varepsilon$.
Thus, by \eqref{phi-dist}, we have that
\begin{equation}\label{KL-phi-dist}
\varphi^{\prime}\big(H_{\delta}(\bm{x}^{k}, \bm{x}^{k-1}, \overline{\gamma}_{k-1})-\zeta\big)
\cdot\operatorname{dist}\big(\bm{0}, \partial H_{\delta}(\bm{x}^{k}, \bm{x}^{k-1}, \overline{\gamma}_{k-1})\big) \geq 1.
\end{equation}
Moreover, we see that
\begin{equation*}
\hspace{-2mm}{\small
\begin{aligned}
&\,\quad\operatorname{dist}\big(\bm{0}, \partial H_{\delta}(\bm{x}^{k}, \bm{x}^{k-1}, \overline{\gamma}_{k-1})\big)\cdot\Delta^{\varphi}_{k,k+M} \\
&=\operatorname{dist}\big(\bm{0}, \partial H_{\delta}(\bm{x}^{k}, \bm{x}^{k-1}, \overline{\gamma}_{k-1})\big)
\cdot\big[\,\varphi(\mathcal{R}_k-\zeta)-\varphi(\mathcal{R}_{k+M}-\zeta)\,\big] \\
&\geq \operatorname{dist}\big(\bm{0}, \partial H_{\delta}(\bm{x}^{k}, \bm{x}^{k-1}, \overline{\gamma}_{k-1})\big)\cdot
\left[\varphi\big(H_{\delta}(\bm{x}^{k}, \bm{x}^{k-1}, \overline{\gamma}_{k-1})-\zeta\big)
-\varphi\big(\mathcal{R}_{k+M}-\zeta\big)\right]\\
&\geq \operatorname{dist}\big(\bm{0}, \partial H_{\delta}(\bm{x}^{k}, \bm{x}^{k-1}, \overline{\gamma}_{k-1})\big)\,\varphi^{\prime}\big(H_{\delta}(\bm{x}^{k}, \bm{x}^{k-1}, \overline{\gamma}_{k-1})-\zeta\big)\cdot\big(H_{\delta}(\bm{x}^{k}, \bm{x}^{k-1}, \overline{\gamma}_{k-1})-\mathcal{R}_{k+M}\big) \\
& \geq H_{\delta}(\bm{x}^{k}, \bm{x}^{k-1}, \overline{\gamma}_{k-1})-\mathcal{R}_{k+M},
\end{aligned}}
\end{equation*}
where the first inequality follows from the monotonicity of $\varphi$ and $H_{\delta}(\bm{x}^{k}, \bm{x}^{k-1}, \overline{\gamma}_{k-1})\leq \mathcal{R}_k$ by Proposition \ref{pro-basic-property}(i); the second inequality follows from the concavity of $\varphi$; the third inequality follows from \eqref{KL-phi-dist} and the hypothesis $H_{\delta}(\bm{x}^{k}, \bm{x}^{k-1}, \overline{\gamma}_{k-1}) > \mathcal{R}_{k+M}$. This, together with \eqref{dist-H-delta} and \eqref{Xik-xkdiff}, yields 
\begin{equation}\label{H-R-Xi}
\textstyle H_{\delta}(\bm{x}^{k}, \bm{x}^{k-1}, \overline{\gamma}_{k-1})
- \mathcal{R}_{k+M}
\leq \frac{\tilde{c}}{\pi}\Delta^{\varphi}_{k,k+M}
\big(\Xi_{k-1} + \Xi_{k-2}\big).
\end{equation}
Now, using the updating rule of $\mathcal{R}_k$, we obtain that
\begin{equation*}
\begin{aligned}
\mathcal{R}_k-\mathcal{R}_{k+M}
&= (1-p_k) \mathcal{R}_{k-1} + p_k H_{\delta}(\bm{x}^{k}, \bm{x}^{k-1}, \overline{\gamma}_{k-1})
- \mathcal{R}_{k+M} \\
&= p_{k}\big(H_{\delta}(\bm{x}^{k}, \bm{x}^{k-1}, \overline{\gamma}_{k-1})-\mathcal{R}_{k+M}\big)
+ (1-p_{k})\big(\mathcal{R}_{k-1}-\mathcal{R}_{k+M}\big) \\
&\leq \big(H_{\delta}(\bm{x}^{k}, \bm{x}^{k-1}, \overline{\gamma}_{k-1})-\mathcal{R}_{k+M}\big)
+ (1-p_{\min})\big(\mathcal{R}_{k-1}-\mathcal{R}_{k+M}\big) \\
&\leq \textstyle\frac{\tilde{c}}{\pi}\Delta^{\varphi}_{k,k+M}
\big(\Xi_{k-1} + \Xi_{k-2}\big)
+ (1-p_{\min})\big(\mathcal{R}_{k-1}-\mathcal{R}_k
+ \mathcal{R}_k-\mathcal{R}_{k+M}\big),
\end{aligned}
\end{equation*}
where the first inequality follows from $p_k\in[p_{\min},1]$ and the second inequality follows from \eqref{H-R-Xi}. Taking square roots on both sides of the above inequality and applying the inequalities $\sqrt{a+b}\leq\sqrt{a}+\sqrt{b}$ and $\sqrt{ab}\leq\frac{a+b}{2}$ for $a, b \geq 0$, we see that 
\begin{equation*}
\begin{aligned}
\sqrt{\mathcal{R}_k-\mathcal{R}_{k+M}}  &\leq \sqrt{\textstyle{\frac{\tilde{c}}{\pi}}\Delta^{\varphi}_{k,k+M}
\big(\Xi_{k-1} + \Xi_{k-2}\big)}
+ \sqrt{1-p_{\min}}
\sqrt{\mathcal{R}_{k-1}-\mathcal{R}_k+\mathcal{R}_k-\mathcal{R}_{k+M}} \\
&\leq \textstyle{\frac{\tilde{c}}{2\pi}}\Delta^{\varphi}_{k,k+M}
+ \textstyle{\frac{1}{2}}\big(\Xi_{k-1} + \Xi_{k-2}\big)
+ \sqrt{1-p_{\min}}
\left(\sqrt{\mathcal{R}_{k-1}-\mathcal{R}_k}+\sqrt{\mathcal{R}_k-\mathcal{R}_{k+M}}\right),
\end{aligned} 
\end{equation*}
which implies that
\begin{equation*}
\textstyle\left(1-\sqrt{1-p_{\min}}\right)\sqrt{\mathcal{R}_k-\mathcal{R}_{k+M}}
\leq \frac{1}{2}\Xi_{k-2}
+ \left(\frac{1}{2}+\sqrt{1-p_{\min}}\right)\Xi_{k-1}
+ \frac{\tilde{c}}{2\pi}\Delta^{\varphi}_{k,k+M}.
\end{equation*}
Substituting this into \eqref{jensen-ineq} completes the proof of \eqref{ineq-sum-Ei} for this case.

\vspace{1mm}
\textit{Step 3.} Without loss of generality, we assume that $J$ is a sufficiently large index such that \eqref{nexPGA-sum-rj} holds and $k_{J}\geq\max\{K_1,K_2\}$. For such $k_J$, we claim that the following relations hold for all $k\geq \ell(k_{J})$:
\begin{eqnarray}
\bm{x}^k&\in&\mathcal{B}_{\widehat{Q}}(\bm{x}^*), \label{first-statement} \\
\textstyle\sum_{i=\ell(k_{J})}^{k}\Xi_i
&\leq& \big(1+2\sqrt{1-p_{\min}}\big)\textstyle\sum_{i=k_{J}-2}^{\ell(k_{J})-1}\Xi_i
+ \textstyle\frac{\tilde{c}}{2\pi}\sum_{i=k_{J}}^{\ell(k_{J})} \varphi\big(\mathcal{R}_i-\zeta\big). \label{second-statement}
\end{eqnarray}

We prove them by induction. First, for all $k \in\left\{k_{J}-1,k_{J}, k_{J}+1,\cdots,\ell(k_{J})\right\}$, it follows from \eqref{Xik-xkdiff} and \eqref{nexPGA-sum-rj} that
\begin{equation*}
\begin{aligned}
\|\bm{x}^k-\bm{x}^*\|
&\leq\|\bm{x}^{k_{J}}-\bm{x}^*\|
+\textstyle\sum_{i=k_{J}-1}^k\|\bm{x}^{i+1}-\bm{x}^{i}\|
\leq \|\bm{x}^{k_{J}}-\bm{x}^*\|
+ \sum_{i=k_{J}-1}^{\ell(k_{J})}\|\bm{x}^{i+1}-\bm{x}^i\| \\
&\textstyle\leq\|\bm{x}^{k_{J}}-\bm{x}^*\|
+ \frac{1}{\pi}\sum_{i=k_{J}-1}^{\ell(k_{J})}\Xi_i
\leq \widehat{Q}.
\end{aligned}
\end{equation*}
This proves that \eqref{first-statement} holds for all $k \in\left\{k_{J}-1,k_{J}, k_{J}+1,\cdots,\ell(k_{J})\right\}$, meaning that the iterates $\bm{x}^{k_J-1},\bm{x}^{k_J},\cdots,\bm{x}^{\ell(k_J)}$ lie within the neighborhood $\mathcal{B}_{\widehat{Q}}(\bm{x}^*)$. Using this fact and applying inequality \eqref{ineq-sum-Ei}, we have that
\begin{equation}\label{sqrtM-Xi-ellkJ}
\begin{aligned}
&\textstyle\big(1-\sqrt{1-p_{\min}}\big)\sqrt{M}\,\Xi_{\ell(k_{J})}
\leq \frac{1-\sqrt{1-p_{\min}}}{\sqrt{M}}\sum_{i=k_{J}}^{\ell(k_{J})} \sum_{t=i}^{\ell(i)}\,\Xi_t \\
&\textstyle\leq \left(\frac{1}{2}+\sqrt{1-p_{\min}}\right)\sum_{i=k_{J}}^{\ell(k_{J})} \big(\Xi_{i-2}+\Xi_{i-1}\big)
+ \frac{\tilde{c}}{2\pi}\sum_{i=k_{J}}^{\ell(k_{J})}
\Delta^{\varphi}_{i,i+M} \\
&\textstyle\leq \big(1+2\sqrt{1-p_{\min}}\big)\sum_{i=k_{J}-2}^{\ell(k_{J})}\Xi_i
+ \frac{\tilde{c}}{2\pi}\sum_{i=k_{J}}^{\ell(k_{J})}\varphi(\mathcal{R}_i-\zeta)\\
&\textstyle=\big(1+2\sqrt{1-p_{\min}}\big)\left(\sum_{i=k_{J}-2}^{\ell(k_{J})-1}\Xi_i+\Xi_{\ell(k_{J})}\right)
+ \frac{\tilde{c}}{2\pi}\sum_{i=k_{J}}^{\ell(k_{J})}\varphi(\mathcal{R}_i-\zeta),
\end{aligned}
\end{equation}
where the first inequality follows from the nonnegativity of $\Xi_k$ and the fact that the term $\Xi_{\ell(k_{J})}$ appears $M$ times in the double sum; the second inequality is obtained by applying \eqref{ineq-sum-Ei} to each $k\in\{k_J,k_J+1,\cdots,\ell(k_J)\}$; the third inequality follows from $\Delta_{i,i+M}^{\varphi}:=\varphi(\mathcal{R}_i-\zeta)-\varphi(\mathcal{R}_{i+M}-\zeta)
\leq\varphi(\mathcal{R}_i-\zeta)$ for any $i$. Recall from the definition of $M$ in \eqref{defNotion} that $\big(1-\sqrt{1-p_{\min}}\big)\sqrt{M}\geq 2\big(1+\sqrt{1-p_{\min}}\big)$, which further implies
\begin{equation}\label{M-relation-geq1}
\big(1-\sqrt{1-p_{\min}}\big)\sqrt{M}
- \big(1+2\sqrt{1-p_{\min}}\big)
\geq 1.
\end{equation}
This, together with \eqref{sqrtM-Xi-ellkJ}, yields that
\begin{equation*}
\textstyle\Xi_{\ell(k_J)}
\leq \big(1+2\sqrt{1-p_{\min}}\big)\sum_{i=k_{J}-2}^{\ell(k_{J})-1}\Xi_i
+ \frac{\tilde{c}}{2\pi}\sum_{i=k_{J}}^{\ell(k_{J})} \varphi\big(\mathcal{R}_i-\zeta\big).
\end{equation*}
This shows that \eqref{second-statement} holds for $k=\ell(k_{J})$.

Next, suppose that \eqref{first-statement} and \eqref{second-statement} hold for all $k$ from $\ell(k_{J})$ to some $K\geq\ell(k_{J})$. It remains to show that they also hold for $k=K+1$. Indeed, since \eqref{second-statement} holds for $k=K$ and $1+2\sqrt{1-p_{\min }}\leq3$ (due to $p_{\min}\in(0,1)$), we have that 
\begin{equation*}
{\textstyle\sum_{i=\ell(k_{J})}^K\Xi_i
\leq 3\sum_{i=k_{J}-2}^{\ell(k_{J})-1}\Xi_i
+ \frac{\tilde{c}}{2\pi}\sum_{i=k_{J}}^{\ell(k_{J})} \varphi\big(\mathcal{R}_i-\zeta\big)},
\end{equation*}
which further implies that
\begin{equation*}
{\textstyle\sum_{i=k_{J}}^K\Xi_i\leq \sum_{i=k_{J}-2}^{\ell(k_{J})-1}\Xi_i
+ \sum_{i=\ell(k_{J})}^K\Xi_i
\leq 4 \sum_{i=k_{J}-2}^{\ell(k_{J})-1}\Xi_i
+ \frac{\tilde{c}}{2\pi}\sum_{i=k_{J}}^{\ell(k_{J})} \varphi\big(\mathcal{R}_i-\zeta\big)}.
\end{equation*}
Using this relation, together with \eqref{Xik-xkdiff}, we see that
\begin{equation*}
\begin{aligned}
\|\bm{x}^{K+1}-\bm{x}^*\|
&\leq \|\bm{x}^{k_{J}}-\bm{x}^*\|
+ \textstyle\sum_{i=k_{J}}^K\|\bm{x}^{i+1}-\bm{x}^i\|
\leq \|\bm{x}^{k_{J}}-\bm{x}^*\|
+ \frac{1}{\pi}\sum_{i=k_{J}}^K\Xi_i \\
&\leq \|\bm{x}^{k_{J}}-\bm{x}^*\|
+ \textstyle{\frac{4}{\pi}}\sum_{i=k_{J}-2}^{\ell(k_{J})-1}\Xi_i
+ \frac{\tilde{c}}{2\pi^2}\sum_{i=k_{J}}^{\ell(k_{J})} \varphi\big(\mathcal{R}_i-\zeta\big).
\end{aligned}
\end{equation*}
This, along with \eqref{nexPGA-sum-rj}, implies that \eqref{first-statement} holds for $k=K+1$.

We now verify that \eqref{second-statement} holds for $k=K+1$. From the above discussion, we see that the iterates $\bm{x}^{k_J-1},\bm{x}^{k_J},\cdots,\bm{x}^{K+1}$ lie within the neighborhood $\mathcal{B}_{\widehat{Q}}(\bm{x}^*)$. Using this fact and applying inequality \eqref{ineq-sum-Ei}, we have that
\begin{equation}\label{ineqadd}
\begin{aligned}
&\textstyle\quad \big(1-\sqrt{1-p_{\min}})\sqrt{M} \sum_{i=\ell(k_{J})}^{K+1}\Xi_i
\leq \frac{1-\sqrt{1-p_{\min}}}{\sqrt{M}}
\sum_{i=k_{J}}^{K+1}\sum_{t=i}^{\ell(i)}\Xi_t  \\
&\textstyle \leq \left(\frac{1}{2}+\sqrt{1-p_{\min}}\right) \sum_{i=k_{J}}^{K+1}\big(\Xi_{i-2}+\Xi_{i-1}\big)
+ \frac{\tilde{c}}{2\pi}\sum_{i=k_{J}}^{K+1}
\Delta^{\varphi}_{i,i+M}  \\
&\textstyle\leq \big(1+2\sqrt{1-p_{\min}}\big) \sum_{i=k_{J}-2}^{K+1}\Xi_i
+ \frac{\tilde{c}}{2\pi}\sum_{i=k_{J}}^{\ell(k_{J})} \varphi\big(\mathcal{R}_i-\zeta\big)\\
&\textstyle=\big(1+2\sqrt{1-p_{\min}}\big) \left(\sum_{i=k_{J}-2}^{\ell(k_{J})-1}\Xi_i+\sum_{i=\ell(k_{J})}^{K+1}\Xi_i\right)
+ \frac{\tilde{c}}{2\pi}\sum_{i=k_{J}}^{\ell(k_{J})} \varphi\big(\mathcal{R}_i-\zeta\big),
\end{aligned}
\end{equation}
where the first inequality follows from the nonnegativity of $\Xi_k$ and the fact that the term $\Xi_i$ with $i\in\{\ell(k_{J}), \cdots, K+1\}$ appears $M$ times in the double sum; the second inequality is obtained by applying \eqref{ineq-sum-Ei} to each $k\in\{k_J,k_J+1,\cdots,K+1\}$; the third inequality follows because 
\begin{equation*}
{\small
\begin{aligned}
\textstyle\sum_{i=k_{J}}^{K+1}\Delta^{\varphi}_{i,i+M}
&=\textstyle\sum_{i=k_{J}}^{K+1}\varphi(\mathcal{R}_i-\zeta)-\sum_{i=k_{J}}^{K+1}\varphi(\mathcal{R}_{i+M}-\zeta)
=\textstyle\sum_{i=k_{J}}^{K+1}\varphi(\mathcal{R}_i-\zeta)-\sum_{i=k_{J}+M}^{K+M+1}\varphi(\mathcal{R}_i-\zeta) \\
&=\textstyle{\sum_{i=k_{J}}^{{k_{J}}+M-1}}\varphi(\mathcal{R}_i-\zeta)-\sum_{i=K+2}^{K+M+1}\varphi(\mathcal{R}_i-\zeta)
\leq\sum_{i=k_{J}}^{\ell({k_{J}})}\varphi(\mathcal{R}_i-\zeta).
\end{aligned}}
\end{equation*}
Using \eqref{ineqadd}, together with \eqref{M-relation-geq1}, we obtain that 
\begin{equation*}
\textstyle\sum_{i=\ell(k_{J})}^{K+1}\Xi_i
\leq \big(1+2\sqrt{1-p_{\min}}\big)\sum_{i=k_{J}-2}^{\ell(k_{J})-1}\Xi_i
+ \frac{\tilde{c}}{2\pi}\sum_{i=k_{J}}^{\ell(k_{J})} \varphi\big(\mathcal{R}_i-\zeta\big),
\end{equation*}
which shows that \eqref{second-statement} holds for $k=K+1$ and completes the induction.

\vspace{1mm}
\textit{Step 4.} Since \eqref{second-statement} holds for all $k\geq \ell(k_{J})$, taking the limit in \eqref{second-statement} yields 
\begin{equation*}
\textstyle\sum_{i=\ell(k_{J})}^{\infty}\Xi_i
\leq \big(1+2\sqrt{1-p_{\min}}\big)\sum_{i=k_{J}-2}^{\ell(k_{J})-1}\Xi_i
+ \frac{\tilde{c}}{2\pi}\sum_{i=k_{J}}^{\ell(k_{J})} \varphi\big(\mathcal{R}_i-\zeta\big)<\infty.  
\end{equation*}
This, together with \eqref{Xik-xkdiff}, yields 
\begin{equation*}
{\textstyle\sum_{i=\ell(k_{J})}^{\infty}\|\bm{x}^{i+1}-\bm{x}^i\|
\leq\frac{1}{\pi}\sum_{i=\ell(k_{J})}^{\infty}\Xi_i
<\infty}, 
\end{equation*}
which implies $\sum_{i=0}^{\infty}\|\bm{x}^{i+1}-\bm{x}^{i}\|<\infty$ and hence $\{\bm{x}^{k}\}$ is convergent.
\end{proof}

Next, based on the KL exponent, we establish the local convergence rates of the generated sequence $\{\bm{x}^k\}$ and its corresponding objective function values. After establishing the convergence of the entire sequence in Theorem \ref{nexPGA-theorem-wholesequence}, we know that if $\{\bm{x}^k\}$ has an accumulation point $\bm{x}^*$, then $\bm{x}^*$ is the unique accumulation point (i.e., the limit point) of $\{\bm{x}^k\}$ and hence \eqref{dist-H-delta-k} holds for all sufficiently large $k$. With this result in hand, we can apply arguments similar to those in \cite[Section 3.2]{qtpq2024convergence}, with suitable adaptations and refinements, to derive the desired rate estimates.

\begin{theorem}\label{theorem-fun-rate}
Suppose that Assumptions \ref{assum-para}, \ref{assum-funs1} and \ref{assumC} hold.  Let $\{\bm{x}^k\}$ be the sequence generated by the nexPGA in Algorithm \ref{algo-nexPGA}, and let $\zeta$ be given in Proposition \ref{pro-basic-property}(ii). Moreover, suppose that the sequence $\{\bm{x}^k\}$ has an accumulation point $\bm{x}^*$ and that the potential function $H_{\delta}$ in \eqref{definition-H-delta} is a KL function with an exponent $\theta \in[0,1)$. Then, the following statements hold for all sufficiently large $k$. \vspace{0.5mm}
\begin{itemize}
\item[{\rm(i)}] If $\theta=0$, there exist $c_1>0$ and $\eta_1\in(0,1)$ such that $\zeta-c_1\eta_1^k\leq F(\bm{x}^k)\leq\zeta$.

\vspace{0.5mm}
\item[{\rm(ii)}] If $\theta \in(0,\frac{1}{2}]$, there exist $c_2>0$ and $\eta_2 \in(0,1)$ such that $|F(\bm{x}^k)-\zeta|\leq c_2 \eta_2^k$.

\vspace{0.5mm}
\item[{\rm(iii)}] If $\theta \in(\frac{1}{2}, 1)$, there exists $c_3>0$ such that $|F(\bm{x}^k)-\zeta|\leq c_3k^{-\frac{1}{2\theta-1}}$.
\end{itemize}
\end{theorem}
\begin{proof}
First, since $\bm{x}^*$ is an accumulation point of $\{\bm{x}^k\}$, it follows from Theorem \ref{nexPGA-theorem-wholesequence} that the whole sequence $\{\bm{x}^k\}$ converges to $\bm{x}^*$. This ensures the existence of a positive constant $\vartheta>0$ such that $\{\bm{x}^k\}\subseteq\mathcal{B}_{\vartheta}(\bm{x}^*):=\left\{\bm{x}\in\mathbb{R}^n:\|\bm{x}-\bm{x}^*\|\leq\vartheta\right\}$. For such $\vartheta$, Proposition \ref{pro-gammak-bounded} further guarantees the existence of a positive constant $\overline{\gamma}_{\vartheta,\bm{x}^*}>0$ (depending  on $\vartheta$ and $\bm{x}^*$) such that $\overline{\gamma}_k\leq\overline{\gamma}_{\vartheta,\bm{x}^*}$ holds for all $k\geq0$. Let $\Delta^k_\mathcal{R}:=\mathcal{R}_k-\zeta$ for all $k\geq0$. It follows from Proposition \ref{pro-basic-property}(ii) that $\{\Delta^k_\mathcal{R}\}$ is non-increasing and $\Delta^k_\mathcal{R}\geq0$. Moreover, we see from \eqref{Rk-descent} and $H_{\delta}(\bm{x}^{k+1}, \bm{x}^k, \overline{\gamma}_{k})\leq\mathcal{R}_{k+1}$ (by Proposition \ref{pro-basic-property}(i)) that
\begin{equation}\label{Rkdiff-xkdiff-square}
d_1\|\bm{x}^{k+1}-\bm{x}^k\|^2
\leq \mathcal{R}_k - \mathcal{R}_{k+1}
= \Delta^k_\mathcal{R} - \Delta^{k+1}_\mathcal{R}
\leq \mathcal{R}_k-H_{\delta}(\bm{x}^{k+1}, \bm{x}^k, \overline{\gamma}_{k})
\end{equation}
for any $k\geq0$, where $d_1:=\frac{(1-\delta)p_{\min}\gamma_{\min}}{8}>0$. Then, for any $k\geq1$, we have that
\begin{equation}\label{Fxk-Rk-1}
\begin{aligned}
|F(\bm{x}^k)-\zeta|
&=\left|H_\delta\big(\bm{x}^k, \bm{x}^{k-1}, \overline{\gamma}_{k-1}\big)-{\textstyle\frac{\delta\overline{\gamma}_{k-1}}{8}}\|\bm{x}^k-\bm{x}^{k-1}\|^2-\zeta\right| \\
&\leq\left|H_\delta\big(\bm{x}^k, \bm{x}^{k-1}, \overline{\gamma}_{k-1}\big)-\zeta\right|
+ {\textstyle\frac{\delta\overline{\gamma}_{k-1}}{8}}\|\bm{x}^k-\bm{x}^{k-1}\|^2  \\
&=\left|\mathcal{R}_{k-1}+{\textstyle\frac{1}{p_k}}(\mathcal{R}_{k}-\mathcal{R}_{k-1})-\zeta\right|
+ {\textstyle\frac{\delta\overline{\gamma}_{k-1}}{8}}\|\bm{x}^k-\bm{x}^{k-1}\|^2 \\
&\leq \Delta^{k-1}_\mathcal{R}
+ {\textstyle\frac{1}{p_k}}(\mathcal{R}_{k-1}-\mathcal{R}_{k})
+ {\textstyle\frac{\delta\overline{\gamma}_{k-1}}{8}}\|\bm{x}^k-\bm{x}^{k-1}\|^2 \\
&\leq \Delta^{k-1}_\mathcal{R}
+ {\textstyle\frac{1}{p_k}}(\Delta^{k-1}_\mathcal{R}-\Delta^{k}_\mathcal{R})
+ {\textstyle\frac{\delta\overline{\gamma}_{k-1}}{8d_1}}(\Delta^{k-1}_\mathcal{R}-\Delta^{k}_\mathcal{R}) \\
&\leq {\textstyle\left(1+\frac{1}{p_{\min}}
+\frac{\delta\overline{\gamma}_{\vartheta,\bm{x}^*}}{8d_1}\right)}
\Delta_\mathcal{R}^{k-1}
= d_2\Delta_\mathcal{R}^{k-1},
\end{aligned}
\end{equation}
where $d_2:=1+\frac{1}{p_{\min}}+\frac{\delta \overline{\gamma}_{\vartheta,\bm{x}^*}}{8d_1}$, the first equality follows from the definition of $H_{\delta}$ in \eqref{definition-H-delta}, the second equality follows from $H_{\delta}(\bm{x}^{k}, \bm{x}^{k-1}, \overline{\gamma}_{k-1})=\mathcal{R}_{k-1}+\frac{1}{p_k}(\mathcal{R}_{k}-\mathcal{R}_{k-1})$ by the updating rule of $\mathcal{R}_k$, the third inequality follows from \eqref{Rkdiff-xkdiff-square} and the last inequality follows from $\overline{\gamma}_k\leq\overline{\gamma}_{\vartheta,\bm{x}^*}$, $\Delta_\mathcal{R}^{k}\geq0$ and $p_k \geq p_{\min}>0$ for all $k\geq0$.

With \eqref{Fxk-Rk-1} at hand, we can characterize the convergence rate of $\{|F(\bm{x}^k)-\zeta|\}$ by analyzing the rate of $\{\Delta_\mathcal{R}^{k}\}$. To this end, we first consider the case where $\Delta_\mathcal{R}^{K_0}=0$ for some $K_0\geq0$. Since $\{\Delta_\mathcal{R}^{k}\}$ is non-increasing, it follows that $\Delta_\mathcal{R}^{k}=0$ for all $k \geq K_0$. This, together with \eqref{Fxk-Rk-1}, immediately proves all statements. From now on, we consider the case where $\Delta_\mathcal{R}^{k}>0$ for all $k\geq0$.

In view of the boundedness of $\{\overline{\gamma}_k\}$ (since $\gamma_{\min}\leq\overline{\gamma}_k\leq\overline{\gamma}_{\vartheta,\bm{x}^*}$ for all $k\geq0$) and the fact that the whole sequence $\{\bm{x}^k\}_{k=0}^{\infty}$ converges to $\bm{x}^*$, we see that the set of cluster points of $\{(\bm{x}^k, \bm{x}^{k-1}, \overline{\gamma}_{k-1})\}_{k=0}^{\infty}$ is contained in
\begin{equation*}
\Upsilon := \left\{(\bm{x}^*, \bm{x}^*, \overline{\gamma}):
{\gamma}_{\min} \leq \overline{\gamma} \leq \overline{\gamma}_{\vartheta,\bm{x}^*}\right\},
\end{equation*}
which is a compact subset of $\operatorname{dom} \partial H_\delta$. Moreover, for any $\overline{\gamma}$ satisfying ${\gamma}_{\min } \leq \overline{\gamma} \leq \overline{\gamma}_{\vartheta,\bm{x}^*}$, it follows from Theorem \ref{nexPGA-theorem-subsequence} that $H_\delta(\bm{x}^*,\bm{x}^*,\overline{\gamma}) = F(\bm{x}^*)={\zeta}$. Thus, we can conclude that $H_{\delta}\equiv{\zeta}$ on $\Upsilon$. This fact, together with the assumption that $H_{\delta}$ is a KL function with an exponent $\theta$ and the uniformized KL property (Proposition~\ref{uniKL}), implies that there exist $\varepsilon>0$, $\nu>0$, and $\varphi\in\Phi_{\nu}$ such that
\begin{equation}\label{phi-dist-2}
\varphi'(H_{\delta}(\bm{u},\bm{v},\gamma)-{\zeta})
\cdot\mathrm{dist}(\bm{0},\,\partial H_{\delta}(\bm{u},\bm{v},\gamma))\geq1,
~~\text{with}~~\varphi(s)=\tilde{a}s^{1-\theta}~\text{for some}~\tilde{a}>0,
\end{equation}
for all $(\bm{u},\bm{v},\gamma)$ satisfying $\mathrm{dist}((\bm{u},\bm{v},\gamma),\,\Upsilon)<\varepsilon$ and ${\zeta}<H_{\delta}(\bm{u},\bm{v},\gamma)<{\zeta}+\nu$. Next, we recall from Lemma \ref{nexPGA-lem-dist-H-delta} that there exist $\tilde{c}>0$, $K_1>0$ and $\alpha>0$ such that
\begin{equation}\label{dist-H-delta-2}
\operatorname{dist}\left(\bm{0}, \partial H_{\delta}(\bm{x}^{k}, \bm{x}^{k-1}, \overline{\gamma}_{k-1})\right)
\leq \tilde{c}\left(\|\bm{x}^{k}-\bm{x}^{k-1}\|+\|\bm{x}^{k-1}-\bm{x}^{k-2}\|\right)
\end{equation}
holds for all $k\in\{j\in\mathbb{N}:\bm{x}^{j+1}\in\mathcal{B}_{\alpha}(\bm{x}^*),j\geq K_1\}$ with $\mathcal{B}_{\alpha}(\bm{x}^*)=\left\{\bm{x}\in\mathbb{R}^n:\|\bm{x}-\bm{x}^*\|\leq\alpha\right\}$. Since $\Upsilon$ contains all the cluster points of $\{(\bm{x}^k, \bm{x}^{k-1}, \overline{\gamma}_{k-1})\}_{k=0}^{\infty}$, we have
\begin{equation*}
\lim_{k \rightarrow \infty} \operatorname{dist}\big((\bm{x}^k, \bm{x}^{k-1}, \overline{\gamma}_{k-1}), \Upsilon\big)=0.
\end{equation*}
This, together with the fact that the sequence $\{\mathcal{R}_k\}$ converges monotonically to $\zeta$ (by Proposition \ref{pro-basic-property}(ii)) and the fact that the whole sequence $\{\bm{x}^k\}$ converges to $\bm{x}^*$, implies that there exists an integer $K_2$ such that $\zeta<\mathcal{R}_k<\zeta+\min\left\{\nu,\,\frac{d_1}{{2}\tilde{a}^2\tilde{c}^2},\,1\right\}$, $\operatorname{dist}((\bm{x}^k, \bm{x}^{k-1}, \overline{\gamma}_{k-1}), \Upsilon)<\varepsilon$, and $\|\bm{x}^k-\bm{x}^*\|\leq\alpha$ whenever $k\geq K_2$. In the following, for notational simplicity, let $\Delta^{k}_{H_{\delta}}:=H_{\delta}(\bm{x}^k, \bm{x}^{k-1}, \overline{\gamma}_{k-1})-\zeta$ and $\overline{K}_1=\max\{K_1,K_2\}+2$. Then, we see that $\Delta^{k}_{H_{\delta}}\leq\Delta^k_{\mathcal{R}}<1$, \eqref{dist-H-delta-2}, and $\Delta_\mathcal{R}^{k-2}<\frac{d_1}{2\tilde{a}^2\tilde{c}^2}$ hold for  all $k\geq\overline{K}_1$. With these preparations, we next proceed to prove the desired results and divide the proof into three steps.

\vspace{1mm}
\textit{\textbf{Step 1}}. First, we claim that for any $k\geq\overline{K}_1$, the following statements hold:
\begin{itemize}
\item [(1a)] If $H_{\delta}(\bm{x}^{k}, \bm{x}^{k-1}, \overline{\gamma}_{k-1})\leq\zeta$, there exists $\rho_1\in(0,1)$ such that $\Delta_{\mathcal{R}}^{k}\leq\rho_1\Delta_{\mathcal{R}}^{k-2}$;
\item[(1b)] If $H_{\delta}(\bm{x}^{k}, \bm{x}^{k-1}, \overline{\gamma}_{k-1})>\zeta$, there exists $a_1>0$ such that $\big(\Delta_{H_{\delta}}^k\big)^{2\theta}\leq a_1\big(\Delta_\mathcal{R}^{k-2}-\Delta_\mathcal{R}^k\big)$.
\end{itemize}

We first consider $H_{\delta}(\bm{x}^{k}, \bm{x}^{k-1}, \overline{\gamma}_{k-1})\leq\zeta$. In this case, together with the updating rule of $\mathcal{R}_k$, $0<p_{\min}\leq p_k\leq1$ and the fact that $\{\Delta^k_{\mathcal{R}}\}_{k=0}^{\infty}$ is non-increasing, we have that
\begin{equation*}
\begin{aligned}
\Delta_{\mathcal{R}}^k
&=p_k\big(H_{\delta}(\bm{x}^{k}, \bm{x}^{k-1},\overline{\gamma}_{k-1})-\zeta\big)
+ (1-p_k)\big(\mathcal{R}_{k-1}-\zeta\big) \\
&\leq(1-p_k)\Delta_{\mathcal{R}}^{k-1} 
\leq(1-p_{\min})\Delta_{\mathcal{R}}^{k-1}
\leq\rho_1\Delta^{k-2}_{\mathcal{R}},
\end{aligned}
\end{equation*}
where $\rho_1:=1-p_{\min}\in(0,1)$. This shows that statement (1a) holds.

We next consider $H_{\delta}(\bm{x}^{k}, \bm{x}^{k-1}, \overline{\gamma}_{k-1})>\zeta$. In this case, it follows from Proposition \ref{pro-basic-property}(i) and $k\geq\overline{K}_1$ that $\zeta<H_{\delta}(\bm{x}^{k}, \bm{x}^{k-1}, \overline{\gamma}_{k-1})\leq\mathcal{R}_k<\zeta+\nu$. Thus, by \eqref{phi-dist-2}, we have
\begin{equation*}
\varphi^{\prime}\big(H_{\delta}(\bm{x}^{k}, \bm{x}^{k-1}, \overline{\gamma}_{k-1})-\zeta\big)
\cdot\operatorname{dist}\big(\bm{0}, \partial H_{\delta}(\bm{x}^{k}, \bm{x}^{k-1}, \overline{\gamma}_{k-1})\big) \geq 1.
\end{equation*}
Moreover, we see that
\begin{equation*}
\begin{aligned}
1&\leq\varphi^{\prime}\big(H_{\delta}(\bm{x}^{k}, \bm{x}^{k-1}, \overline{\gamma}_{k-1})-\zeta\big)
\cdot\operatorname{dist}\big(\bm{0}, \partial H_{\delta}(\bm{x}^{k}, \bm{x}^{k-1}, \overline{\gamma}_{k-1})\big)\\
&\leq \tilde{a}(1-\theta)\cdot(\Delta_{H_{\delta}}^k)^{-\theta}
\cdot\tilde{c}\big(\|\bm{x}^{k}-\bm{x}^{k-1}\|
+\|\bm{x}^{k-1}-\bm{x}^{k-2}\|\big) \\
&\leq \tilde{a}\tilde{c}(1-\theta)\cdot(\Delta_{H_{\delta}}^k)^{-\theta}
\cdot\sqrt{2\big(\|\bm{x}^k-\bm{x}^{k-1}\|^2
+\|\bm{x}^{k-1}-\bm{x}^{k-2}\|^2\big)} \\
&\leq \sqrt{2/d_1}\tilde{a}\tilde{c}(1-\theta)
\cdot(\Delta_{H_{\delta}}^k)^{-\theta}\cdot\sqrt{\mathcal{R}_{k-2}-\mathcal{R}_k}\\
&=\sqrt{2/d_1}\tilde{a}\tilde{c}(1-\theta)
\cdot(\Delta_{H_{\delta}}^k)^{-\theta}\cdot\sqrt{\Delta_\mathcal{R}^{k-2}-\Delta_\mathcal{R}^k},
\end{aligned}
\end{equation*}
where the second inequality follows from \eqref{dist-H-delta-2} and the last inequality follows from \eqref{Rkdiff-xkdiff-square}. This inequality further yields that
\begin{equation*}
\big(\Delta_{H_{\delta}}^k\big)^{2\theta}
\leq (2/d_1)\tilde{a}^2\tilde{c}^2(1-\theta)^2
\big(\Delta_\mathcal{R}^{k-2}-\Delta_\mathcal{R}^k\big)
= a_1\big(\Delta_\mathcal{R}^{k-2}-\Delta_\mathcal{R}^k\big),
\end{equation*}
where $a_1:=(2/d_1)\tilde{a}^2\tilde{c}^2(1-\theta)^2$. This shows that statement (1b) holds.

\vspace{1mm}
\textit{\textbf{Step 2}}. Now, we claim that for any $k\geq\overline{K}_1$, the following statements hold:
\begin{itemize}
\item [(2a)] If $\theta=0$, $F(\bm{x}^k)\leq\zeta$ and there exists $\rho_2\in(0,1)$ such that $\Delta_{\mathcal{R}}^{k}\leq\rho_2\Delta_{\mathcal{R}}^{k-2}$;
\item[(2b)] If $\theta\in(0,\frac{1}{2}]$, there exists $\rho_3\in(0,1)$ such that $\Delta_{\mathcal{R}}^{k}\leq\rho_3\Delta_{\mathcal{R}}^{k-2}$;
\item[(2c)] If $\theta\in(\frac{1}{2},1)$, there exists $a_2>0$ such that $(\Delta_{\mathcal{R}}^{k})^{1-2\theta}-(\Delta_{\mathcal{R}}^{k-2})^{1-2\theta}\geq a_2$.
\end{itemize}

Statement (2a). Suppose that $\theta=0$. We consider the following two cases. We first consider $H_{\delta}(\bm{x}^{k}, \bm{x}^{k-1}, \overline{\gamma}_{k-1})\leq\zeta$. In this case, for any $k\geq\overline{K}_1$, it follows from statement (1a) in \textit{\textbf{Step 1}} and the definition of $H_{\delta}$ in \eqref{definition-H-delta} that $F(\bm{x}^k)\leq\zeta$ and there exists $\rho_1\in(0,1)$ such that $\Delta_{\mathcal{R}}^{k}\leq\rho_1\Delta_{\mathcal{R}}^{k-2}$. We next consider $H_{\delta}(\bm{x}^{k}, \bm{x}^{k-1}, \overline{\gamma}_{k-1})>\zeta$. In this case, we see from statement (1b) with $\theta=0$ in \textit{\textbf{Step 1}} and $\Delta_{\mathcal{R}}^{k}\geq0$ that, for any $k\geq\overline{K}_1$,
\begin{equation*}
\textstyle\Delta_\mathcal{R}^{k-2}\geq\Delta_\mathcal{R}^{k-2}-\Delta_\mathcal{R}^k
\geq\frac{1}{a_1}=\frac{d_1}{2\tilde{a}^2\tilde{c}^2},
\end{equation*}
which contradicts to the fact that $\Delta_\mathcal{R}^{k-2}<\frac{d_1}{2\tilde{a}^2\tilde{c}^2}$ holds whenever $k\geq \overline{K}_1$. Thus, this case cannot happen. Combining with these two cases, we prove statement (2a).

Statement (2b). Suppose that $\theta\in(0,\frac{1}{2}]$. We consider two cases.

We first consider $H_{\delta}(\bm{x}^{k}, \bm{x}^{k-1}, \overline{\gamma}_{k-1})\leq\zeta$. In this case, for any $k\geq\overline{K}_1$, it follows from the statement (1a) in \textit{\textbf{Step 1}} that there exists $\rho_1\in(0,1)$ such that $\Delta_{\mathcal{R}}^{k}\leq\rho_1\Delta_{\mathcal{R}}^{k-2}$, which gives the desired result.

We next consider $H_{\delta}(\bm{x}^{k}, \bm{x}^{k-1}, \overline{\gamma}_{k-1})>\zeta$. In this case, for any $k\geq\overline{K}_1$, it follows from $0<\Delta_{H_{\delta}}^k\leq\Delta_\mathcal{R}^k<1$, $2\theta\in(0,1]$ and statement (1b) in \textit{\textbf{Step 1}} that
\begin{equation*}
\textstyle\Delta_{H_{\delta}}^k\leq(\Delta_{H_{\delta}}^k)^{2\theta} \leq a_1(\Delta_\mathcal{R}^{k-2}-\Delta_\mathcal{R}^k)\leq a_1(\Delta_\mathcal{R}^{k-2}-\Delta_{H_{\delta}}^k)
\quad\Longrightarrow\quad
\Delta_{H_{\delta}}^k\leq\frac{a_1}{1+a_1}\Delta_\mathcal{R}^{k-2}.
\end{equation*}
This, along with the updating rule of $\mathcal{R}_k$, $H_{\delta}(\bm{x}^{k}, \bm{x}^{k-1}, \overline{\gamma}_{k-1})\leq\mathcal{R}_k\leq\mathcal{R}_{k-1}\leq\mathcal{R}_{k-2}$ (by Proposition \ref{pro-basic-property}(i)\&(ii)) and $p_k\in[p_{\min},1]$ yields that
\begin{equation*}
\begin{aligned}
\Delta_\mathcal{R}^k
&= p_kH_{\delta}(\bm{x}^{k}, \bm{x}^{k-1}, \overline{\gamma}_{k-1})+(1-p_k)\mathcal{R}_{k-1}-\zeta 
\leq p_kH_{\delta}(\bm{x}^{k}, \bm{x}^{k-1}, \overline{\gamma}_{k-1})+(1-p_k)\mathcal{R}_{k-2}-\zeta\\
&= p_k\Delta_{H_{\delta}}^k+(1-p_k)\Delta_\mathcal{R}^{k-2}
\leq {\textstyle\left(\frac{a_1}{1+a_1}p_k+1-p_k\right)}\Delta_\mathcal{R}^{k-2}
= {\textstyle\left(1-\frac{p_k}{1+a_1}\right)}\Delta_\mathcal{R}^{k-2} 
\leq {\textstyle\left(1-\frac{p_{\min}}{1+a_1}\right)}\Delta_\mathcal{R}^{k-2}.
\end{aligned}
\end{equation*}

Combining with the above two cases, we can conclude that $\Delta_{\mathcal{R}}^{k}\leq\rho_3\Delta_{\mathcal{R}}^{k-2}$, where $\rho_3:=\max\left\{\rho_1,1-\frac{p_{\min}}{1+a_1}\right\}\in(0,1)$. This shows that statement (2b) holds.

Statement (2c). Suppose that $\theta\in(\frac{1}{2},1)$, We consider the following two cases.

We first consider $H_{\delta}(\bm{x}^{k}, \bm{x}^{k-1}, \overline{\gamma}_{k-1})\leq\zeta$. In this case, for any $k\geq\overline{K}_1$, it follows from statement (1a) in \textit{\textbf{Step 1}} that there exists $\rho_1\in(0,1)$ such that $\Delta_{\mathcal{R}}^{k}\leq\rho_1\Delta_{\mathcal{R}}^{k-2}$. Since $1-2\theta<0$ and $\Delta_{\mathcal{R}}^{k-2},\,\Delta_{\mathcal{R}}^{k}>0$, we further have that
\begin{equation*}
\big(\Delta_{\mathcal{R}}^{k}\big)^{1-2\theta}
\geq \rho_1^{1-2\theta}\big(\Delta_{\mathcal{R}}^{k-2}\big)^{1-2\theta},
\end{equation*}
which implies that
\begin{equation*}
\textstyle\big(\Delta_{\mathcal{R}}^{k}\big)^{1-2\theta}
- \big(\Delta_{\mathcal{R}}^{k-2}\big)^{1-2\theta}
\geq \big({\rho_1^{1-2\theta}}-1\big)\big(\Delta_{\mathcal{R}}^{k-2}\big)^{1-2\theta}
\geq\big({\rho_1^{1-2\theta}}-1\big)\big(\Delta_{\mathcal{R}}^{\overline{K}_1-2}\big)^{1-2\theta}
>0,
\end{equation*}
where the second inequality follows from the facts that $\{\Delta_{\mathcal{R}}^k\}$  is non-increasing, $\rho_1\in(0,1)$, $1-2\theta<0$, and ${\rho_1^{1-2\theta}}-1>0$. This gives the desired result.

We next consider $H_{\delta}(\bm{x}^{k}, \bm{x}^{k-1}, \overline{\gamma}_{k-1})>\zeta$. In this case, $\Delta_{H_{\delta}}^k>0$ and it follows from statement (1b) in \textit{\textbf{Step 1}} that
\begin{equation}\label{frac-1-a1-leq}
a_1^{-1}\leq\big(\Delta_{H_{\delta}}^k\big)^{-2\theta} \big(\Delta_{\mathcal{R}}^{k-2}-\Delta_{\mathcal{R}}^{k}\big).
\end{equation}
Next, we define $g(s):=s^{-2 \theta}$ for $s \in(0, \infty)$. It is easy to see that $g$ is non-increasing. Then, for any $k \geq \overline{K}_1$, we further consider the following two cases.
\begin{itemize}
\item If $g(\Delta_{H_\delta}^k) \leq 2 g(\Delta_{\mathcal{R}}^{k-2})$, it follows from \eqref{frac-1-a1-leq} that
    \begin{equation*}
    \begin{aligned}
    \textstyle\frac{1}{a_1}
    &\leq g(\Delta_{H_\delta}^k)\big(\Delta_{\mathcal{R}}^{k-2}
    -\Delta_{\mathcal{R}}^{k}\big)
    \leq 2g(\Delta_{\mathcal{R}}^{k-2}) \big(\Delta_{\mathcal{R}}^{k-2}-\Delta_{\mathcal{R}}^{k}\big) \\
    &\leq \textstyle 2 \int_{\Delta_{\mathcal{R}}^{k}}^{\Delta_{\mathcal{R}}^{k-2}} g(s)\,\mathrm{d}s
    = \frac{2(\Delta_{\mathcal{R}}^{k-2})^{1-2 \theta}-2(\Delta_{\mathcal{R}}^{k})^{1-2 \theta}}{1-2 \theta},
    \end{aligned}
    \end{equation*}
    which, together with $1-2\theta<0$, implies that
    \begin{equation*}
    (\Delta_{\mathcal{R}}^{k})^{1-2\theta}
    -(\Delta_{\mathcal{R}}^{k-2})^{1-2\theta}
    \geq (2\theta-1)/(2a_1).
    \end{equation*}

\item If $g(\Delta_{H_\delta}^k) > 2 g(\Delta_{\mathcal{R}}^{k-2})$, it follows that $\Delta_{H_\delta}^k<2^{-\frac{1}{2\theta}}\Delta_{\mathcal{R}}^{k-2}$.
    This, along with the updating rule of $\mathcal{R}_k$, $H_{\delta}(\bm{x}^{k}, \bm{x}^{k-1}, \overline{\gamma}_{k-1})\leq\mathcal{R}_k\leq\mathcal{R}_{k-1}\leq\mathcal{R}_{k-2}$ (by Proposition \ref{pro-basic-property}(i)\&(ii)) and $p_k\in[p_{\min},1]$, yields that
    \begin{equation*}
    \begin{aligned}
    \textstyle\Delta_\mathcal{R}^k
    &= p_kH_{\delta}(\bm{x}^{k}, \bm{x}^{k-1}, \overline{\gamma}_{k-1})+(1-p_k)\mathcal{R}_{k-1}-\zeta\\
    &\leq p_k H_{\delta}(\bm{x}^{k}, \bm{x}^{k-1}, \overline{\gamma}_{k-1})+(1-p_k)\mathcal{R}_{k-2}-\zeta\\
    &= p_k\Delta_{H_{\delta}}^k+(1-p_k)\Delta_\mathcal{R}^{k-2}\leq \big[1-\big(1-2^{-\frac{1}{2\theta}}\big)p_k\big]
    \Delta_\mathcal{R}^{k-2}
    \leq d_3\Delta_\mathcal{R}^{k-2},
    \end{aligned}
    \end{equation*}
    where $d_3:=1-\big(1-2^{-\frac{1}{2\theta}}\big)p_{\min}\in(0,1)$. This, together with $\Delta_{\mathcal{R}}^{k-2}>0$ and $\Delta_{\mathcal{R}}^{k}>0$, implies that
    \begin{equation*}
    (\Delta_{\mathcal{R}}^{k})^{1-2\theta}-(\Delta_{\mathcal{R}}^{k-2})^{1-2\theta}
    \geq({d_3^{1-2\theta}}-1)(\Delta_{\mathcal{R}}^{k-2})^{1-2\theta}
    \geq({d_3^{1-2\theta}}-1)(\Delta_{\mathcal{R}}^{\overline{K}_1-2})^{1-2\theta}>0,
    \end{equation*}
    where the second inequality follows from the facts that $\{\Delta_{\mathcal{R}}^k\}$ is non-increasing, $d_3\in(0,1)$ and $1-2\theta<0$.

\end{itemize}

In view of the above, we have that
\begin{equation*}
\textstyle(\Delta_{\mathcal{R}}^{k})^{1-2\theta}-(\Delta_{\mathcal{R}}^{k-2})^{1-2\theta}
\geq
a_2:=\min\left\{({\rho_1^{1-2\theta}}-1)(\Delta_{\mathcal{R}}^{\overline{K}_1-2})^{1-2\theta},
\,({d_3^{1-2\theta}}-1)(\Delta_{\mathcal{R}}^{\overline{K}_1-2})^{1-2\theta},
\,\frac{2\theta-1}{2a_1}\right\}.
\end{equation*}
This shows that statement (2c) holds.

\vspace{1mm}
\textit{\textbf{Step 3}}. We are now ready to prove our final results.

\textit{Statement (i)}. Suppose that $\theta=0$. For any $k\geq\overline{K}_1+1$, combining statement (2a) in \textit{\textbf{Step 2}}, \eqref{Fxk-Rk-1} and the fact that $\{\Delta_{\mathcal{R}}^k\}$ is non-increasing, we have that $F(\bm{x}^k)\leq\zeta$ and
\begin{equation}
\begin{aligned}\label{thetaineq1}
\textstyle\max\left\{|F(\bm{x}^k)-\zeta|,\mathcal{R}_k-\zeta\right\}
&\leq \max\{d_2,1\}\Delta_\mathcal{R}^{k-1}
\leq \max\{d_2,1\}\rho_2^{\lfloor\frac{k-\overline{K}_1+1}{2}\rfloor} \Delta_\mathcal{R}^{\overline{K}_1-2} \\
&\leq \max\{d_2,1\}\rho_2^{\frac{k-\overline{K}_1}{2}} \Delta_\mathcal{R}^{\overline{K}_1-2}
=c_1\eta_1^k,
\end{aligned}
\end{equation}
where $\lfloor a \rfloor$ denotes the largest integer smaller than or equal to $a$, $c_1:=\max\{d_2,1\}\rho_2^{-\frac{\overline{K}_1}{2}} \Delta_\mathcal{R}^{\overline{K}_1-2}$, $\eta_1:=\sqrt{\rho_2}\in(0,1)$, and the last inequality holds because $\lfloor\frac{k-\overline{K}_1+1}{2}\rfloor\geq\frac{k-\overline{K}_1}{2}$ and $\rho_2\in(0,1)$. This also implies that $\zeta-c_1\eta_1^k\leq F(\bm{x}^k)\leq \zeta$, and thus proves statement (i).

\textit{Statement (ii)}. Suppose that $\theta\in(0,\frac{1}{2}]$. Using similar arguments as in the above case, we can obtain that
\begin{equation}\label{thetaineq2}
    \max\left\{|F(\bm{x}^k)-\zeta|,\mathcal{R}_k-\zeta\right\}\leq c_2\eta_2^k
\end{equation}
holds for any $k\geq\overline{K}_1+1$, where $c_2>0$ and $\eta_2\in(0,1)$.

\textit{Statement (iii)}. Suppose that $\theta\in(\frac{1}{2},1)$. Let $\pi_k=(k-\overline{K}_1) \bmod 2$ for any $k \geq \overline{K}_1$. Then, by the nonnegativity of $\{\Delta_{\mathcal{R}}^k\}_{k=0}^{\infty}$ and $1-2\theta<0$, we have
\begin{equation*}
\begin{aligned}
\textstyle(\Delta_{\mathcal{R}}^k)^{1-2 \theta} & \geq(\Delta_{\mathcal{R}}^k)^{1-2 \theta}-(\Delta_{\mathcal{R}}^{\overline{K}_1+\pi_k})^{1-2 \theta}=\textstyle\sum_{j=1}^{(k-\overline{K}_1-\pi_k)/2}((\Delta_{\mathcal{R}}^{\overline{K}_1+\pi_k+2j})^{1-2 \theta}-(\Delta_{\mathcal{R}}^{\overline{K}_1+\pi_k-2+2j})^{1-2 \theta}) \\
& \textstyle\geq \frac{(k-\overline{K}_1-\pi_k) a_2}{2} \geq \frac{a_2}{4} k,
\end{aligned}
\end{equation*}
where the last inequality holds whenever $k \geq 2(\overline{K}_1+1) \geq 2(\overline{K}_1+\pi_k)$. Finally, using this relation, \eqref{Fxk-Rk-1} and the fact that $\{\Delta_{\mathcal{R}}^k\}$ is non-increasing, we see that, for all $k \geq 2(\overline{K}_1+1)+1$,
\begin{equation}\label{thetaineq3}
\begin{aligned}
\textstyle\quad~\max\left\{|F(\bm{x}^k)-\zeta|,\mathcal{R}_k-\zeta\right\}&\leq \max\{d_2,1\} \Delta_{\mathcal{R}}^{k-1}
\leq \max\{d_2,1\}(4/a_2)^{\frac{1}{2\theta-1}}(k-1)^{-\frac{1}{2\theta-1}}
\\
&\leq \max\{d_2,1\}(4/a_2)^{\frac{1}{2\theta-1}}[k/(k-1)]^{\frac{1}{2\theta-1}}\cdot k^{-\frac{1}{2 \theta-1}}\\
&\leq \max\{d_2,1\}(8/a_2)^{\frac{1}{2\theta-1}}k^{-\frac{1}{2\theta-1}}=c_3k^{-\frac{1}{2 \theta-1}},
\end{aligned}
\end{equation}
where $c_3:=\max\{d_2,1\}(8/a_2)^{\frac{1}{2\theta-1}}$, and the last inequality follows from $\frac{k}{k-1}\leq2$ and $\frac{1}{2\theta-1}\geq0$. This proves statement (iii).
\end{proof}

\begin{theorem}\label{theorem-seq-rate}
Under the same assumptions as in Theorem \ref{theorem-fun-rate},
the following statements hold for all sufficiently large $k$.  \vspace{1mm}
\begin{itemize}
\item[{\rm(i)}] If $\theta \in[0,\frac{1}{2}]$, there exist $d_1>0$ and $\varrho \in(0,1)$ such that $\|\bm{x}^k-\bm{x}^*\|\leq d_1\varrho^k$.
\vspace{1mm}
\item[{\rm(ii)}] If $\theta \in(\frac{1}{2}, 1)$, there exists $d_2>0$ such that $\|\bm{x}^k-\bm{x}^*\|\leq d_2k^{-\frac{1-\theta}{2\theta-1}}$.
\end{itemize}
\end{theorem}
\begin{proof}
First, recall the notations used in the previous analysis that
\begin{equation*}
\begin{aligned}
\textstyle M&:=\textstyle{\left\lceil\frac{2\big(1+\sqrt{1-p_{\min}}\big)}{1-\sqrt{1-p_{\min}}}\right\rceil^2}, \quad \ell(k):=k+M-1, \quad \Xi_k:=\sqrt{\mathcal{R}_k-\mathcal{R}_{k+1}},\\
\Delta_{i,j}^{\varphi}&:=\varphi(\mathcal{R}_i-\zeta)-\varphi(\mathcal{R}_j-\zeta),
\quad \pi:=\textstyle{\sqrt{\frac{(1-\delta)p_{\min}\gamma_{\min}}{8}}},
\quad \Delta^k_{\mathcal{R}}:=\mathcal{R}_k-\zeta, 
\end{aligned}
\end{equation*}
where $\lceil a \rceil$ denotes the least integer greater than or equal to $a$. Since $\bm{x}^*$ is an accumulation point of $\{\bm{x}^k\}$, it follows from Theorem \ref{nexPGA-theorem-wholesequence} and the corresponding proof that the whole sequence $\{\bm{x}^k\}_{k=0}^{\infty}$ converges to $\bm{x}^*$,
\begin{equation}\label{Xik-xkdiff-App}
\textstyle\|\bm{x}^{k+1}-\bm{x}^k\|\leq\frac{\Xi_k}{\pi},
\end{equation}
\begin{equation}\label{M-relation-App}
\big(1-\sqrt{1-p_{\min}}\big)\sqrt{M}
- \big(1+2\sqrt{1-p_{\min}}\big)
\geq 1,
\end{equation}
and there exists an integer $K_1$ such that
\begin{equation}\label{ineq-sum-Ei-App}
\textstyle\frac{1-\sqrt{1-p_{\min }}}{\sqrt{M}}\sum_{i=k}^{\ell(k)}\Xi_i
\leq \left(\frac{1}{2}+\sqrt{1-p_{\min}}\right)
\big(\Xi_{k-2}+\Xi_{k-1}\big)
+ \frac{\tilde{c}}{2\pi}\Delta^{\varphi}_{k,k+M},
\quad \forall\,k\geq K_1.
\end{equation}
Moreover, since $\{\mathcal{R}_k\}_{k=0}^{\infty}$ is non-increasing and converges to $\zeta$ (by Proposition \ref{pro-basic-property}(ii)) and \eqref{thetaineq1}, \eqref{thetaineq2}, \eqref{thetaineq3} hold for all sufficiently large $k$, there exists an integer $K_2$ such that $\zeta\leq\mathcal{R}_k\leq\zeta+1$ (i.e., $0\leq\Delta_\mathcal{R}^k\leq1$) and \eqref{thetaineq1}, \eqref{thetaineq2}, \eqref{thetaineq3} hold whenever $k\geq K_2$.

Next, we claim that there exist $t_1>0$ and $t_2>0$ such that the following inequality holds: 
\begin{equation}\label{sum-infty-Ei}
\textstyle\|\bm{x}^k-\bm{x}^*\|\leq t_1(\Delta_\mathcal{R}^{k-2})^{\frac{1}{2}}
+ t_2 (\Delta_\mathcal{R}^{k-2})^{1-\theta},
\quad \forall\,k\geq K_1.
\end{equation}
Indeed, for any $k \geq K_1$, we see that
\begin{equation}\label{ineqadd-tildek}
\begin{aligned}
&\quad \textstyle\big(1-\sqrt{1-p_{\min}})\sqrt{M} \sum_{i=\ell(k)}^{\tilde{k}}\Xi_i
\leq \frac{1-\sqrt{1-p_{\min}}}{\sqrt{M}}
\sum_{i=k}^{\tilde{k}}\sum_{t=i}^{\ell(i)}\Xi_t  \\
& \textstyle\leq \left(\frac{1}{2}+\sqrt{1-p_{\min}}\right) \sum_{i=k}^{\tilde{k}}\big(\Xi_{i-2}+\Xi_{i-1}\big)
+ \frac{\tilde{c}}{2\pi}\sum_{i=k}^{\tilde{k}}
\Delta^{\varphi}_{i,i+M}  \\
&\leq \textstyle\big(1+2\sqrt{1-p_{\min}}\big) \sum_{i=k-2}^{\tilde{k}}\Xi_i
+ \frac{\tilde{c}\tilde{a}}{2\pi}\sum_{i=k}^{\ell(k)} \big(\mathcal{R}_i-\zeta\big)^{1-\theta}\\
&=\textstyle\big(1+2\sqrt{1-p_{\min}}\big) \left(\textstyle\sum_{i=k-2}^{\ell(k)-1}\Xi_i+\textstyle\sum_{i=\ell(k)}^{\tilde{k}}\Xi_i\right)
+ \frac{\tilde{c}\tilde{a}}{2\pi}\sum_{i=k}^{\ell(k)} \big(\mathcal{R}_i-\zeta\big)^{1-\theta}
\end{aligned}
\end{equation}
holds for all $\tilde{k}\geq\ell(k)$,
where the first inequality follows from the nonnegativity of $\Xi_k$ and the fact that the term $\Xi_i$ with $i\in\{\ell(k), \cdots, \tilde{k}\}$ appears $M$ times in the double sum; the second inequality is obtained by applying \eqref{ineq-sum-Ei-App} with $\varphi(s)=\tilde{a}s^{1-\theta}$ to each $k\in\{k,k+1,\cdots,\tilde{k}\}$; the third inequality follows because
\begin{equation*}
\begin{aligned}
\textstyle\sum_{i=k}^{\tilde{k}}\Delta^{\varphi}_{i,i+M}
&=\textstyle\sum_{i=k}^{\tilde{k}}\varphi(\mathcal{R}_i-\zeta)-\sum_{i=k}^{\tilde{k}}\varphi(\mathcal{R}_{i+M}-\zeta)
=\sum_{i=k}^{\tilde{k}}\varphi(\mathcal{R}_i-\zeta)-\sum_{i=k+M}^{\tilde{k}+M}\varphi(\mathcal{R}_i-\zeta) \\
&=\textstyle\sum_{i=k}^{{k}+M-1}\varphi(\mathcal{R}_i-\zeta)-\sum_{i=\tilde{k}+1}^{\tilde{k}+M}\varphi(\mathcal{R}_i-\zeta)
\leq\sum_{i=k}^{\ell({k})}\varphi(\mathcal{R}_i-\zeta)=\sum_{i=k}^{\ell({k})}\tilde{a}(\mathcal{R}_i-\zeta)^{1-\theta}.
\end{aligned}
\end{equation*}
Using \eqref{ineqadd-tildek}, together with \eqref{M-relation-App}, we obtain that
\begin{equation*}
\textstyle\sum_{i=\ell(k)}^{\tilde{k}}\Xi_i
\leq \big(1+2\sqrt{1-p_{\min}}\big)\sum_{i=k-2}^{\ell(k)-1}\Xi_i
+ \frac{\tilde{c}\tilde{a}}{2\pi}\sum_{i=k}^{\ell({k})}(\mathcal{R}_i-\zeta)^{1-\theta},
\end{equation*}
which further implies that
\begin{equation}\label{Ei-k-tildek}
\textstyle\sum_{i=k}^{\tilde{k}}\Xi_i= \sum_{i=k}^{\ell(k)-1}\Xi_i
+ \sum_{i=\ell(k)}^{\tilde{k}}\Xi_i
\leq \big(2+2\sqrt{1-p_{\min}}\big) \sum_{i=k-2}^{\ell(k)-1}\Xi_i
+ \frac{\tilde{c}\tilde{a}}{2\pi}\sum_{i=k}^{\ell(k)} \big(\mathcal{R}_i-\zeta\big)^{1-\theta}.
\end{equation}
Moreover, it follows from  $\ell(k)=k+M-1$ and the fact that $\{\mathcal{R}_k\}$ converges monotonically to $\zeta$ (by Proposition \ref{pro-basic-property}(ii)) that
\begin{equation*}
\textstyle\sum_{i=k-2}^{\ell(k)-1}\Xi_i=\sum_{i=k-2}^{\ell(k)-1}\sqrt{\mathcal{R}_i-\mathcal{R}_{i+1}}\leq\sum_{i=k-2}^{\ell(k)-1}\sqrt{\mathcal{R}_i-\zeta}\leq (M+1) \sqrt{\mathcal{R}_{k-2}-\zeta}
=(M+1) (\Delta_\mathcal{R}^{k-2})^{\frac{1}{2}}
\end{equation*}
and
\begin{equation*}
\textstyle\sum_{i=k}^{\ell({k})}(\mathcal{R}_i-\zeta)^{1-\theta}\leq M(\mathcal{R}_k-\zeta)^{1-\theta}\leq M(\mathcal{R}_{k-2}-\zeta)^{1-\theta}= M(\Delta_{\mathcal{R}}^{k-2})^{1-\theta}.
\end{equation*}
Using the two above relations and passing to the limit $\tilde{k}\to\infty$ in \eqref{Ei-k-tildek}, we further have
\begin{equation*}
\begin{aligned}
\textstyle\sum_{i=k}^{\infty}\Xi_i
&\leq \big(2+2\sqrt{1-p_{\min}}\big)\textstyle\sum_{i=k-2}^{\ell(k)-1}\Xi_i
+ \frac{\tilde{c}\tilde{a}}{2\pi}\sum_{i=k}^{\ell(k)} \big(\mathcal{R}_i-\zeta\big)^{1-\theta}\\
&\textstyle\leq\big(2+2\sqrt{1-p_{\min}}\big)(M+1)(\Delta_\mathcal{R}^{k-2})^{\frac{1}{2}}
+ \frac{\tilde{c}\tilde{a}M}{2\pi} (\Delta_\mathcal{R}^{k-2})^{1-\theta}.
\end{aligned}
\end{equation*}
This, together with the triangle inequality and \eqref{Xik-xkdiff-App}, yields
\begin{equation*}
\begin{aligned}
\|\bm{x}^k-\bm{x}^*\|
&\leq\sum\limits_{i=k}^{\infty}\|\bm{x}^{i+1}-\bm{x}^i\|\leq\textstyle\frac{1}{\pi}\sum_{i=k}^{\infty}\Xi_i \leq\frac{\big(2+2\sqrt{1-p_{\min}}\big)(M+1)}{\pi}(\Delta_\mathcal{R}^{k-2})^{\frac{1}{2}}
+ \frac{\tilde{c}\tilde{a}M}{2\pi^2} (\Delta_\mathcal{R}^{k-2})^{1-\theta}.
\end{aligned}
\end{equation*}
This proves \eqref{sum-infty-Ei} with $t_1:=\pi^{-1}\big(2+2\sqrt{1-p_{\min}}\big)(M+1)$ and $t_2:=\frac{\tilde{c}\tilde{a}M}{2\pi^2}$.

With the above inequality in hand, we consider the following two cases.

\textit{Case (i).} $\theta\in[0,\frac{1}{2}]$. In this case, for any $k\geq\max\{K_1,K_2\}+2$, we have that $1-\theta\geq\frac{1}{2}$, $\Delta_{\mathcal{R}}^{k-2}\leq1$, \eqref{thetaineq1}, \eqref{thetaineq2} and \eqref{sum-infty-Ei} hold. Then, we further have
\begin{equation*}
\begin{aligned}
\|\bm{x}^k-\bm{x}^*\|&\leq t_1(\Delta_\mathcal{R}^{k-2})^{\frac{1}{2}}
+ t_2 (\Delta_\mathcal{R}^{k-2})^{1-\theta}\leq (t_1+t_2)(\Delta_\mathcal{R}^{k-2})^{\frac{1}{2}}.
\end{aligned}
\end{equation*}
Combining this with \eqref{thetaineq1} and \eqref{thetaineq2}, we obtain
\begin{equation*}
\|\bm{x}^k-\bm{x}^*\|
\leq(t_1+t_2)\max\{\sqrt{c_1},\sqrt{c_2}\}{\max\{\sqrt{\eta_1},\sqrt{\eta_2}\}}^{k-2}=d_1\varrho^k,
\end{equation*}
where $d_1:=(t_1+t_2)\max\{\sqrt{c_1},\sqrt{c_2}\}{\max\{\sqrt{\eta_1},\sqrt{\eta_2}\}}^{-2}>0$ and $\varrho:=\max\{\sqrt{\eta_1},\sqrt{\eta_2}\}\in(0,1)$.
This proves statement (i).

\textit{Case (ii).} $\theta\in(\frac{1}{2},1)$. In this case, for any $k\geq\max\{K_1,K_2\}+3$, we have that $1-\theta<\frac{1}{2}$, $\Delta_{\mathcal{R}}^{k-2}\leq1$, \eqref{thetaineq3} and \eqref{sum-infty-Ei} hold. Then we further have
\begin{equation*}
\|\bm{x}^k-\bm{x}^*\|\leq t_1(\Delta_\mathcal{R}^{k-2})^{\frac{1}{2}}
+ t_2 (\Delta_\mathcal{R}^{k-2})^{1-\theta}\leq (t_1+t_2)(\Delta_\mathcal{R}^{k-2})^{1-\theta}.
\end{equation*}
Combining this with \eqref{thetaineq3}, we obtain
\begin{equation*}
\begin{aligned}
\textstyle\|\bm{x}^k-\bm{x}^*\|
&\leq(t_1+t_2)c_3^{1-\theta}{(k-2)}^{-\frac{1-\theta}{2\theta-1}}\leq(t_1+t_2)c_3^{1-\theta}{[k/(k-2)]}^{\frac{1-\theta}{2\theta-1}}\cdot{k}^{-\frac{1-\theta}{2\theta-1}}\\
&\leq 3^{\frac{1-\theta}{2\theta-1}}(t_1+t_2)c_3^{1-\theta}{k}^{-\frac{1-\theta}{2\theta-1}}= d_2{k}^{-\frac{1-\theta}{2\theta-1}},
\end{aligned}
\end{equation*}
where $d_2:=3^{\frac{1-\theta}{2\theta-1}}(t_1+t_2)c_3^{1-\theta}>0$, and the last inequality follows from $\frac{k}{k-2}\leq3$ and $\frac{1-\theta}{2\theta-1}>0$. This proves statement (ii).
\end{proof}

In summary, the global sequential convergence and local convergence rates of both the generated sequence and objective function values are established in this section \textit{without} requiring the global Lipschitz continuity of $\nabla f$ or any boundedness-type assumptions on the generated sequence. As discussed in Section~\ref{sec-algo}, our nexPGA framework encompasses and complements numerous existing methods and can potentially give rise to new potentially accelerated variants. Consequently, our analysis provides a unified convergence theory that can be applied to many well-known algorithms and their improved extensions, including those studied in \cite{gtt2018DC,jkm2023convergence,kl2025convergence,km2022convergence,wcp2018proximal,yang2024proximal}.

%%%%%%%%%%%%%%%%%%%%%%%%%%%%%%%%%%%%%%%%%%%%%%%%%
\section{Numerical experiments}\label{sec-num-exp}

In this section, we conduct some preliminary numerical experiments to evaluate the performance of our nexPGA in Algorithm \ref{algo-nexPGA} for solving the $\ell_{1\text{-}2}$ regularized least squares problem. All experiments are run in {\sc Matlab} R2023a on a PC with Intel processor i7-12700K@3.60GHz (with 12 cores and 20 threads) and 64GB of RAM, equipped with a Windows OS.

The $\ell_{1\text{-}2}$ regularized least squares problem has received significant attention in recent years; see, e.g., \cite{lyhx2015computing,wcp2018proximal,ylhx2015minimization}. This problem is mathematically formulated as

\begin{equation}\label{probleml12}
\min\limits_{\bm{x}\in\mathbb{R}^n}~ F(\bm{x}):=\frac{1}{2}\|A\bm{x}-\bm{b}\|^2 + \lambda\big( \|\bm{x}\|_1 - \|\bm{x}\| \big), 
\end{equation}
where $A\in\mathbb{R}^{m\times n}$, $\bm{b}\in\mathbb{R}^m$, and $\lambda>0$ is a regularization parameter. Problem \eqref{probleml12} can be reformulated in the form of \eqref{ge-DC-problem} in two different ways: \vspace{1mm}
\begin{itemize}[leftmargin=0.4cm]
\item[] \textbf{Decomposition I}: $f(\bm{x})=\frac{1}{2}\|A\bm{x}-\bm{b}\|^2$, $P_1(\bm{x})=\lambda(\|\bm{x}\|_1-\|\bm{x}\|)$, and $P_2(\bm{x})=0$;
\vspace{1mm}
\item[] \textbf{Decomposition II}: $f(\bm{x})=\frac{1}{2}\|A\bm{x}-\bm{b}\|^2$, $P_1(\bm{x})=\lambda\|\bm{x}\|_1$, and $P_2(\bm{x})=\lambda\|\bm{x}\|$.
\end{itemize}
\vspace{1mm}
It is easy to verify that Assumptions \ref{assum-funs1} and \ref{assumC} are satisfied under \textbf{Decomposition I}, and also under \textbf{Decomposition II} provided that $2\lambda<\|A^{\top}\bm{b}\|_{\infty}$ (see, e.g.,
\cite[Example 4.1]{wcp2018proximal}). Thus, nexPGA is applicable for solving \eqref{probleml12} using either decomposition.

In our experiments, we will evaluate nexPGA with $\delta=0.1$ for \textbf{Decomposition I} (denoted by nexPGA), nexPGA with $\delta=0.1$ for \textbf{Decomposition II} (denoted by nexPGA-DC), and nexPGA with $\delta=0$ for \textbf{Decomposition I} (denoted by NPG). For all three methods, we follow Assumption \ref{assum-para} to set parameters as follows: $\tau=1.56$, $\eta=0.8$, $\beta_{\max}=10$, $\gamma_{\min }=10^{-6}$, $\gamma_{\max}=10^6$, and $p_k\equiv0.01$. Moreover, we set $\gamma_{0,0}=1$ and  
\begin{equation*}
{\textstyle\gamma_{k,0}=\min \left\{\max \left\{\max \left\{\frac{\left\langle\overline{\bm{y}}^k-\overline{\bm{y}}^{k-1}, \nabla f\left(\overline{\bm{y}}^k\right)-\nabla f\left(\overline{\bm{y}}^{k-1}\right)\right\rangle}{\|\overline{\bm{y}}^k-\overline{\bm{y}}^{k-1}\|^2},
\,0.9\overline{\gamma}_{k-1}\right\}, \,\gamma_{\min}\right\}, \,\gamma_{\max}\right\}.}  
\end{equation*}
In addition, for nexPGA(-DC), we choose the initial extrapolation parameters $\{\beta_{k,0}\}$ by $\beta_{k,0}=(t_{k-1}-1)/t_k$ with $t_{k+1}=\frac{1+\sqrt{1+4t_k^2}}{2}$ and $t_{-1}=t_0=1$.

We also include in our comparisons the proximal gradient method with extrapolation and line search (PGels) \cite{yang2024proximal} applied to \textbf{Decomposition I}, and the proximal difference-of-convex algorithm with extrapolation (pDCAe) \cite{wcp2018proximal} applied to \textbf{Decomposition II}. For both algorithms, we adopt the parameter settings recommended in their respective references. In addition, we initialize all algorithms at the origin and set the maximum running time to $\mathrm{T}^{\max}$ for all algorithms. The specific values of $\mathrm{T}^{\max}$ are specified in Figure \ref{FigET}.

In the following, we consider $\lambda \in \{0.1, \,0.01\}$ and $n\in\{3000,5000,10000\}$, with $m=0.1n$ and $s=0.2m$. For each triple $(m, n, s)$, we generate one random trial as follows. First, we generate a matrix $A \in \mathbb{R}^{m\times n}$ with i.i.d. (independent and identically distributed) standard Gaussian entries. Then, we uniformly at random choose a subset $\mathcal{S}$ of size $s$ from $\{1, \cdots, n\}$ and construct an $s$-sparse vector $\hat{\bm{x}}\in\mathbb{R}^{n}$, which has i.i.d. standard Gaussian entries on $\mathcal{S}$ and has zeros on $\mathcal{S}^c$. Finally, we set $\bm{b} = A\hat{\bm{x}}+0.01\cdot\hat{\bm{z}}$, where $\hat{\bm{z}}\in\mathbb{R}^{m}$ is a vector with i.i.d. standard Gaussian entries.

To evaluate the performances of different methods, we follow \cite{yang2024proximal,ypc2017nonmonotone} to use an evolution of objective values. To introduce this evolution, we first define $e(k):=(F(\bm{x}^k) - F^{\min})/({F(\bm{x}^0) - F^{\min}})$, where $F(\bm{x}^k)$ denotes the objective value at $\bm{x}^k$ obtained by a method and $F^{\min}$ denotes the minimum of the terminating objective values obtained among \textit{all} methods in a trial generated as above. For a method, let $\mathcal{T}(k)$ denote the \textit{total} computational time (from the beginning) when it obtains $\bm{x}^k$. Thus, $\mathcal{T}(0)=0$ and $\mathcal{T}(k)$ is non-decreasing with respect to $k$. We then define the evolution of objective values obtained by a particular method with respect to time $t\geq0$ as $E(t):=\min\left\{e(k):k\in\{i: \mathcal{T}(i) \leq t\}\right\}$. Note that $0 \leq E(t)\leq1$ (since $0 \leq e(k)\leq1$ for all $k$) and $E(t)$ is non-increasing with respect to $t$. It can be considered as a normalized measure of the reduction of the function value with respect to time. Then, one can take the average of $E(t)$ over several independent trials, and plot the average $E(t)$ within time $t$ for a given method.

Figure \ref{FigET} presents the average $E(t)$ over 10 independent trials for different methods applied to problem \eqref{probleml12}. Several observations can be drawn from the results: 
\begin{itemize}[leftmargin=0.7cm]
\item \textbf{(Line Search vs. No Line Search)} For \textbf{Decomposition II}, our nexPGA-DC always outperforms pDCAe by reducing the function value at a faster rate. This demonstrates the effectiveness of incorporating an average-type nonmonotone line search into pDCAe.

\item \textbf{(Extrapolation vs. No Extrapolation)} We also observe that nexPGA outperforms NPG. This highlights the importance of incorporating extrapolation steps into the line search framework to achieve possible acceleration.

\item \textbf{(ZH-type vs. GLL-type)} Lastly, with appropriately chosen parameters, our nexPGA exhibits comparable or even superior performance to PGels. This is indeed expected, as both methods share a similar algorithmic framework, but differ in their line search acceptance criteria. Notably, as discussed in Sections \ref{sec-algo} and \ref{sec-kl-analysis}, compared to PGels, our nexPGA is problem-parameter-free and has stronger theoretical guarantees under weaker assumptions.
\end{itemize}

\begin{figure}[ht]
\centering
\subfigure[$\lambda=0.1$]{
\includegraphics[width=4cm]{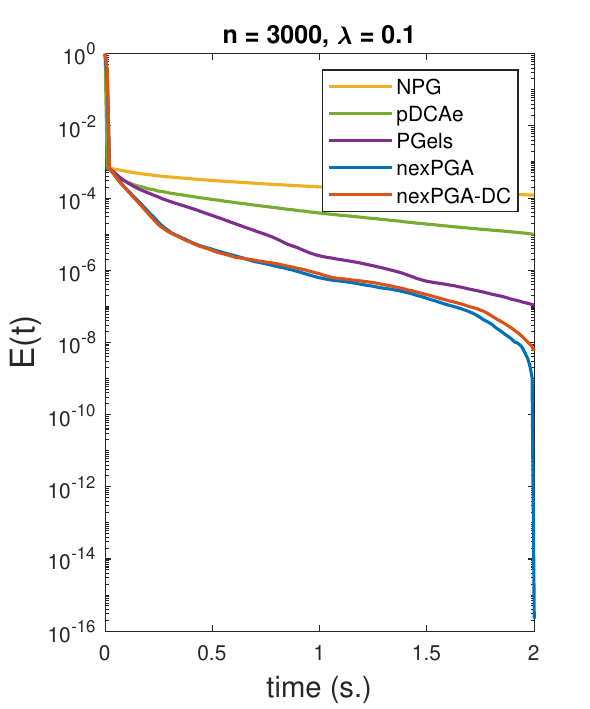}
\includegraphics[width=4cm]{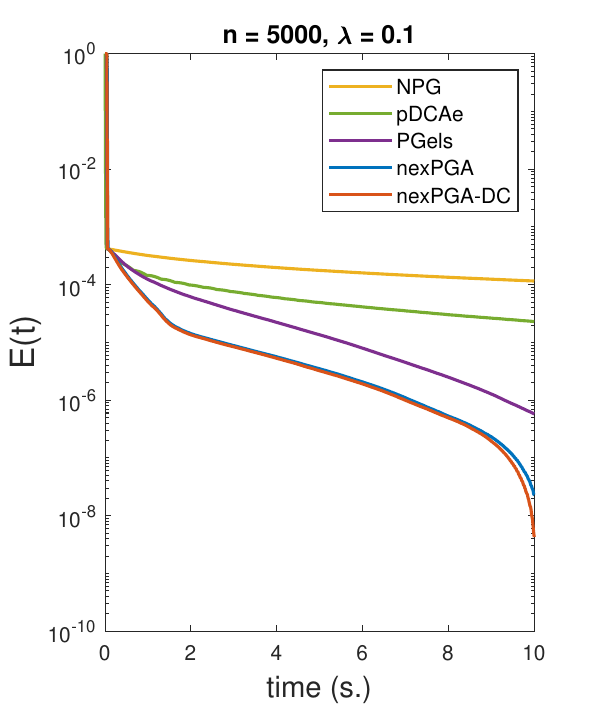}
\includegraphics[width=4cm]{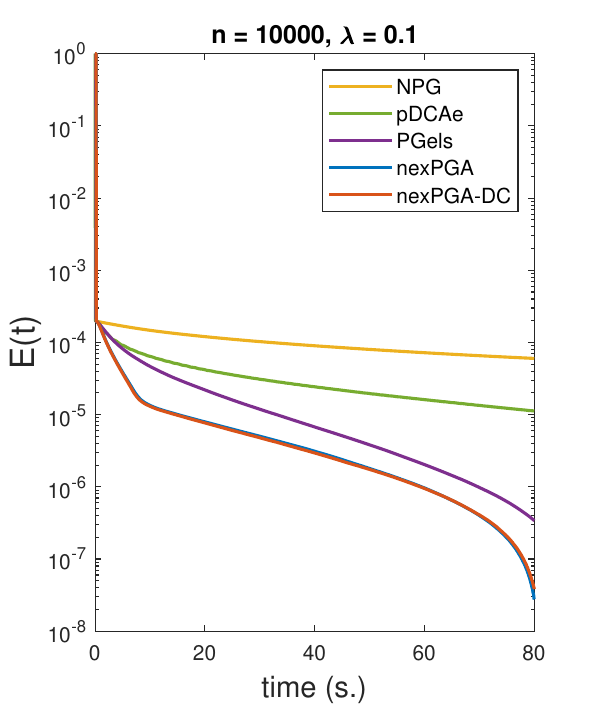}
}
\subfigure[$\lambda=0.01$]{
\includegraphics[width=4cm]{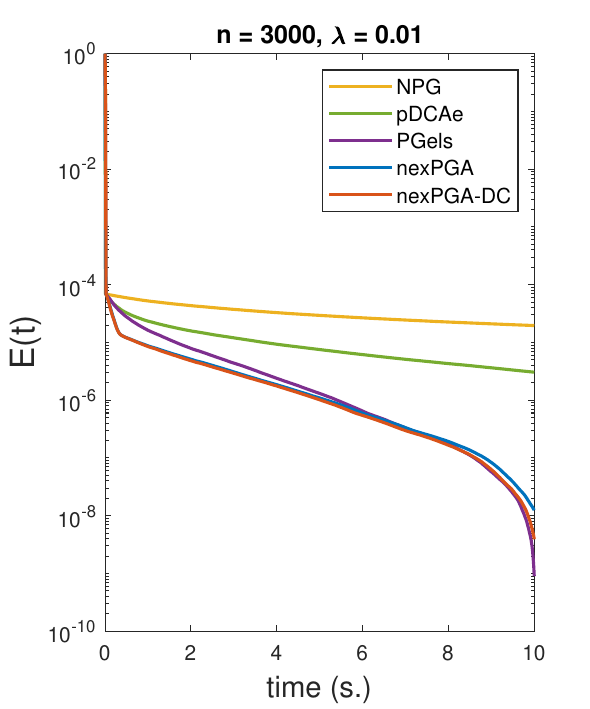}
\includegraphics[width=4cm]{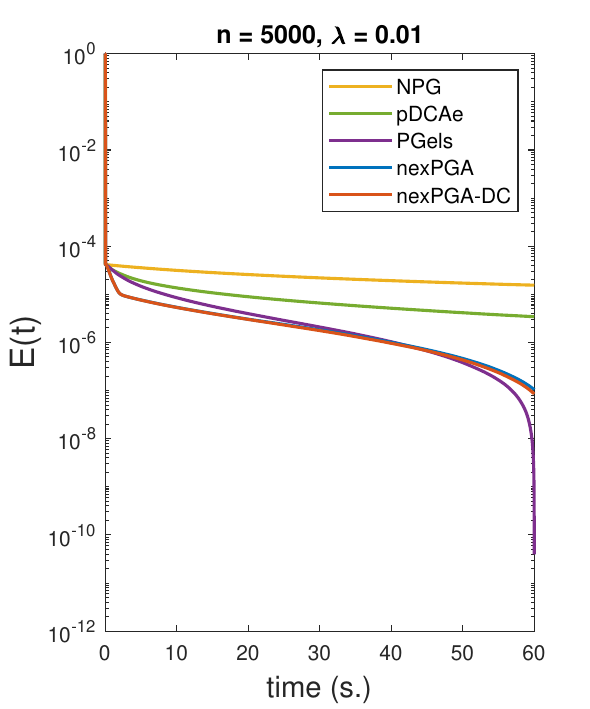}
\includegraphics[width=4cm]{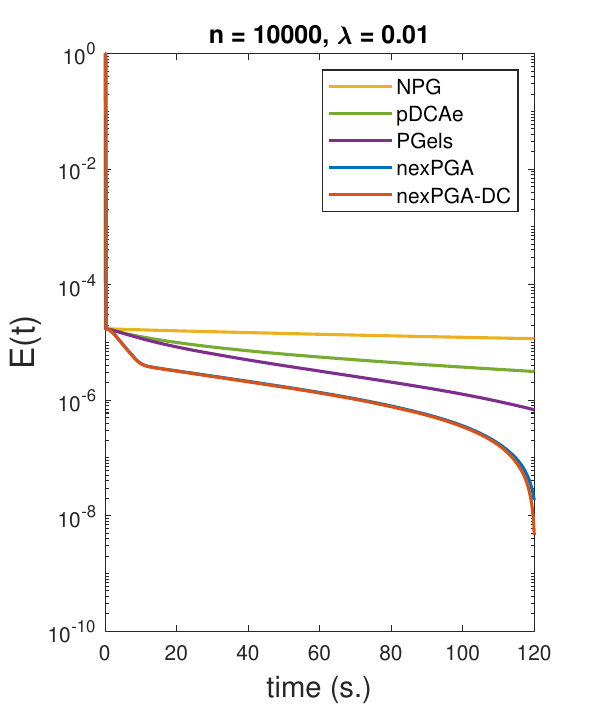}
}

\caption{Average $E(t)$ of 10 independent trials of different methods for solving \eqref{probleml12}.}\label{FigET}
\end{figure}

In summary, the simultaneous incorporation of the ZH-type nonmonotone line search and extrapolation in nexPGA is both necessary and impactful. These enhancements can improve the performance of the PG method and its variants, as evidenced by numerical results. Moreover, nexPGA operates without requiring problem-specific parameters or the global Lipschitz continuity of the gradient, yet it still has strong theoretical guarantees and achieves encouraging empirical performance. Finally, as our model \eqref{ge-DC-problem} encompasses a wide range of application problems, our nexPGA can offer broad applicability and has the potential to serve as a versatile algorithmic framework for solving large-scale nonconvex optimization problems.

%%%%%%%%%%%%%%%%%%%%%%%%%%%%%%%%%%%%%%%%%%
\section{Conclusion}\label{sec-conclusion}
In this paper, we study a general composite optimization model beyond the standard global Lipschitz gradient continuity setting. To solve this problem, we propose a novel algorithm that combines an extrapolation step with a carefully designed nonmonotone line search and establish complete global convergence results without imposing boundedness assumptions on the iterates. To the best of our knowledge, this is the first extrapolated proximal-type algorithm developed under such relaxed conditions. We believe that this work will advance the theoretical understanding of nonmonotone proximal-type methods in the absence of a global Lipschitz gradient assumption.

%%%%%%%%%%%%%%%%%%%%%%%%%%%%%%%%%%
\bibliographystyle{plain}
\bibliography{Ref_nexPGA}

\end{document}